\documentclass{article}

\usepackage[margin=1in]{geometry}

\title{Mixing Rates for Hamiltonian Monte Carlo Algorithms in Finite and Infinite Dimensions}

\author{Nathan E. Glatt-Holtz, Cecilia F. Mondaini \\
  \scriptsize{emails: negh@tulane.edu, cf823@drexel.edu} }

\usepackage{amsfonts, amssymb, amsmath, amsthm, multicol}
\usepackage[colorlinks=true, pdfstartview=FitV, linkcolor=blue,
            citecolor=blue, urlcolor=blue]{hyperref}
\usepackage[usenames]{color}
\definecolor{Red}{rgb}{0.7,0,0.1}
\definecolor{Green}{rgb}{0,0.7,0}
\usepackage{accents}
\usepackage{comment}
\usepackage{float}
\usepackage{graphicx}
\usepackage{multirow}
\usepackage{parcolumns}
\usepackage{subcaption}
\usepackage{textcomp}
\usepackage{mathrsfs}
\usepackage{mathtools}
\usepackage{dsfont}
\usepackage{enumerate}
\usepackage[capitalize,nameinlink,noabbrev]{cleveref}

\usepackage{cjhebrew}

\usepackage{pdfsync}
\numberwithin{equation}{section}

\newtheorem{Thm}{Theorem}[section]
\newtheorem{Lem}[Thm]{Lemma}
\newtheorem{Prop}[Thm]{Proposition}
\newtheorem{Cor}[Thm]{Corollary}

\newtheorem{Def}[Thm]{Definition}
\newtheorem{Rmk}[Thm]{Remark}

\newtheorem{assumption}[Thm]{Assumption}

\newtheorem*{Thm*}{Theorem}

\newcommand{\beq}{\begin{equation}}
\newcommand{\eeq}{\end{equation}}
\newcommand{\bea}{\begin{eqnarray}}
\newcommand{\eea}{\end{eqnarray}}
\newcommand{\beas}{\begin{eqnarray*}}
\newcommand{\eeas}{\end{eqnarray*}}

\newcommand{\pd}{\partial}

\newcommand{\indFn}[1]{1 \! \! 1_{#1}}
\newcommand{\Prb}{\mathbb{P}}
\newcommand{\E}{\mathbb{E}}
\newcommand{\Wass}{\mathcal{W}}
\newcommand{\Co}{\mathfrak{C}}

\newcommand{\RR}{\mathbb{R}}
\newcommand{\TT}{\mathbb{T}}
\newcommand{\NN}{\mathbb{N}}
\def\H{\mathbb{H}}

\newcommand{\Tr}{\operatorname{Tr}}

\newcommand{\trhoe}{\tilde{\rho}}
\newcommand{\rhoep}{\rho}

\newcommand{\trhob}{\tilde{\rho}_{\beta}}

\newcommand{\rt}{\gamma}

\newcommand{\bv}{\mathbf{v}}
\newcommand{\bu}{\mathbf{u}}
\newcommand{\bz}{\mathbf{z}}
\newcommand{\bw}{\mathbf{w}}

\newcommand{\bq}{\mathbf{q}}

\newcommand{\btq}{\tilde{\mathbf{q}}}

\newcommand{\bff}{\mathbf{f}}
\newcommand{\bfg}{\mathbf{g}}
\newcommand{\bP}{\mathbf{P}}
\newcommand{\bp}{\mathbf{p}}
\newcommand{\tq}{\tilde{q}}

\newcommand{\txi}{\tilde{\xi}}
\newcommand{\bbf}{\mathbf{f}}

\newcommand{\be}{\mathbf{e}}

\newcommand{\npV}{\mathbf{V}}

\newcommand{\bfe}{\mathbf{e}}

\newcommand{\Obs}{\mathcal{O}}
\newcommand{\data}{\mathcal{Y}}
\newcommand{\GP}{\Gamma^{-1/2}}
\newcommand{\FD}{\Lambda}

\newcommand{\bE}{\mathbb{E}}
\newcommand{\cC}{\mathcal{C}}
\newcommand{\mM}{\mathcal{M}}
\newcommand{\mB}{\mathcal{B}}
\newcommand{\mL}{\mathcal{L}}

\newcommand{\mN}{\mathcal{N}}
\newcommand{\cS}{\mathcal{S}}
\newcommand{\bcS}{\boldsymbol{\mathcal{S}}}
\newcommand{\bcR}{\boldsymbol{\mathcal{R}}}

\newcommand{\Uz}{L_0}
\newcommand{\Ua}{L_1}
\newcommand{\Ub}{L_2}
\newcommand{\Uc}{L_3}
\newcommand{\Ud}{L_4}
\newcommand{\Ue}{L_5}

\newcommand{\sma}{\kappa_1}
\newcommand{\smb}{\kappa_2}
\newcommand{\smc}{\kappa_3}
\newcommand{\smd}{\kappa_4}
\newcommand{\sme}{\kappa_5}
\newcommand{\smeb}{\sme(n_0)}

\newcommand{\bPsi}{\boldsymbol{\Psi}}
\newcommand{\tQ}{\widetilde{Q}}

\newcommand{\tsig}{\tilde{\sigma}}
\newcommand{\KL}{D_{\text{KL}}} 

\newcommand{\lM}{\lambda_{\mM}}
\newcommand{\lC}{\lambda_{\cC}}
\newcommand{\LM}{\Lambda_{\mM}}
\newcommand{\LC}{\Lambda_{\cC}}
\newcommand{\tP}{\widetilde{P}}

\newcommand{\norm}[1]{\left\lvert #1 \right\rvert}
\newcommand{\normga}[1]{\left\lvert #1 \right\rvert_{\gamma,\alpha}}
\newcommand{\nga}[1]{\left\lvert #1 \right\rvert_{\gamma}}
\newcommand{\ngad}[1]{\left\lvert #1 \right\rvert_{- \gamma}}
\newcommand{\tv}[1]{\left\| #1 \right\|_{\text{TV}}}

\newcommand{\dlf}{\ell}
\newcommand{\BPhi}{\bar{\Phi}}
\newcommand{\Hp}{\mathbb{V}}

\begin{document}
\markboth{}{}

\maketitle

\begin{abstract}
  We establish the geometric ergodicity of the preconditioned
  Hamiltonian Monte Carlo (HMC) algorithm defined on an
  infinite-dimensional Hilbert space, as developed in
  \cite{BePiSaSt2011}. This algorithm can be used as a basis to sample
  from certain classes of target measures which are absolutely
  continuous with respect to a Gaussian measure. Our work addresses an
  open question posed in \cite{BePiSaSt2011}, and provides an
  alternative to a recent proof based on exact coupling techniques
  given in \cite{BoEbPHMC}. The approach here establishes convergence
  in a suitable Wasserstein distance by using the weak Harris theorem
  together with a generalized coupling argument. We also show that a
  law of large numbers and central limit theorem can be derived as a
  consequence of our main convergence result. Moreover, our approach yields a novel proof of mixing rates for the classical finite-dimensional HMC algorithm. As such, the methodology
  we develop provides a flexible framework to tackle the rigorous
  convergence of other Markov Chain Monte Carlo
  algorithms. Additionally, we show that the scope of our result
  includes certain measures that arise in the Bayesian approach to
  inverse PDE problems, cf. \cite{stuart2010inverse}. Particularly, we
  verify all of the required assumptions for a certain class of
  inverse problems involving the recovery of a divergence free vector
  field from a passive scalar, \cite{borggaard2018bayesian}.
\end{abstract}

{\noindent \small
   }

\setcounter{tocdepth}{1}
\tableofcontents

\newpage

\section{Introduction}

It has long been appreciated that Markov chains can be employed as an
effective computational tool to sample from probability measures.
Starting from a desired `target' probability distribution $\mu$ on a
space $\H$ one seeks a Markov transition kernel $P$ for which $\mu$ is
an invariant and which moreover maintains desirable mixing properties
with respect to this $\mu$.  In particular in Bayesian statistics
\cite{kaipio2005statistical, marzouk2007stochastic,
  martin2012stochastic, stuart2010inverse, dashti2017bayesian,
  borggaard2018bayesian} and in computational chemistry
\cite{ChandlerWolynes,
  Miller2005,Miller2005a,Craig2004,Craig2005,Craig2005a,Habershon2008,
  Habershon2013,Lu2018,KoBoMi2019} such Markov chain Monte Carlo
methods (MCMC) play a critical role by efficiently resolving
high-dimensional distributions possessing complex multimodal and
correlation structures which typically arise.  However,
notwithstanding their broad use in a variety of applications, the
theoretical and practical understanding of the mixing rates of these
chains remains poorly understood.

The initial mathematical foundation of MCMC methods was set in the
late 40's by Metropolis and Ulam in \cite{Metropolis49}, and later
improved with the development of the Metropolis-Hastings algorithm in
\cite{metropolis1953equation,hastings1970monte}.  Further notable
developments in the late 80's and 90's derived MCMC algorithms based
on suitable Hamiltonian \cite{DuKePeRo1987, neal1993probabilistic} and
Langevin dynamical systems \cite{grenander1994representations,
  besag1994comments}.  See e.g.  \cite{Betancourt2019, Li2008,
  robert2013monte} for a further general overview of the field.  In
view of exciting applications for the Bayesian approach to PDE inverse
problems and in transition path sampling \cite{hairer2005analysis,
  hairer2007analysis, ReVa2005, HaStVo2009, HSV2011,
  stuart2010inverse, martin2012stochastic, bui2014analysis,
  dashti2017bayesian, PetraEtAl2014, BuiNguyen2016,
  borggaard2018bayesian}, an important recent advance in the MCMC
literature \cite{tierney1998note,BeRoStVo2008, BePiSaSt2011,
  cotter2013mcmc} concerns the development of algorithms which are
well defined on infinite-dimensional spaces.  These methods have the
scope to partially beat the `curse of dimensionality' since one
expects that the number of samples required to effectively resolve the
target distribution to be independent of the degree of numerical
discretization.  However validating such claims of efficacy concerning
this recently discovered class of infinite dimensional algorithms both
in theory and in practice is an exciting and rapidly developing
direction in current research.

This work provides an analysis of mixing rates for one particular
class of methods among the MCMC algorithms mentioned above, known as
Hybrid or Hamiltonian Monte Carlo (HMC) sampling; cf.
\cite{DuKePeRo1987,Li2008,neal2011mcmc,BePiSaSt2011}. For HMC sampling
the general idea consists in taking advantage of a Hamiltonian dynamic
taylored to the structure of the target $\mu$, a distribution which
functions as the marginal onto position space of the Gibbs measure
associated to the dynamics.  As such this `Hamiltonian approach'
produces nonlocal and nonsymmetric moves on the state space, allowing
for more effective sampling from distributions with complex
correlation structures in comparison to more traditional random walk
based methods.  Indeed the efficacy of the HMC approach has led to its
widespread adoption in the statistics community as exemplified for
example by the success of the STAN software package
\cite{GelmanLeeGuo2015, Stan2016}. However, notwithstanding notable
recent work, the theoretical understanding of optimal mixing rates for
HMC based methods remains rather incomplete both in terms of optimal
tuning of algorithmic parameters and in terms of the allowable
structure of the target measure admitted by the theory \cite{BoEbPHMC,
  BoEbZi2018, LivingstoneEtAl2019, BoSaActaN2018, DurmusEtAl2017,
  BeskosEtAl2013, BeGiShiFaStu2017, BeKaPa2013, BuiNguyen2016, BoSa2016, MaPiSmi2018, MaSmi2017, MaSmi2019}.

We are particularly focused here on a version of HMC introduced in
\cite{BePiSaSt2011} where the authors consider a preconditioned
Hamiltonian dynamics in order to derive a sampler which is well
defined in the infinite-dimensional Hilbert space setting. While
recent work \cite{BuiNguyen2016, BeGiShiFaStu2017,
  borggaard2018bayesian} has shown that this `infinite-dimensional'
algorithm can be quite effective in practice, the question of rigorous
justification of mixing rates posed in \cite{BePiSaSt2011} as an open
problem has only very recently been addressed in the work
\cite{BoEbPHMC} in the case of \textit{exact}
(i.e. non-temporally-discretized) and preconditioned HMC. In
\cite{BoEbPHMC}, the authors follow an approach based on an exact
coupling method recently considered in \cite{EbGuZi2016, BoEbZi2018}. Here we develop an alternative approach to establishing
mixing rates for preconditioned HMC based on the so called weak Harris
theorem \cite{hairer2008spectral, hairer2011asymptotic,
  hairer2014spectral, butkovsky2018generalized} combined with suitable
`nudging' in the velocity variable which plays an analogous role to
that provided by the classical Foias-Prodi estimate in the ergodic
theory of certain classes of nonlinear SPDEs;
cf. \cite{mattingly2002exponential, kuksin2012mathematics,
  GlattHoltzMattinglyRichards2016}.  As such we believe the
alternative approach that we consider here to be more flexible in
certain ways providing a basis for further future analysis of MCMC
algorithms. Furthermore, our approach for the exact dynamics developed
here can be modified to derive mixing rates in the more interesting
and practical case for discretized HMC. This later challenge will be
taken up in future work.

Our main results can be summarized as follows. We show exponential
mixing rates for the exact preconditioned HMC with respect to an
appropriate Wasserstein distance in the space of probability measures
on $\mathbb{H}$. For suitable observables, we show that this mixing
implies a strong law of large numbers and a central limit
theorem. In addition, we use very similar arguments to obtain a novel
proof of mixing rates for the finite-dimensional HMC. Finally, the
second part of the paper is concerned with the application of the
theoretical mixing result to the PDE inverse problem of determining
a background flow from partial observations of a passive scalar that
is advected by the flow. A careful analysis of this inverse problem
within a Bayesian framework is carried out in
\cite{borggaard2018bayesian}, where the authors also provide numerical
simulations showing the effectiveness of the infinite-dimensional HMC
algorithm from \cite{BePiSaSt2011} in approximating the target
distribution in this case. Here our task is to show that this example,
for suitable observations of the passive scalar, satisfies all the
conditions needed for our theoretical mixing result to hold, thus
complementing the numerical experiments in
\cite{borggaard2018bayesian} with rigorous mixing rates. In the sequel
we provide a more detailed summary of the results obtained in the bulk
of this manuscript.

\subsection{Overview of the Main Results}
\label{sec:Main:Thm:Overview}

The preconditioned Hamiltonian Monte Carlo algorithm from
\cite{BePiSaSt2011} which we analyze here can be described as
follows. Fix a separable Hilbert space $\H$ with norm $| \cdot |$ and
inner product $\langle \cdot, \cdot \rangle$. Let $\mathcal{B}(\H)$
denote the associated Borel $\sigma$-algebra and let $\Pr(\H)$ denote
the set of Borel probability measures on $\H$. Suppose we wish to consider a
target measure $\mu \in \Pr(\H)$ which is given in the Gibbsian form
\begin{align}\label{eq:mu}
	\mu(d\bq ) \propto \exp( - U(\bq)) \mu_0(d\bq),
\end{align}
where $U: \H \to \RR$ is a potential function.  Here $\mu_0$ is a
probability measure on $\H$ typically corresponding to the prior
distribution when we consider a $\mu$ derived as a Bayesian posterior.
Following a standard formulation in the infinite dimensional setting,
we assume in what follows that $\mu_0$ is a centered Gaussian
distribution on $\H$, i.e.  $\mu_0 = \mN(0, \cC)$, with $\cC$ being a
symmetric, strictly positive-definite, trace-class linear operator on
$\H$.

Consider the following preconditioned Hamiltonian dynamics
\begin{align}
  \label{eq:dynamics:xxx}
  \frac{d\bq_t}{dt} = \bv_t,
  \quad
  \frac{d\bv_t}{dt} = - \bq_t - \cC D U(\bq_t),
  \quad \mbox{ with initial condition }
  (\bq_0, \bv_0) \in \H \times \H,
\end{align}
where $\bv \in \H$ denotes a `velocity' variable, so that
\eqref{eq:dynamics:xxx} describes the evolution of the 
`position-velocity' pair $(\bq, \bv)$ in the extended phase space
$\H \times \H$. Here we adopt the notation $\bq_t$ and $\bv_t$ to
denote the value at time $t$ of the variables $\bq$ and $\bv$,
respectively. The associated Hamiltonian function, a formal invariant
of the flow in \eqref{eq:dynamics:xxx}, is given by
\begin{align*}
  H(\bq, \bv)
  = \langle \cC^{-1} \bq, \bq \rangle + U(\bq)
  + \langle \cC^{-1} \bv, \bv \rangle
  \quad \text{ for suitable } (\bq, \bv) \in \H \times \H.
\end{align*}

The exact preconditioned HMC algorithm works as follows. Starting from
any $\bq_0 \in \H$, draw $\bv_0 \sim \mN(0, \cC)$ and run the
Hamiltonian dynamics with initial condition $(\bq_0, \bv_0)$ for a
chosen temporal duration $T > 0$. Thus a forward step is proposed as
the projection on the $\bq$-coordinate of the solution of
\eqref{eq:dynamics:xxx} starting from $(\bq_0, \bv_0)$ at time $T$,
i.e. $\bq_T(\bq_0, \bv_0)$. The associated Markov transition kernel
$P: \H \times \mathcal{B}(\H) \to [0,1]$ is then given as
\begin{align}
 	P(\bq_0, A) = \Prb( \bq_T(\bq_0, \bv_0) \in A) 
  \quad \text{ with } \bv_0 \sim \mathcal{N}(0,\cC),
  \label{eq:HMC:kernel:overview}
\end{align}
for every $A \in \mB(\H)$. 
We adopt the notation $P^n$ for $n$
steps of the Markov kernel $P$ and recall that $P$ acts as
\begin{align*}
  \nu P(\cdot) = \int P(\bq, \cdot) \nu(d \bq), \quad
  P\Phi(\cdot) = \int \Phi(\bq) P(\cdot, d\bq) 
\end{align*}
on measures $\nu \in \Pr(\H)$ and observables $\Phi: \H \to \RR$,
respectively. This kernel $P$ leaves invariant the desired target
probability measure $\mu$ given in \eqref{eq:mu}, namely
$\mu P = \mu$, as was demonstrated in \cite{BePiSaSt2011} and recalled
in \cref{eq:invariance} below. Clearly, in practice, one is not able
to integrate \eqref{eq:dynamics:xxx} exactly so that one must instead
resort to suitable numerical discretizations. These numerical
integration schemes are designed so as to ensure that fundamental
properties of Hamiltonian dynamics are preserved, such as time
reversibility and volume-preservation or `symplectiness' --see
e.g. \cite{BoSaActaN2018} for a survey. In this work we only analyze
the exact dynamics, as the discretized case requires additional
techniques and will be the subject of future work.

Let us now sketch a simplified version of our main result, given in
rigorous and complete detail in \cref{thm:weak:harris} below. Our
mixing result for the Markov kernel $P$ defined in
\eqref{eq:HMC:kernel:overview} is given with respect to a suitably
constructed Wasserstein distance on $\Pr(\H)$.  Namely, starting from
$\varepsilon > 0$ and $\eta > 0$, consider
$\tilde \rho: \H \times \H \to \mathbb{R}^+$ defined as
\begin{align}
   \tilde{\rho}(\bq, \btq) 
  := \sqrt{\left(
  \frac{| \bq - \btq |}{\varepsilon } \wedge 1 \right)
  \left(1 + \exp(\eta |\bq|^2) + \exp(\eta |\btq|^2) \right)
  }.
  \label{eq:t:rho:def}
\end{align}	
Here $\varepsilon$ corresponds to the small scales at which we can
match small perturbations in the initial position $\bq_0$ with a
corresponding perturbation in the initial velocity $\bv_0$ in
\eqref{eq:dynamics:xxx}. On the other hand, for sufficiently small
$\eta >0$, the function $V(\bq) = \exp(\eta |\bq|^2)$ is a
\emph{Foster-Lyapunov} (or, simply, \emph{Lyapunov}) function for $P$
in the sense of \cref{def:Lyap} and \cref{prop:FL} below.

The mapping $\tilde \rho$ is a \textit{distance-like function} in
$\H$, i.e. it is a symmetric and lower-semicontinuous non-negative
function such that $\tilde \rho(\bq, \btq) = 0$ holds if and only if
$\bq = \btq$. We denote by
$\Wass_{\tilde{\rho}}: \Pr(\H) \times \Pr(\H) \to \mathbb{R}^+ \cup
\{\infty\}$ the following extension of $\tilde \rho$ to $\Pr(\H)$:
\begin{align}
  \Wass_{\tilde{\rho}} ( \nu_1, \nu_2) 
      = \inf_{\Gamma \in \Co(\nu_1, \nu_2)} 
  \int_{\Hp \times \Hp} \tilde \rho(\bq, \btq) \Gamma(d \bq, d \btq),
  \label{eq:Wass:Def}
\end{align}
where $\Co(\nu_1,\nu_2)$ denotes the set of all \emph{couplings} of
$\nu_1$ and $\nu_2$, i.e. the set of all measures
$\Gamma \in \Pr(\H \times \H)$ with marginals $\nu_1$ and $\nu_2$.  We
notice that, on the other hand, the mapping
$\rho(\bq,\btq) = (|\bq - \btq|/\varepsilon) \wedge 1$ defines a
standard metric in $\H$. As such, its associated extension
$\Wass_\rho$ to $\Pr(\H)$ coincides with the usual Wasserstein-1
distance, \cite{villani2008optimal}.

With the above notation, we have the following convergence result. For
the complete, detailed and general formulation, see
\cref{thm:weak:harris} below.

\begin{Thm}
  \label{thm:Main:Result:Overview}
  Suppose that $\cC$ is a symmetric strictly positive-definite trace
  class operator and that $U \in C^2(\H)$ satisfies the global bound
  \begin{align}
    L_1 := \sup_{\bq \in \H} |D^2U(\bq)| < \infty
    \label{eq:intro:global:hessian:c}
  \end{align}
  and the following dissipativity condition
  \begin{align}
    | \bq|^2 + \langle \bq, CDU(\bq)\rangle \geq L_2 |\bq|^2 - L_3
    \quad \mbox{ for all } \bq \in \H,
  \label{eq:disp:cond:over}
\end{align}
for some constants $L_2 > 0$ and $L_3 \geq 0$. Let $\lambda_1$ denote
the largest eigenvalue of $\cC$. 

Then, there exists an integration time $T = T(\lambda_1, L_1, L_2)$
for which the associated Markov kernel $P$ as defined in
\eqref{eq:HMC:kernel:overview} satisfies, with respect to
$\tilde \rho$ defined in \eqref{eq:t:rho:def},
\begin{align}\label{ineq:main:result}
    \Wass_{\tilde \rho}( \nu_1 P^n, \nu_2 P^n)  
  \le c_1 e^{-c_2 n}  \Wass_{\tilde \rho}( \nu_1, \nu_2)
  \quad \mbox{ for any } \nu_1, \nu_2 \in \Pr(\H) 
  \mbox{ and } n \in \mathbb{N},
\end{align} 
for some $\varepsilon > 0$ as in \eqref{eq:t:rho:def} and some
positive constants $c_1, c_2$ which depend only on the integration
time $T > 0$, the constants $L_i$, $i=1,2,3$, associated to the
potential function $U$, and the covariance operator $\cC$. In
particular, \eqref{ineq:main:result} implies that $\mu$ defined in
\eqref{eq:mu} is the unique invariant measure for $P$.  Moreover,
taking $\nu_1 = \delta_{\bq_0}$, the Dirac delta concentrated at some
$\bq_0 \in \H$, and $\nu_2 = \mu$, it follows from
\eqref{ineq:main:result} that $P^n(\bq_0,\cdot)$ converges
exponentially to $\mu$ with respect to $\Wass_{\tilde \rho}$ as
$n \to \infty$. In addition, for any suitably regular observable
$\Phi: \H \to \RR$,
\begin{align*}
\left| P^n \Phi(\bq_0) - \int \Phi(\bq') \mu(dq') \right|
          \leq L_\Phi c_1 e^{-n c_2} \int \sqrt{1 + \exp(\eta |\bq_0|^2) 
    + \exp(\eta |\bq'|^2)} \mu(d \bq') \quad \mbox{ for all } n \in \NN,
\end{align*}
for some $\eta > 0$ and $L_\Phi >0$.

Further, taking $\{Q_n(\bq_0) \}_{n \in \mathbb{N}}$ to be the
process associated to $\{P^n\}_{n \in \mathbb{N}}$ starting from
$\bq_0 \in \H$, i.e. $Q_n(\bq_0) \sim P(Q_{n-1}(\bq_0), \cdot)$ we
have, for any $\bq_0 \in \H$ and any suitably regular observable
$\Phi: \H \to \RR$, that
\begin{align*}
  X_n
  := \frac{1}{n} \sum_{k =1}^n \Phi( Q_k(\bq_0))
  - \int \Phi(\bq) \mu(d \bq)
  \rightarrow 0 \quad \mbox{ as $n \to \infty$ almost surely}
\end{align*}
and that
\begin{align*}
  \Prb( a <  \sqrt{n} X_n \leq b)
  \to \frac{1}{\sqrt{2 \pi \sigma^2}}\int_a^b
  e^{-\frac{x^2}{2 \sigma^2}} dx \quad \mbox{ as $n \to \infty$ for any } 
  a,b \in \RR \mbox{ with } a < b,
\end{align*}  
where $\sigma = \sigma(\Phi)$. In other words,
$\{Q_n(\bq_0) \}_{n \geq 0}$ satisfies a strong law of large numbers
(SLLN) and a central limit theorem (CLT).
\end{Thm}

With similar arguments as used in the proof of \cref{thm:Main:Result:Overview} 
(cf. \cref{thm:weak:harris}), we can also provide a new proof of mixing rates 
for the classical finite-dimensional HMC algorithm, as specified by the dynamics 
\eqref{eq:FD:H:dynam}. This is carried out in \cref{thm:main:FD} below, and 
complemented by further comparisons with the assumptions in the main infinite-dimensional 
result in \cref{rmk:FD}.

Having formulated our mixing result for the exact HMC algorithm
associated with \eqref{eq:mu} we would like to be able to demonstrate
that the conditions
\eqref{eq:intro:global:hessian:c}-\eqref{eq:disp:cond:over} which we
impose on the potential $U$ can be verified in concrete examples
specifically as would apply to the Bayesian approach to PDE inverse
problems.  Here, as an illustrative example, we consider the problem of
recovering a divergence free fluid flow $\bq$ from the sparse and
noisy observation of a passive solute $\theta(\bq)$ as was recently
studied in \cite{borggaard2018bayesian, borggaard2018Consistency}.

To be specific let
\begin{align}
  \pd_t \theta + \bq \cdot \nabla \theta = \kappa \Delta \theta,
  \quad \theta(0) = \theta_0
  \label{eq:ad:eqn}
\end{align}
where the solution evolves on the periodic box $\TT^2$, namely
$\theta : [0,\infty) \times \TT^2  \rightarrow \RR$ and $\kappa > 0$ is
a fixed diffusion parameter.  Given a sufficiently regular initial
condition $\theta_0 : \TT^2 \to \RR$, which we take to be known in
advance, we specify the (linear) observation procedure
\begin{align}
  \Obs(\theta) := \left\{  \int_0^\infty \int_{\TT^2} \theta (t,x) K_j(t,x)
  dx dt \right\}_{j = 1}^m
  \label{eq:gen:linear:form}
\end{align}
where $m \geq 1$ represents the number of separate observations of
$\theta$ and $K_j$ are the associated `observation kernels'.  Positing an
additive observation noise $\eta$, we have the following statistical
model linking any suitably regular, divergence free,
$\bq: \TT^2 \to \RR^2$ with a resulting data set $\data$ as
\begin{align*}
  \data = \Obs(\theta(\bq)) + \eta,
\end{align*}
where $\theta(\bq)$ represents the solution of
\eqref{eq:gen:linear:form} corresponding to $\bq$ so that
$\theta(\bq)$ sits in an appropriate solution space which we specify
in rigorous detail below in \cref{def:adr_weak}.

Following the Bayesian statistical inversion formalism
\cite{kaipio2005statistical, dashti2017bayesian}, given a fixed
observation $\data \in \RR^m$ and a prior distribution $\mu_0$ on a
suitable Hilbert space of divergence free, periodic vector fields and
a probability density function $p_\eta: \RR^m \to \RR$ for the
observation noise $\eta$, we obtain a posterior distribution
\begin{align}
  \mu^\data(d\bq) \propto \exp(- U(\bq))\mu_0(d \bq)
  \quad
  \text{ where }
  \quad
  U(\bq) = - \log(p_\eta(\data - \Obs(\theta(\bq))).
  \label{eq:post:passive:scal}
\end{align}
see e.g. \cite{dashti2017bayesian}, \cite[Appendix
C]{borggaard2018bayesian}.  For simplicity of presentation, we focus
here on the typical situation where $\eta \sim N(0, \Gamma)$, with
$\Gamma$ a symmetric, strictly positive definite covariance
operator on $\RR^m$. In this case $U$ takes the form
\begin{align}
  U(\bq) =  | \Gamma^{-1/2}(\data - \Obs(\theta(\bq)))|^2,
  \label{eq:pot:passive:scal:g:n}
\end{align}
where $| \cdot |$ represents the usual Euclidean norm on $\RR^m$.

Our main results here, \cref{prop:DU:DsqU:Bnds} and
\cref{cor:DU:DsqU:Bnd:w:C}, show that in the case of `spectral
observations', i.e. when
\begin{align*}
  |\Obs( \theta)|  \leq c_0 \sup_{t \leq t^*} \int_{\TT^2}|\theta(t,x)|^2 dx, 
\end{align*}
for some $t^* \geq 0$, we can verify the conditions imposed on the
potential function $U$ (cf. \eqref{eq:intro:global:hessian:c} and more
generally \cref{B123} below) and in particular establish suitable
global bounds on $D^2U$.  On the other hand, for the interesting cases
of `point-observations' where
\begin{align}
  |\Obs( \theta)| \leq c_0 \sup_{t \leq t^*, x \in \TT^2} |\theta(t,x)|
  \label{eq:pnt:obs:informal}
\end{align}
for some $t^* \geq 0$, or for observations involving gradients or
other higher order derivatives of $\theta$, we can only show local
bounds on $D^2U$.

\subsection*{Overview of the Proof}

Our proof follows the approach of the weak Harris theorem developed in
\cite{hairer2011asymptotic}, which is an elegant generalization of the
classical Harris mixing results, \cite{harris1956existence,
  meyn2012markov, hairer2011yet}. It establishes necessary conditions
for two point contraction at small, intermediate and large scales in a
fashion well adapted to the Wasserstein metric, a notion of distance
which is crucially needed for many types of processes evolving on
infinite dimensional spaces. We should emphasize the authors in
\cite{hairer2011asymptotic} provide clarity and flexibility in their
approach by developing a class of distance-like functions
(cf. \eqref{eq:t:rho:def}) which allows one to establish global
contractivity directly and thus avoiding the need for intricate
pathwise coupling constructions considered elsewhere in the literature.

As such, the main difficulties here lie in showing that the necessary
assumptions of the weak Harris theorem are valid in our context.
These assumptions amount to showing, with respect to
$\rho: \H \times \H \to [0,1]$ defined as
$\rho (\bq, \btq) = 1 \wedge(|\bq - \btq|/\varepsilon)$, with
$\varepsilon > 0$ fixed, that the following is true: there exists
$m \in \mathbb{N}$ sufficiently large such that
\begin{enumerate}[(i)]
\item $P^m$ is $\rho$-contracting, i.e. there exists
  $0 < \delta_1 < 1$ such that
  \begin{align}
    \label{oview:contr}
    \Wass_\rho(P^m(\bq_0, \cdot), P^m(\btq_0, \cdot))
    \leq \delta_1 \rho(\bq, \btq) \quad \mbox{ for all }
    \bq_0, \btq_0 \in \H \mbox{ with }
    \rho(\bq_0, \btq_0) < 1;
  \end{align}
\item For level sets of the form
  $A_K:= \{\bq \in \H \,: \, | q | \leq K\}$, for $K > 0$, $A_K$ is
  \emph{$\rho$-small} for $P^m$, i.e. there exists $0 < \delta_2 < 1$
  and $m \geq 1$ such that
  \begin{align}\label{oview:small}
    \Wass_\rho (P^m(\bq_0, \cdot), P^m(\btq_0, \cdot)) \leq 1 -
    \delta_2
    \quad \mbox{ for all } \bq_0, \btq_0 \in A_K.
  \end{align}
\end{enumerate}
Finally we need a Lyapunov condition:
\begin{itemize}
\item[(iii)] For a suitable $V: \H \to \RR^+$ that
  \begin{align}\label{oview:Lyap}
    P^n V(\bq) \leq \kappa_V^n V(\bq) + K_V,
  \end{align}
  for every $\bq \in \H$ and $n \geq 1$ where $\kappa_V \in (0,1)$ and
  $K_V > 0$ are independent of $\bq$ and $n$.
\end{itemize}
Roughly speaking the conditions (i)--(iii) correspond to establishing
a two-point contraction at small, intermediate and large
scales respectively.

Following an approach developed in the stochastic PDE literature
\cite{mattingly2002exponential, kuksin2012mathematics,
  GlattHoltzMattinglyRichards2016,
  hairer2008spectral,hairer2011asymptotic, butkovsky2018generalized},
the idea consists in establishing (i) and (ii) above without
explicitly constructing a coupling between $P^m(\bq_0, \cdot)$ and
$P^m(\btq_0, \cdot)$. Instead, we construct an `approximate' coupling
by defining a modified process $\tilde P(\bq_0, \btq_0, \cdot)$ in a
control-like approach.  We define the process $\tP$ by imposing a
suitable `shift' in the initial velocity $\bv_0$ in
\eqref{eq:HMC:kernel:overview} depending on the initial positions
$\bq_0$, $\btq_0$. Namely, for a fixed integration time $T > 0$,
we take
 \begin{align}\label{oview:def:tP} 
  \tP(\bq_0, \btq_0, A) := \mathbb{P} (\bq_T (\btq_0, \tilde\bv_0) \in A) 
   \quad \mbox{ with }\tilde\bv_0 = \bv_0 + \cS(\bq_0, \btq_0), 
   \quad \bv_0 \sim \mN(0, \cC),
\end{align}
for every $A \in \mB(\H)$. Here we consider a shift
$\cS(\bq_0, \btq_0)$ which is inspired by estimates developed in
\cite{BoEbZi2018}; $\cS$ is defined so as to ensure a suitable
contraction between two solutions of \eqref{eq:dynamics:xxx} starting
from $(\bq_0, \bv_0)$ and $(\btq_0, \tilde \bv_0)$ at the final time
$T > 0$.

Since $\rho$ is a metric in $\H$, the corresponding extension
$\Wass_\rho$ is a metric in $\Pr(\H)$ and in fact coincides with the
Wasserstein-1 distance. Thus, by the triangle inequality,
\begin{align}\label{ineq:Wrho:TI}
  \Wass_\rho(P^m(\bq_0, \cdot), P^m(\btq_0, \cdot)) 
  \leq \Wass_\rho(P^m(\bq_0, \cdot), \tP^m(\bq_0, \btq_0, \cdot)) 
    + \Wass_\rho(\tP^m(\bq_0, \btq_0, \cdot), P^m(\btq_0, \cdot)),
\end{align}
where $\tP^m$ denotes the $m$-fold iteration of $\tP$, corresponding
to a sequence $(\bv_0^{(1)}, \ldots, \bv_0^{(m)})$ of initial
velocities drawn from $\mN(0, \cC)$ and shifted as in
\eqref{oview:def:tP} with $\bq_0, \btq_0$ replaced with the starting
positions from each iteration. In view of establishing
\eqref{oview:contr} and \eqref{oview:small}, the first term on the
right-hand side of \eqref{ineq:Wrho:TI} is estimated by first showing
a contraction result between two solutions of \eqref{eq:dynamics:xxx}
starting from $(\bq_0, \bv_0)$ and $(\btq_0, \tilde \bv_0)$ with
respect to $\rho$ in $\H$, which is then extended to $\Wass_\rho$ in
$\Pr(\H)$. Such contraction result follows solely from assumption
\eqref{eq:intro:global:hessian:c} on the potential function $U$
together with a smallness assumption on the integration time $T$; see
\cref{prop:FP:Type:contract} below. Moreover, assumption
\eqref{eq:intro:global:hessian:c} implies that the only possible
source of nonlinearity in the dynamics \eqref{eq:dynamics:xxx},
i.e. $DU$, is Lipschitz, which in particular guarantees the
well-posedness of \eqref{eq:dynamics:xxx} as we detail in
\cref{prop:well:posed}.

The second term on the right-hand side \eqref{ineq:Wrho:TI} represents
a `cost of control' term and in fact the tuning parameter
$\varepsilon$ appearing in $\rho$ specifies the scales at which this
cost does not `become too large'.  We estimate this term with the help
of Girsanov's theorem from which we obtain a bound in terms of the
Radon-Nikodym derivative between the law $\sigma_m$ of the velocity
path $(\bv_0^{(1)}, \ldots, \bv_0^{(m)})$ and the law
$\tilde \sigma_m$ of the associated shifted velocity path
$(\tilde\bv_0^{(1)}, \ldots, \tilde \bv_0^{(m)})$, i.e.  Girsanov
provides us with $d \sigma_m/ d \tilde\sigma_m$. Here we notice that,
in order to guarantee that $d \sigma_m/ d \tilde\sigma_m$ is
well-defined, we define the shift $\cS$ in \eqref{oview:def:tP} to be
in a finite-dimensional subspace of $\H$
(cf. \eqref{def:Phin}). Indeed, looking at the case $m = 1$ for
simplicity, notice that if $\bv_0 \sim \mN(0,\cC)$ then
$\tilde \bv_0 \sim \mN(\cS(\bq_0, \btq_0), \cC)$ and, by the
Feldman-Hajek theorem (see, e.g., \cite[Theorem 2.23]{DaZa2014}),
$\mN(0,\cC)$ and $\mN(\cS(\bq_0, \btq_0), \cC)$ are mutually singular
unless $\cS(\bq_0, \btq_0)$ belongs to the Cameron-Martin space of
$\mN(0,\cC)$. Notably, the Cameron-Martin space of $\mN(0,\cC)$ has
$\mN(0,\cC)$-measure zero when $\H$ is infinite-dimensional. This
illustrates the fact that two measures in an infinite-dimensional
space are frequently mutually singular. However, by considering a
velocity shift $\cS$ that belongs to an N-dimensional subspace
$\H_N \subset \H$, for some $N \in \NN$, we can show that $\sigma_m$
and $\tilde\sigma_m$ are mutually absolutely continuous, with an
estimate of $d\sigma_m/d\tilde \sigma_m$, and thus of the second term
in \eqref{ineq:Wrho:TI}, that depends on the dimension $N$. Here $N$
is chosen so as to obtain a suitable contraction between different
trajectories of \eqref{eq:dynamics:xxx} and hence to provide a useful
estimate of the first term in \eqref{ineq:Wrho:TI} (see
\cref{prop:FP:Type:contract} and \cref{prop:local_contr}). For this
purpose, $N$ must be chosen to be sufficiently large, but is
nevertheless a fixed parameter depending only on the potential
function $U$ through the constant $\Ua$ from
\eqref{eq:intro:global:hessian:c} (see \eqref{cond:t:N:contr} below).
 
The third part of the proof consists in showing that such $V$ is a
Lyapunov function for $P$ as given in \cref{prop:FL} below.  Here, in
addition to quadratic exponential function $V(\bq) = \exp(\eta |q|^2)$
as in \eqref{eq:t:rho:def} we in fact show that any function of the
form $V(\bq) = |\bq|^i$, $i \in \mathbb{N}$, is also a Lyapunov function. The
result of \cref{prop:FL} follows from both assumptions
\eqref{eq:intro:global:hessian:c} and \eqref{eq:disp:cond:over} on the
potential $U$ together with a smallness assumption on the integration
time $T$. Notably, assumption \eqref{eq:disp:cond:over} on $U$ is only
imposed in order to obtain this Lyapunov structure. Indeed, condition
\eqref{eq:disp:cond:over} provides a coercivity-like property for $DU$
in \eqref{eq:dynamics:xxx} which, when complemented with the smallness
assumption on $T$, allows us to show the required exponential decay of
such functions $V$ modulo a constant, thus proving the Lyapunov
property.  


It remains to leverage the spectral gap now established,
\eqref{ineq:main:result}, to prove a Law of Large numbers (LLN) and
Central Limit Theorem (CLT) type result for the implied Markov
process.  While this implication is extensively developed in the
literature, and recently generalized to the situation where the
spectral gap appears in the Wasserstein sense
\cite{komorowski2012central,kulik2017ergodic}, it was not immediately
clear that these results are easily applied as a black box to our
situation. Instead, for clarity of presentation, we provide an
independent proof of the LLN and CLT in an appendix which is carefully
adapted to our situation where the $\tilde{\rho}$ in
\eqref{ineq:main:result} is only distance-like.  While we are in
particular following the road map laid out in
\cite{komorowski2012central}, we believe our proof may be of some
independent interest.

\subsection*{Organization of the Manuscript}

The rest of the manuscript is organized as follows.  In
\cref{sec:Preliminaries} we provide the complete details of our
mathematical setting including the assumptions on the covariance
operator $\cC$ and the potential $U$ in \eqref{eq:dynamics:xxx}.
\cref{sec:apriori} provides certain a priori bounds on
\eqref{eq:dynamics:xxx} and concludes with the low-mode nudging bound
that we use to synchronize the positions of two processes by suitably
coupling their momenta.  Lyapunov estimates on the exact Hamiltonian
Monte Carlo dynamics are given in \cref{sec:f:l:bnd}.  In
\cref{sec:Ptwise:contract:MD} we combine the bounds in the previous
two sections to establish the pointwise contractivity of the Markovian
dynamics, namely the so called $\rho$-small and $\rho$-contractivity
conditions.  The main result on geometric ergodicity is stated
rigorously in \cref{sec:Proof:Main} followed by the proof using the
weak Harris theorem \cite{hairer2011asymptotic}.
\cref{sec:finite:dim:HMC} details how our approach also provides a
novel proof for the finite dimensional setting.  Finally in
\cref{sec:bayes:AD} we establish that the conditions of the main
theorem apply to the Bayesian statistical inversion problem of
estimating a divergence free vector field $\bq$ from the partial
observation of a scalar quantity advected by the flow.
\cref{sec:LLN:CLT} shows how the law of large numbers and the central
limit theorem follow in our setting from our main result on spectral
gaps.

\section{Preliminaries}
\label{sec:Preliminaries}

This section collects various mathematical preliminaries and sets down
the precise assumptions which we use below in the statements of the 
main results of the paper.

\subsection{The Gaussian reference measure}
Let $\H$ be a separable and real Hilbert space with inner product
$\langle \cdot, \cdot \rangle$ and norm $| \cdot |$.  We take
$\mathcal{N}(0, \cC)$ to denote the centered normal distribution on
$\H$ with covariance operator $\cC$.  See
e.g. \cite{bogachev1998gaussian, DaZa2014} for generalities concerning
Gaussian measures on Hilbert space.  In this paper we {\em always}
assume that $\cC$ satisfies the following conditions.
\begin{assumption}
\label{A123}
$\cC: \H \to \H$ is a trace class, symmetric and strictly positive
definite linear operator.  Thus, by the spectral theorem, we have a
complete orthonormal basis $\{ \bfe_i \}_{i \in \mathbb{N}}$ of $\H$
which are the eigenfunctions of $\cC$.  We write corresponding
eigenvalues $\{ \lambda_i \}_{i \in \mathbb{N}}$ in non-increasing order
and note that the trace class condition amounts to
  \begin{align}
    \Tr(\cC) := \sum_i \lambda_i < \infty. 
    \label{eq:Tr:Def}
  \end{align}
\end{assumption} 
\noindent We will also make frequent use of fractional powers of $\cC$
which we define as follows.
\begin{Def}
  \label{def:frack:pow}
  For any $\gamma  \in \RR$, we define fractional power $\cC^\gamma$ 
  of $\cC$ by
  \begin{align*}
    \cC^{\gamma} \bff = \sum_{i} \lambda_i^{\gamma} \langle \bff,  \bfe_i \rangle \bfe_i \;,
  \end{align*} 
  which makes sense for any $\bff \in \H_{\gamma}$. Here $\H_\gamma$ is 
  defined as
  \begin{align}\label{def:Hgamma}
    \H_{\gamma} = \{ \bff \in \H | \;  | f |_{\gamma} < \infty\}
    \quad \text{ where }
     | \bff |_{\gamma}^2
    :=  |\cC^{-\gamma} \bff |^2 
    = \sum_i \lambda_i^{-2\gamma} \langle \bff,  \bfe_i \rangle^2
  \end{align}
  when $\gamma \geq 0$.  For $\gamma < 0$, $\H_\gamma$ is defined as
  the dual of $\H_{-\gamma}$ relative to $\H$. In addition, for every
  $\gamma \in \RR$, we define the inner product
  $\langle \cdot, \cdot \rangle_{\gamma} = \langle \cC^{-\gamma}
  \cdot, \cC^{-\gamma} \cdot \rangle$.
\end{Def}
\noindent
According to  \cref{def:frack:pow}, it follows that
$\H_{-\tilde{\gamma}} \subseteq \H_{-\gamma}$ for every
$\gamma, \tilde{\gamma} \in \RR$ with $\gamma \geq \tilde{\gamma}$.  Moreover,
note that $\H_{1/2}$ is the Cameron-Martin space associated with
$\mathcal{N}(0, \cC)$ with inner product
$\langle \cdot, \cdot \rangle_{1/2} = \langle \mathcal{C}^{-1/2}
\cdot, \mathcal{C}^{-1/2} \cdot \rangle$
and norm $| \cdot |_{1/2} = | \mathcal{C}^{-1/2} \cdot |$;
see \cite[Chapter 2]{DaZa2014}.

In terms of these fractional spaces $\H_\gamma$ we have the following
`Poincar\'e' and `reverse-Poincar\'e' inequalities.  For this purpose
and for later use we define, for $N \geq 1$,
\begin{align}
  \Pi_N \bff = \sum_{j \leq N} \langle \bff, \bfe_j \rangle \bfe_j,
  \quad 
  \Pi^N \bff = \sum_{j > N} \langle \bff, \bfe_j \rangle \bfe_j,
  \label{eq:get:high:low}
\end{align}
namely the projection of $\bff \in \H$ onto `low' and `high' modes.
\begin{Lem}\label{lem:ineq:C}
  Given any $\gamma, \tilde{\gamma} \in \RR$ with
  $\gamma \geq \tilde{\gamma}$, the following hold:
  \begin{align}
    \norm{\cC^{\gamma} \bff} \leq 
    \lambda_1^{(\gamma - \tilde{\gamma})} \norm{\cC^{\tilde{\gamma}} \bff},
    \label{eq:Poincare}
  \end{align}
  when $\bff \in \H_{-\tilde{\gamma}}$. Moreover, for any $N \geq 1$,
  \begin{align}
    \norm{\cC^{\gamma} \Pi^N \bff} 
    \leq \lambda_{N+1}^{(\gamma - \tilde{\gamma})} 
    \norm{\cC^{\tilde{\gamma}} \Pi^N \bff },
    \label{eq:Inv:Poincare}
  \end{align}
  for any $\bff \in \H_{-\tilde{\gamma}}$.
\end{Lem}

In certain applications, one may wish to define the Markovian dynamics
associated to \eqref{eq:dynamics:xxx} only on $\H_\gamma$ for some
$\gamma \in (0, 1/2)$, which is a strict subset of $\H$. For this
reason, in what follows we consider our underlying phase space to be
more generally given by $\H_\gamma$, for some $\gamma \in [0, 1/2)$.
This leads us to introduce the following additional assumption which
will sometimes be imposed:
\begin{assumption}
\label{ass:higher:reg:C}
  For some $\gamma \in [0,1/2)$, $\cC^{1-2\gamma}$ is trace class. Namely,
  \begin{align}
    \Tr(\cC^{1 - 2 \gamma}) := \sum_i \lambda_i^{1- 2\gamma} < \infty. 
    \label{eq:Tr:Def:b}
  \end{align}
\end{assumption}
\noindent Under \cref{ass:higher:reg:C} we have the following regularity
property
\begin{Lem}
  \label{lem:diff:gauss:for:diff:folks}
  Suppose that $\mu_0$ is $\mathcal{N}(0,\cC)$ defined on $\H$ with $\cC$
  under \cref{A123}, \cref{ass:higher:reg:C}.    Then $\mu_0$
  is also $\mathcal{N}(0,\cC^{1 -2\gamma})$ defined on $\H_\gamma$.
\end{Lem}

\begin{Rmk}
  We typically think of the covariance $\cC$ as a `smoothing
  operator'. A simple example of $\cC$ satisfying the above
  assumptions is $A^{-1}$ where $A = - \partial_{xx}$ is the second
  derivative on $[0, \pi]$ endowed with Dirichlet boundary conditions.
  Note that, with this choice of $\cC$, the spaces $\H_\gamma$
  correspond to the usual $L^2$-based Sobolev space $H^{\gamma/2}$
  with the Cameron-Martin space given by $H^1$.  A more involved
  variation on this theme will be considered below in
  \cref{sec:bayes:AD} when we consider an application to a PDE inverse
  problem.
\end{Rmk}

\subsection{Conditions on the potential}
In what follows we impose the following regularity conditions on the
potential energy function $U$ from \eqref{eq:mu}. Note that in particular assumption $(B1)$ below is compatible with the setting imposed in \cite{BePiSaSt2011}; see \cref{rmk:Comp:beskos} below.
\begin{assumption}\label{B123}
 For a fixed value of $\gamma \in [0,1/2)$ the potential in
 \eqref{eq:dynamics:xxx} $U: \H_\gamma \to \RR$ is twice
 Fr\'echet differentiable  and 
 \begin{itemize}
 \item[(B1)] There exists $\Ua >0$ such that
   \begin{align} \label{Hess:bound}
     | D^2 U(\bff) |_{\mL_2(\H_\gamma)}  
     =  | \cC^{\gamma} D^2 U(\bff) \cC^{\gamma}|_{\mL_2(\H)} \leq \Ua 
   \end{align}
   for any $\bff \in \H_\gamma$, where $|\cdot|_{\mL_2(\H_\gamma)}$
   and $| \cdot|_{\mL_2(\H)}$ denote the usual operator norms for real
   valued bilinear operators defined on $\H_\gamma \times \H_\gamma$
   and on $\H \times \H$, respectively.
\item[(B2)] There exists $\Ub >0$ and $\Uc \ge 0$ such that, for this
  value of $\gamma \in [0,1/2)$
  \begin{align} \label{dissip:new}
    \nga{\bff}^2 + \langle  \bff, \cC D U(\bff) \rangle_{\gamma} 
    \ge \Ub \nga{\bff}^2 - \Uc
  \end{align}
  for every $\bff \in \H_\gamma$.
\end{itemize}
\end{assumption}

A number of remarks are in order regarding \cref{B123}:
\begin{Rmk}
  \label{rmk:B123:Consequences}
  \mbox{}
  \begin{itemize}
\item[(i)] \cref{B123} (B1) and the mean value theorem imply that
\begin{align}
  \ngad{DU(\bff)-DU(\bfg)} \le \Ua \nga{\bff - \bfg}
  \quad 
  \label{lipschitz:new}
\end{align}
for any $\bff, \bfg \in \H_\gamma$
and, in particular, 
\begin{align}
  \ngad{D U (\bff)} \leq \Ua \nga{\bff} + \Uz  
  \label{cond:sublin:U}
\end{align}
for every $\bff \in \H_\gamma$, where $\Uz = \ngad{D U (0)}$. Inequalities
\eqref{lipschitz:new} and \eqref{cond:sublin:U} will be used
extensively in the analysis below.

\item[(ii)] If $U$ satisfies, in addition, the following property:
\begin{itemize}
\item[(B3)] There exists $\Ud \in [0,\lambda_1^{-1 + 2 \gamma})$ and
  $\Ue \geq 0$ such that
	\begin{align}
	\ngad{D U(\bff)} \leq \Ud \nga{\bff} + \Ue, 
          \quad \text{ for any } \bff \in \H_\gamma,
          \label{eq:B3a} 
	\end{align}
\end{itemize}
then (B2) is automatically satisfied. Indeed, we have
\begin{align}\label{B3:to:B2}
	\nga{\bff}^2 + \langle \bff, \cC D U(\bff) \rangle_{\gamma} 
	&\geq \nga{\bff}^2 - | \langle \bff, \cC D U(\bff) \rangle_{\gamma}|  
	\geq \nga{\bff}^2 - \nga{\bff} |\cC^{1 - \gamma} D U (\bff)| \notag\\
	&\geq \nga{\bff}^2 - \lambda_1^{1 - 2 \gamma} \nga{\bff} \ngad{D U (\bff)},
\end{align}
where the last inequality follows from Lemma \ref{lem:ineq:C} and the
fact that $\gamma \in [0,1/2)$. Using \eqref{eq:B3a} in
\eqref{B3:to:B2} and Young's inequality, we obtain
\begin{align*}
  \nga{\bff}^2 + \langle \bff, \cC D U (\bff) \rangle_{\gamma}
	 &\geq (1 - \lambda_1^{1 - 2 \gamma} \Ud ) \nga{\bff}^2 
             - \lambda_1^{1 - 2 \gamma} \Ue \nga{\bff} \notag \\
	 &\geq \frac{1 - \lambda_1^{1 - 2 \gamma} \Ud}{2} \nga{\bff}^2 - C,
\end{align*}
where $C \in \RR^+$ is a constant depending on
$\lambda_1^{1 - 2\gamma}$, $\Ud$, $\Ue$.  Notice that, in particular,
if $U$ satisfies (B1) with $\Ua \in [0,\lambda_1^{-1 + 2 \gamma})$,
then (B3) is verified with $\Ud = \Ua$ and $\Ue = L_0$
(cf. \eqref{cond:sublin:U}).
\item[(iii)] Assumptions (B1) and (B2) imply that the constants $\Ua$
  and $\Ub$ satisfy the following relation:
\begin{align}
\label{ineq:K:L}
  \Ub \leq 1 + \lambda_1^{1 - 2 \gamma} \Ua.
\end{align}
Indeed, from (B2), Lemma \ref{lem:ineq:C} and \eqref{cond:sublin:U},
we obtain that
\begin{align*}
  (\Ub - 1) \nga{\bff}^2 - \Uc 
  &\leq \langle \bff, \cC D U (\bff) \rangle_{\gamma} 
  \leq \lambda_1^{1 - 2 \gamma} \nga{\bff} \ngad{D U (\bff)} 
  \leq  \lambda_1^{1 - 2 \gamma} \Ua \nga{\bff}^2 
    + \Uz \lambda_1^{1 - 2 \gamma} \nga{\bff} \\
  &\leq (\delta + \lambda_1^{1 - 2 \gamma} \Ua) \nga{\bff}^2 
    + \frac{(\Uz \lambda_1^{1 - 2 \gamma})^2}{4 \delta}, 
\end{align*}
for any $\delta > 0$, so that
\begin{align*}
  (\Ub - 1 - \lambda_1^{1 - 2 \gamma} \Ua - \delta) \nga{\bff}^2 
  \leq \Uc + \frac{(\Uz \lambda_1^{1 - 2 \gamma})^2}{4 \delta} 
\end{align*}
holds for any $\bff \in \H_{\gamma}$, and every $\delta > 0$,  
which implies \eqref{ineq:K:L}.
\end{itemize}

\end{Rmk}

This paper is concerned with sampling from probability distributions
on $\H$ that have a density with respect to $\mathcal{N}(0, \cC)$
which are of the form \eqref{eq:mu}.  In order that this is indeed the
case and furthermore to ensure the invariance of $\mu$ with respect to
the Markovian dynamics defined with respect to \eqref{eq:dynamics:xxx}, we
assume the following condition.
\begin{assumption}
\label{ass:integrability:cond}
Taking $\gamma \in [0,1/2)$ as in \cref{B123} we suppose that, for 
any $\varepsilon > 0$ there exists an $M = M(\varepsilon) \geq 0$, such that 
\begin{align*}
  U(\bff) \geq M - \varepsilon | \bff |_{\gamma}^2  \quad \text{ for any } \bff
  \in \H_\gamma .
\end{align*}
\end{assumption}

\begin{Rmk}
\label{rmk:Comp:beskos}
We notice that \cref{B123} $(B1)$ and \cref{ass:integrability:cond}
above are equivalent to conditions 3.2 and 3.3 imposed in
\cite{BePiSaSt2011}. Indeed such assumptions are applied there in
order to show the well-posedness of the dynamics in
\eqref{eq:dynamics:xxx} as well as to show that the measure $\mu$
defined in \eqref{eq:mu} is an invariant measure associated to
\eqref{eq:dynamics:xxx}. Such results are recalled in
\cref{prop:well:posed} and \cref{eq:invariance} below,
respectively. However, as pointed out in the introduction, condition
\cref{B123} $(B2)$ is further imposed in our setting in order to
obtain the Lyapunov structure \eqref{oview:Lyap}, which together with
the contractivity and smallness properties
\eqref{oview:contr}-\eqref{oview:small} allows us to obtain our main
convergence result, \cref{thm:weak:harris} below.
\end{Rmk}

\subsection{Well-Posedness of the Hamiltonian Dynamics}
\label{subsec:well:pos}

In the following proposition, we recall a well-posedness result of the
Hamiltonian dynamics in \eqref{eq:dynamics:xxx}, as shown in
\cite{BePiSaSt2011}. We consider the usual norm on the product space
$\H_{\gamma} \times \H_{\gamma}$ with the slight abuse of notation:
\begin{align}
  \label{def:nprod}
    \nga{(\bq,\bv)} := \nga{\bq} + \nga{\bv} \quad \mbox{ for all } 
  (\bq,\bv) \in \H_{\gamma} \times \H_{\gamma}.
\end{align}
\begin{Prop}
  \label{prop:well:posed}
  Suppose $\cC$ satisfies \cref{A123} and that $U$ maintains \cref{B123}, (B1).
  Let $\gamma \in [0,1/2)$ be as in \cref{B123}.
\begin{itemize}
\item[(i)] For any $(\bq_0, \bv_0) \in \H_\gamma \times \H_\gamma$, there exists
  a unique $(\bq,\bv) = (\bq(\bq_0, \bv_0), \bv(\bq_0, \bv_0))$ with
  \begin{align}
    (\bq,\bv) \in C^{1}(\RR; \H_\gamma \times \H_\gamma)
    \label{eq:sol:reg}
  \end{align}
  and obeying \eqref{eq:dynamics:xxx}.  The resulting solution
  operators $\{\Xi_t\}_{t \in \RR}$ defined via
   \begin{align*}
     \Xi_t(\bq_0,\bv_0) = \bq_t(\bq_0,\bv_0)
   \end{align*}
   are all continuous maps from $\H_\gamma \times \H_\gamma$ to $\H_\gamma$.
\item[(ii)] Under the additional restriction on $\cC$ of
  \cref{ass:higher:reg:C} and fixing an integration time $T > 0$
  the random variable
  \begin{align*}
    Q_1(\bq_0) = \bq_T(\bq_0, \bv_0), \quad \bv_0 \sim \mathcal{N}(0,\cC)
  \end{align*}
  is well defined in $\H_\gamma$ for any $\bq_0 \in \H_\gamma$.
  Moreover
\begin{align}
  P(\bq_0, A) := \Prb( Q_1(\bq_0) \in A)
  \label{eq:PEHMC:kernel:def}
\end{align}
defines a Feller Markov transition kernel on $\H_\gamma$.
\end{itemize}
\end{Prop}

\begin{proof}
  The first item follows from a standard Banach fixed point argument,
  i.e. it suffices to show that, given any
  $(\bq_0, \bv_0) \in \H_\gamma \times \H_\gamma$ and any
  $t_0 \in \RR$, the mapping
  \begin{align*}
    G(\bp,\bu)(t) := (\bq_0, \bv_0) 
         + \int_{t_0}^{t} (\bu(s), - \bp(s) - \cC DU(\bp(s))ds,
  \end{align*}
  is a contraction mapping on the space of continuous
  $(\H_\gamma \times \H_{\gamma})$-valued functions defined on
  $I := [t_0 -\delta,t_0 + \delta] \subset \RR$, that is on
  $C(I; \H_\gamma \times \H_\gamma)$, for some $\delta> 0$
  sufficiently small independent of $(\bq_0, \bv_0)$ and $t_0$.
  
  Observe that, with \eqref{lipschitz:new} and \eqref{eq:Poincare},
  \begin{align}
    |\cC^{1-\gamma} (DU(\bp) - DU(\tilde{\bp}))| 
    \leq \lambda_1^{1-2\gamma} \Ua |\cC^{-\gamma} (\bp - \tilde{\bp})| 
    \quad \mbox{ for all } \bp, \tilde{\bp} \in \H_{\gamma}.
    \label{eq:NLT:Lip:gam:est}
  \end{align}
  Thus, for any
  $(\bp,\bu), (\tilde{\bp},\tilde{\bu}) \in C(I; \H_\gamma \times
  \H_\gamma)$, using \eqref{def:nprod} and \eqref{eq:NLT:Lip:gam:est},
  \begin{align*}
    \sup_{t \in I}\nga{G(\bp,\bu)(t) - G(\tilde{\bp},\tilde{\bu})(t)}
       \leq \delta (1+\lambda_1^{1-2\gamma} \Ua )
    \sup_{t \in I} 
    \nga{(\bp,\bu) - (\tilde{\bp},\tilde{\bu})}.
  \end{align*}
  Therefore, $G$ is a contraction mapping on
  $C(I; \H_\gamma \times \H_\gamma)$ for
  $\delta < (1+ \lambda_1^{1-2\gamma} \Ua)^{-1}$.  Similar
  argumentation establishes the desired continuity of $\Xi_t$, thus
  completing the proof.
\end{proof}

\subsection{Formulation of Precondition Hamiltonian 
  Monte Carlo Chain}
\label{sec:HMC:def}
Having fixed an integration time $T > 0$, we denote by $Q_n(\bq_0)$ as
a random variable arising as the $n$ step dynamics of the exact
Preconditioned Hamiltonian Monte Carlo (PHMC) chain
\eqref{eq:PEHMC:kernel:def} starting from $\bq_0 \in \H$.  Namely, we
iteratively draw $Q_n(\bq_0) \sim P(Q_{n-1}(\bq_0), \cdot)$ for
$n \geq 1$ starting from $Q_0(\bq_0) = \bq_0$.  We can write
$Q_n(\bq_0)$ more explicitly as a transformation of the sequence of
Gaussian draws for the velocity as follows: Let $\H^{\otimes n}$
denote the product of $n$ copies of $\H$. Given a sequence
$\{\bv_0^{(j)}\}_{j \in \mathbb{N}}$ of i.i.d. draws from
$\mathcal{N}(0,\cC)$, we denote by $\npV_0^{(n)}$ the noise
path
\begin{align}
  \npV_0^{(n)} 
  := (\bv_0^{(1)},\ldots, \bv_0^{(n)}) 
  \sim \mathcal{N}(0,\cC)^{\otimes n},
  \label{eq:noise:path:notation}
\end{align}
where $\mathcal{N}(0,\cC)^{\otimes n}$ denotes the measure
on $\H^{\otimes n}$ given as the product of $n$ copies of
$\mathcal{N}(0,\cC)$. Taking $\mathcal{B}(\H)$ to be the Borel
$\sigma$-algebra on $\H$, we define
$Q_1(\bq_0): \H \to \H$ to be the Borel random variable
defined as
\begin{align*}
    Q_1(\bq_0)(\bv_0^{(1)}) = \bq_t(\bq_0, \bv_0^{(1)}) \quad 
  \mbox{ where } \bv_0^{(1)} \sim \mN(0, \cC).
\end{align*}
Iteratively, we define for every $n \geq 2$ the Borel random variable
$Q_n(\bq_0): \H^{\otimes n} \to \H$ given by
\begin{align}
    Q_n(\bq_0)(\npV_0^{(n)}) = \bq_t(Q_{n-1}(q_0)(\npV_0^{(n-1)}), \bv_0^{(n)}) 
  \quad \mbox{ where } \npV_0^{(n)} \sim \mN(0,\cC)^{\otimes n}.
  \label{eq:n:step:chain:Noise:p}
\end{align}
With these notations we can write the $n$-step iterated transition
kernels as
\begin{align}
  P^n(\bq_0,A) := \Prb(Q_n(\bq_0) \in A)
  \label{eq:n:step:PHMC:kernel}
\end{align}
for any $\bq_0 \in \H_\gamma$ and $A \in \mB(\H_\gamma)$. Or, equivalently, $P^n(\bq_0, \cdot)$ is the push-forward of $\mN(0,\cC)^{\otimes n}$ by the mapping $Q_n(\bq_0)$, i.e.
\begin{align}
    P^n(\bq_0, A) = Q_n(\bq_0)^*\mN(0,\cC)^{\otimes n}(A) = \mN(0,\cC)^{\otimes n}(Q_n(\bq_0)^{-1}(A)) 
    \label{def:Pn:Qn}
\end{align}
for every $\bq_0 \in \H_\gamma$ and $A \in \mB(\H_\gamma)$.

We recall an invariance result for \eqref{eq:mu} from
\cite{BePiSaSt2011} in our setting.
\begin{Prop}
  \label{eq:invariance}
  Under the conditions given in \cref{prop:well:posed} and additionally
  imposing \cref{ass:integrability:cond} we have that
  \begin{align*}
    \mathfrak{M}(d\bq, d\bv) \propto e^{-U(q)}\mu_0(d\bq) \times \mu_0(d\bv) 
  \end{align*}
  defines a probability measure on $\H_\gamma \times \H_\gamma$ which
  is invariant under $\{\Xi_t\}_{t \geq 0}$ namely
  \begin{align*}
    \int_{\H_\gamma \times \H_\gamma} f(\Xi_t(\bq,\bv)) \mathfrak{M}(d\bq, d\bv)
    = \int_{\H_\gamma \times \H_\gamma}f(\bq,\bv) \mathfrak{M}(d\bq, d\bv)
  \end{align*}
  holds for every $f \in C_b( \H_\gamma \times \H_\gamma)$
  and every $t \geq 0$.  As a consequence, $\mu$ given
  in \eqref{eq:mu} is a Borel probability measure on $\H_\gamma$ which is
  invariant for $P$ defined by \eqref{eq:PEHMC:kernel:def}.
\end{Prop}

\section{A Priori Bounds for the Deterministic Dynamics}
\label{sec:apriori}

This section provides various a priori bounds on the dynamics
specified by \eqref{eq:dynamics:xxx}. The proofs rely solely on the
bound on $D^2 U$ given in \eqref{Hess:bound}. In fact, they are
obtained by using inequalities \eqref{lipschitz:new} and
\eqref{cond:sublin:U}, that follow as a consequence of
\eqref{Hess:bound}.

\begin{Prop}
  \label{prop:basic:b:apriori}
  Impose \cref{A123} and \cref{B123}, (B1) and fix any $T \in \RR^+$
  satisfying
    \begin{align}
    \label{cond:t:apriori}
        T \leq (1 + \lambda_1^{1 - 2\gamma} \Ua)^{-1/2},
    \end{align}
    where the constant $\Ua$ is given in \eqref{lipschitz:new} and
    $\lambda_1$ is the top eigenvalue of $\cC$.  Then the dynamics
    defined by \eqref{eq:dynamics:xxx} maintains the bounds
    \begin{align}
    \label{sup:qs}
        \sup_{t\in [0, T]} \nga{ \bq_t(\bq_0,\bv_0)-(\bq_0+t \bv_0)} 
        \leq (1 + \lambda_1^{1 - 2\gamma} \Ua) 
             T^2 \max\{\nga{q_0},\nga{\bq_0 + T \bv_0} \} 
            + \lambda_1^{1 - 2\gamma} \Uz T^2
    \end{align}
    and
    \begin{multline}
    \label{sup:vs}
        \sup_{t\in [0, T]} \nga{\bv(t) - \bv_0} 
        \leq (1 + \lambda_1^{1 - 2\gamma} \Ua) 
             T [ 1 + (1 + \lambda_1^{1 - 2\gamma} \Ua)T^2] 
              \max \left\{ \nga{\bq_0}, \nga{\bq_0 + T \bv_0} \right\} \\
        + \lambda_1^{1 - 2\gamma} \Uz T 
          [ 1 + (1 + \lambda_1^{1 - 2\gamma} \Ua) T^2 ],
    \end{multline}
    for any $(\bq_0,\bv_0) \in \H_\gamma \times \H_\gamma$, with $\Uz$ as
    given in \eqref{cond:sublin:U}.
\end{Prop}
\begin{proof}
  Integrating the first equation in \eqref{eq:dynamics:xxx} twice and then applying the operator
  $\cC^{-\gamma}$, we obtain
  \begin{align}\label{int:dyn}
    \cC^{-\gamma} \bq_t 
    = \cC^{-\gamma} (\bq_0 + t \bv_0) 
      - \int_0^t \int_0^{s} \left[ \cC^{-\gamma} \bq_\tau 
       + \cC^{1 - \gamma} D U(\bq_\tau)  \right] d \tau ds, 
  \end{align}
  for each $t \in [0,T]$.  From Lemma \ref{lem:ineq:C} and inequality
  \eqref{cond:sublin:U}, we obtain
  \begin{align}
    \nga{\bq_t - (\bq_0 + t \bv_0)}
    \leq& (1 + \lambda_1^{1 - 2\gamma} \Ua) 
      \int_0^t \int_0^{s} \nga{\bq_{\tau}} d \tau d s 
      + \lambda_1^{1 - 2\gamma} \Uz \frac{T^2}{2} 
                \notag\\
    \leq& (1 + \lambda_1^{1 - 2\gamma} \Ua) 
      \int_0^t \int_0^{s} \nga{\bq_\tau - (\bq_0 + \tau \bv_0)} d \tau d s 
          \notag\\
       &\quad \qquad + (1 + \lambda_1^{1 - 2\gamma} \Ua) 
         \int_0^t \int_0^{s} \nga{\bq_0 + \tau \bv_0} d \tau d s 
         + \lambda_1^{1 - 2\gamma} \Uz \frac{T^2}{2} 
              \notag\\
    \leq& (1 + \lambda_1^{1 - 2\gamma} \Ua) \frac{T^2}{2} 
      \sup_{\tau \in [0,T]} \nga{\bq_\tau - (\bq_0 + \tau \bv_0)} 
          \notag\\
      &\quad \qquad + (1 + \lambda_1^{1 - 2\gamma} \Ua) \frac{T^2}{2} 
              \max \{ \nga{\bq_0}, \nga{\bq_0 + T \bv_0}\} 
        + \lambda_1^{1 - 2\gamma} \Uz \frac{T^2}{2}.
        \label{bound:qs}
  \end{align}
  Here note that, using the convexity of the function
  $f(\tau) = \nga{\bq_0 + \tau \bv_0}$, we have
  \begin{align}
    \sup_{\tau \in [0,T]} \nga{\bq_0 + \tau \bv_0}
    \leq \max \{ \nga{\bq_0}, \nga{\bq_0 + T \bv_0}  \}
            \label{eq:stupid:convex:bnd}
  \end{align}
  which we used in the final bound in \eqref{bound:qs}.
  Thus, using assumption \eqref{cond:t:apriori} and taking the
  supremum with respect to $t \in [0,T]$ in \eqref{bound:qs}, we
  conclude the first bound \eqref{sup:qs}.

  Turn next to second bound \eqref{sup:vs}, integrating the second equation in
  \eqref{eq:dynamics:xxx} once and using \cref{lem:ineq:C} and inequality
  \eqref{cond:sublin:U} again, we have
  \begin{align}
  \label{sup:vs:a} 
  \nga{\bv_t - \bv_0} 
    \leq (1 + \lambda_1^{1 - 2\gamma} \Ua) \int_0^t \nga{ \bq_s} d s 
          + \lambda_1^{1 - 2\gamma} \Uz t 
    \leq (1 + \lambda_1^{1 - 2\gamma} \Ua) T \sup_{s \in [0,T]} \nga{\bq_\tau} 
    + \lambda_1^{1 - 2\gamma} \Uz T
  \end{align}
  for every $t \in [0,T]$.  From \eqref{sup:qs}, it follows that
  \begin{align}
  \label{sup:qs:b} 
    \sup_{t\in [0, T]} 
     \nga{\bq_s} \leq [1 + (1 + \lambda_1^{1 - 2\gamma} \Ua) T^2]
      \max\{\nga{\bq_0},\nga{\bq_0 + T \bv_0} \} + \lambda_1^{1 - 2\gamma} \Uz T^2.  
  \end{align}
  Hence, we conclude \eqref{sup:vs} from \eqref{sup:vs:a} and
  \eqref{sup:qs:b}, completing the proof.
\end{proof}

\begin{Prop}\label{lem:contr:1}
  Impose \cref{A123}, \cref{B123}, (B1) and consider any $T \in \RR^+$
  satisfying
  \begin{align}
  \label{cond:t:contr:1}
        T \leq (1 + \lambda_1^{1 - 2 \gamma} \Ua)^{-1/2},
  \end{align}
  where $\Ua$ is as in \eqref{Hess:bound} and $\lambda_1$ is the top
  eigenvalue of $\cC$.  Then, for any
  $(\bq_0,\bv_0), (\tilde \bq_0, \tilde \bv_0) \in \H_\gamma \times
  \H_\gamma$,
  \begin{multline}	
      \sup_{t\in[0, T]} 
        \nga{\bq_t(\bq_0,\bv_0) - \bq_t(\tilde \bq_0, \tilde \bv_0)
                                -(\bq_0-\tilde \bq_0)-t(\bv_0-\tilde \bv_0)} \\
        \leq (1+ \lambda_1^{1 - 2 \gamma} \Ua) T^2 \max\left\{\nga{\bq_0-\tilde \bq_0},
               \nga{(\bq_0-\tilde \bq_0) + T(\bv_0-\tilde \bv_0)}\right\}.
       \label{eq:Hgamma:lip:bnd}
   \end{multline}
\end{Prop}

\begin{Rmk}
  \label{rmk:taking:stock:lip_1}
  Observe that, given any $\bq_0, \btq_0, \bv_0 \in \H_\gamma$, by choosing
  \begin{align}
    \tilde{\bv}_0 := \bv_0 + \frac{1}{T} (\bq_0 - \tilde{\bq}_0),
    \label{eq:naive:nudge:IC}
  \end{align}
  then under \eqref{eq:Hgamma:lip:bnd} we obtain 
  \begin{align}
    |\cC^{-\gamma} \left[ \bq_T(\bq_0,\bv_0)-\bq_T(\tilde \bq_0, \tilde \bv_0)\right]| 
     \leq (1+ \lambda_1^{1 - 2 \gamma} \Ua) T^2 |\cC^{-\gamma}(\bq_0-\tilde \bq_0)|,
    \label{eq:full:monte:lip:bnd}
  \end{align}
  which thus yields a contraction when
  $T < (1+ \lambda_1^{1 - 2 \gamma} \Ua)^{-1/2}$.  This observation
  for the initial conditions in \eqref{eq:naive:nudge:IC} has
  previously been employed in \cite{BoEbZi2018} and, in the finite
  dimensional case where $\H = \RR^k$ for some $k \in \NN$, this bound
  can be used directly as a crucial step towards establishing the
  $\rho$-smallness and $\rho$-contraction conditions for the weak
  Harris theorem in \cite{hairer2011asymptotic}, as we illustrate
  below in \cref{sec:finite:dim:HMC}.
  
  The idea behind definition \eqref{eq:naive:nudge:IC} comes from the
  fact that for the simplified version of the dynamics in
  \eqref{eq:dynamics:xxx} where $d\bv_t/dt = 0$, the positions of two
  associated trajectories starting from $(\bq_0, \bv_0)$ and
  $(\btq_0, \tilde \bv_0)$, with $\tilde \bv_0$ as in
  \eqref{eq:naive:nudge:IC}, will coincide at time $T$. With a similar
  line of reasoning, one could consider a slighly better
  approximation of the dynamics in \eqref{eq:dynamics:xxx} by assuming
  instead $U = 0$, in which case the associated dynamics
  $d\bq_t/dt = \bv_t$, $d\bv_t/dt = \bq_t$ describes the motion of a
  simple pendulum. Here by defining
  $\tilde \bv_0 = \bv_0 + (\bq_0 - \btq_0) (\cos T/\sin T) $ one again
  concludes that the positions of two trajectories starting from
  $(\bq_0, \bv_0)$ and $(\btq_0, \tilde \bv_0)$ coincide after time
  $T$. While we could obtain similar results by using the latter
  approach, this would require the same type of assumptions we already
  impose in the first case, thus not showing a significant difference
  at least at the theoretical level. For simplicity, we then chose the
  first approach for our presentation. We remark however that the
  second approach, as being associated to a better approximation of
  \eqref{eq:dynamics:xxx}, could lead to slightly less stringent
  constants on the conditions for the integration time $T$ in
  comparison to \eqref{cond:t:contr:1}.
  
  More generally, we may view \eqref{eq:naive:nudge:IC} as addressing
  a control problem.  In fact, the methodology of the weak Harris
  theorem developed here could in principle allow the use of a wide
  variety of controls. More specifically, we are interested in any
  `reasonable' mapping
  $\Psi: \H_\gamma \times \H_\gamma \times \H_\gamma \to \H_\gamma$
  such that, for any $\bq_0, \btq_0, \bv_0 \in \H_\gamma$ and any
  suitable value of $T>0$, one would have
  \begin{align*}
    \bq_T(\bq_0,\bv_0) \approx \bq_T(\tilde \bq_0, \Psi(\bq_0, \btq_0,\bv_0 )).
  \end{align*}
  In this connection one might hope to make a more delicate use of the
  Hamiltonian dynamics, presumably tailored to the fine properties of
  a particular potential $U$ of interest, to obtain refined results on
  convergence to equilibrium. In particular, we expect that the
  constraints imposed on $T$ by \cref{prop:basic:b:apriori} are
  overzealous, and could potentially be improved by a different type
  of control.
 
  On the other hand, in the infinite dimensional Hilbert space setting
  which we are primarily focused on here, even
  \eqref{eq:naive:nudge:IC} is insufficient for the aim of
  establishing contractivity in the Markovian dynamics, as the law of
  this choice of $\tilde{\bv}_0$ is not generically absolutely
  continuous with respect to the law of $\bv_0$; cf.
  \cref{prop:local_contr} and \cref{prop:smallness} below.  We proceed
  instead by using the refinement \eqref{eq:kinetic_coupling_idea}
  which is shown to produce a contraction in
  \cref{prop:FP:Type:contract}.  Here we are making use of some of the
  intuition and approach to ergodicity in the stochastic fluids
  literature, cf. \cite{mattingly2002exponential,
    kuksin2012mathematics, GlattHoltzMattinglyRichards2016}.  In these
  works one modifies the noise path on low modes with the expectation
  that if one induces a contraction on the large scale dynamics for
  sufficiently many low frequency modes then the high frequencies (or
  small scales) will also contract, being enslaved to the behavior of
  the system at large scales.  This effect, sometimes referred as a
  Foias-Prodi bound \cite{foias1967comportement}, is widely observed
  in the fluids and infinite dimensional dynamical systems
  literature.
\end{Rmk}
\begin{proof}[Proof of \cref{lem:contr:1}]
  Let $\bz_t = \bq_t(\bq_0,\bv_0) - \bq_t(\widetilde{\bq}_0, \widetilde{\bv}_0)$ 
  and $\bw_t = d\bz_t/ dt$. Then, for any $t > 0$,  $\bz_t$ satisfies
  \begin{align}
  \label{d2:zt}
    \frac{d^2 \bz_t}{dt^2} = - \bz_t - \cC g(t)
  \end{align}
  where
  \begin{align}
  \label{def:g}
    g(t) := D U(\bq_t(\bq_0,\bv_0)) 
    - DU(\bq_t(\widetilde{\bq}_0, \widetilde{\bv}_0)).
  \end{align}
  Therefore, for every $t \geq 0$,
  \begin{align*}
    \cC^{- \gamma} \bz_t 
    = \cC^{-\gamma}(\bz_0 + t \bw_0) 
    - \int_0^t \int_0^{s} [\cC^{-\gamma} \bz_\tau + \cC^{1 - \gamma} g(\tau)] d \tau ds.
  \end{align*}
  By using \cref{lem:ineq:C} and inequality \eqref{lipschitz:new}, we obtain
  \begin{align*}
    \nga{ \bz_t - (\bz_0 + t \bw_0)} 
    &\leq \int_0^t \int_0^{s} \left[ \nga{\bz_\tau} 
         + \lambda_1^{1 - 2 \gamma}\ngad{ g(\tau)} \right] d \tau d s \\
    &\leq (1 + \lambda_1^{1 - 2 \gamma} \Ua) \int_0^t \int_0^{s} \nga{ \bz_\tau} d \tau ds.
  \end{align*}
  The remaining portion of the proof follows analogously as in the
  proof of \eqref{sup:qs}.
\end{proof}

In view of \cref{rmk:taking:stock:lip_1} the bounds in \cref{lem:contr:1}
are not sufficient for our application to prove the
$\rho$-contractivity and $\rho$-smallness conditions for the weak
Harris theorem below in \cref{sec:Ptwise:contract:MD}. For this
purpose we consider a modified version of \eqref{eq:naive:nudge:IC}
where the shift only involves a low-modes finite-dimensional approximation of $\bq_0 - \tilde{\bq}_0$.

Before proceeding let us introduce some notation. Split $\H$ into a
space $\H_N := \operatorname{span}\{\be_1, \cdots, \be_N\}$ and its
orthogonal complement $\H^N$; so that $\H = \H_N \oplus \H^N$ where
$N$ satisfies the second condition in \eqref{cond:t:N:contr}, below.
Recall, as in \eqref{eq:get:high:low}, that, given $\bbf \in \H$, we
denote by $\Pi_N \bbf$ and $\Pi^N \bbf$ the orthogonal projections
onto $\H_N$ and $\H^N$, respectively.  This splitting is defined such
that the Lipschitz constant of the projection of $- \cC D U(\bbf)$
onto $\H^N$ is at most $1/4$. 

For any $\gamma \in [0,1/2)$ and $\alpha \in \RR^+$, we consider the
following auxiliary norm:
\begin{align}
 |\bbf|_{\gamma, \alpha} := |\Pi_N \bbf|_{\gamma} + \alpha | \Pi^N \bbf|_{\gamma}, 
    \; \text{for any } \bbf \in \H_{\gamma}.
      \label{eq:wn:gamma:alpha}
\end{align}
\begin{Rmk}
     \label{rmk:equiv:norms}
Notice that $\normga{\,\cdot\,}$ is equivalent to $\nga{\,\cdot\,}$ and 
\begin{align}
  \label{equiv:norms}
  \min \{1,\alpha\} \nga{\bbf} \leq \normga{\bbf} 
  \leq \sqrt{2} \max \{1, \alpha\} \nga{\bbf},
  \; \text{ for all } \bbf \in \H_{\gamma}.
\end{align}
In particular, for $\alpha$ defined as in \eqref{eq:alpha:lip:def} below, we have
\begin{align}
  \label{equiv:norms:b}
  \nga{\bbf} \leq \normga{\bbf} \leq \sqrt{2} \alpha \nga{\bbf},
  \; \text{ for all } \bbf \in \H_{\gamma}.
\end{align}
\end{Rmk}

\begin{Prop}
  \label{prop:FP:Type:contract}
  Impose \cref{A123}, \cref{B123}, (B1). Let
  $(\bq_0,\bv_0), (\tilde \bq_0, \tilde \bv_0) \in \H_\gamma \times
  \H_\gamma$ such that
  \begin{align} 
      \label{eq:kinetic_coupling_idea}
      \Pi^N \tilde \bv_0 =  \Pi^N \bv_0 
      \quad \text{and} \quad  
      \Pi_N \tilde \bv_0 
       = \Pi_N \bv_0 + T^{-1} (\Pi_N \bq_0 - \Pi_N \tilde \bq_0).
  \end{align}
  Assume that $T \in \RR^+$ and $N \in \mathbb{N}$ satisfy
  \begin{align} 
    \label{cond:t:N:contr}
        T \leq \frac{1}{[2(1 + \lambda_1^{1 -2 \gamma} \Ua)]^{1/2}} 
        \quad \mbox{and} \quad 
        \lambda_{N+1}^{1 - 2 \gamma} \leq \frac{1}{4 \Ua},
  \end{align}
  and let 
  \begin{align}
    \label{eq:alpha:lip:def}
        \alpha = 4(1 + \lambda_1^{1 - 2 \gamma} \Ua).
  \end{align}
  Here $\gamma$ is specified in \cref{B123}, $\Ua$ is as in
  \eqref{Hess:bound} and $\lambda_j$ represent the eigenvalues of
  $\cC$ in descending order as in \cref{A123}. Then,
  \begin{align}
        \normga{\bq_T(\bq_0,\bv_0) - \bq_T(\tilde \bq_0, \tilde \bv_0)} 
        \leq \sma
                \normga{\bq_0 - \tilde \bq_0},
      \label{eq:FP:type:cont:Multi}
    \end{align}
    where $|\cdot|_{\gamma, \alpha}$ is the norm defined in \eqref{eq:wn:gamma:alpha} and
    \begin{align*}
        \sma = 1 - \frac{T^2}{12}.
    \end{align*}
\end{Prop}
\begin{proof}
  As in the proof of \cref{lem:contr:1}, let us denote
  $\bz_t := \bq_t(\bq_0,\bv_0) - \bq_t(\tilde \bq_0, \tilde \bv_0)$ and
  $\bw_t = d\bz_t/ dt$, for all $t \geq 0$. Notice that
  \begin{align}
   \label{eq:coupling}
    \Pi_N \bz_0 + T \Pi_N \bw_0 = 0 \quad \mbox{and} \quad \Pi^N \bw_0 = 0.
  \end{align}
  Applying $\cC^{-\gamma}$ to \eqref{d2:zt}, projecting onto $\H_N$
  and integrating, yields
  \begin{align*}
    \cC^{-\gamma} \Pi_N \bz_T 
    =  -\int_0^T \int_0^s 
       \left[ \cC^{-\gamma} \Pi_N \bz_\tau 
              + \cC^{1-\gamma} \Pi_N g(\tau) \right] d \tau ds,
  \end{align*}
  with $g(\cdot)$ defined as in \eqref{def:g}. Thus, using
  \eqref{eq:Poincare} in \cref{lem:ineq:C} and \eqref{Hess:bound} of
  \cref{B123}, we estimate
  \begin{align}
    \nga{ \Pi_N \bz_T} \leq&
       \int_0^T \! \int_0^s \left[ \nga{ \bz_\tau} 
           + \lambda_1^{1 - 2 \gamma} \ngad{g(\tau) } \right] d \tau ds 
    \leq  (1 + \lambda_1^{1 - 2 \gamma} \Ua) 
          \frac{T^2}{2} \sup_{s \in [0,T]} \nga{ \bz_s}
          \notag\\
    =& \frac{\alpha T^2}{8} \sup_{s \in [0,T]} \nga{ \bz_s}.
      \label{est:ztl}
   \end{align}
  
  On the other hand, by Duhamel's formula, we have
  \begin{align*}
    \bz_T = \bz_0 \cos(T) + \bw_0 \sin(T) 
           - \int_0^T \sin(T - s) \, \cC g(s) ds,
  \end{align*}
  and hence, with \eqref{eq:coupling},
  \begin{align*}
    \cC^{-\gamma} \Pi^N \bz_T 
    = \cC^{-\gamma} \Pi^N \bz_0 \cos(T) - 
        \int_0^T \sin(T - s) \, \cC^{1 - \gamma} \Pi^N g(s) ds
  \end{align*}
  Now, using $(ii)$ of \cref{lem:ineq:C} and $(B1)$ of \cref{B123}, we
  estimate
  \begin{align*}
    \nga{\Pi^N \bz_T}
    \leq& \nga{ \Pi^N \bz_0} \cos(T)
          + \lambda_{N+1}^{1 - 2 \gamma} \Ua
               \int_0^T \sin(T - s) \nga{ \bz_s} ds \\
    \leq& \nga{\Pi^N \bz_0} \cos(T)
            + \frac{1 - \cos(T)}{4} \sup_{s \in [0,T]} \nga{\bz_s}.
  \end{align*}
  where for the final inequality we used the second condition in
  \eqref{cond:t:N:contr}.  Therefore, using that
  $\cos(s) \leq 1 - s^2/2 + s^4/24$ and $1 - \cos(s) \leq s^2/2$ for
  every $s \in \RR$, yields
  \begin{align}
    \label{est:zth}
    \nga{\Pi^N \bz_T}
    \leq \left( 1 - \frac{T^2}{2}
          + \frac{T^4}{24} \right)\nga{\Pi^N \bz_0}
          + \frac{T^2}{8} \sup_{s \in [0,T]} \nga{ \bz_s}.
  \end{align}
  
  From \cref{lem:contr:1} and a bound as in
  \eqref{eq:stupid:convex:bnd} it follows that
  \begin{align*}
  \sup_{s \in [0,T]} \nga{ \bz_s }
      \leq [1 + (1 + \lambda_1^{1 -2 \gamma} \Ua) T^2]
         \max \left\{ \nga{ \bz_0}, \nga{\bz_0 + T \bw_0} \right\}.
  \end{align*}
  However from \eqref{eq:coupling} we have
  $\bz_0 + T \bw_0 = \Pi^N \bz_0$, so that
  $\max \{ \nga{ \bz_0}, \nga{\bz_0 + T \bw_0}\} = \nga{ \bz_0}$. With
  this and the first condition in \eqref{cond:t:N:contr}, we therefore
  obtain
  \begin{align}
    \sup_{s \in [0,T]} \nga{ \bz_s }
    \leq [1 + (1 + \lambda_1^{1 -2 \gamma} \Ua)T^2] \nga{ \bz_0}
    \leq \frac{3}{2} \nga{ \bz_0}.
    \label{unif:bound:z}
  \end{align}
  Using \eqref{unif:bound:z} in \eqref{est:ztl} and in
  \eqref{est:zth}, we obtain
  \begin{align*}
    \nga{\Pi_N \bz_T} \leq \frac{3 \alpha T^2}{16} \nga{\bz_0}
  \end{align*}
  and
  \begin{align*}    
    \nga{\Pi^N \bz_T}
       \leq \left( 1 - \frac{T^2}{2} + \frac{T^4}{24} \right)\nga{\Pi^N \bz_0}
              + \frac{3 T^2}{16} \nga{\bz_0},
 \end{align*}
  so that finally
  \begin{align}
    \normga{\bz_T}
        =& \nga{ \Pi_N \bz_T} + \alpha \nga{\Pi^N \bz_T}
       \leq \frac{3\alpha T^2}{8} \nga{ \bz_0}
            +  \alpha \left( 1 - \frac{T^2}{2}
           + \frac{T^4}{24} \right) \nga{ \Pi^N \bz_0}
    \notag\\
    \leq& \frac{3 \alpha T^2}{8} \nga{\Pi_N \bz_0}
      + \alpha \left( 1 - \frac{T^2}{8}
                + \frac{T^4}{24} \right) \nga{ \Pi^N \bz_0}.
    \label{est:normga:zt}
  \end{align}
  From the first condition in \eqref{cond:t:N:contr} and the
  definition of $\alpha$ in \eqref{eq:alpha:lip:def}, it follows in
  particular that $\alpha T^2 \leq 2$ and also $T \leq 1$, so that
  $T^4 \leq T^2$. Therefore, from \eqref{est:normga:zt}, we have
  \begin{align*}
    \normga{\bz_T} \leq
    \frac{3}{4} \nga{\Pi_N \bz_0}
    + \alpha \left( 1 - \frac{T^2}{12} \right) \nga{\Pi^N \bz_0}
      \leq \max \left\{ 1 - \frac{T^2}{12}, \frac{3}{4} \right\} \normga{\bz_0}  
    = \left( 1 - \frac{T^2}{12} \right) \normga{\bz_0},
  \end{align*}
  where the equality above follows again from the fact that
  $T \leq 1$, by the first condition in \eqref{cond:t:N:contr}. This
  completes the proof.
\end{proof}

\section{Foster-Lyapunov Structure}
\label{sec:f:l:bnd}

This section provides the details of the Foster-Lyapunov structure for
the Markov kernel $P$ defined by \eqref{eq:PEHMC:kernel:def} under
\cref{ass:higher:reg:C}, \cref{B123}. First, we recall the underlying
definition:

\begin{Def}\label{def:Lyap}
  We say that $V: \H_{\gamma} \to \mathbb{R}^+$ is a
  \emph{Foster-Lyapunov} (or, simply, a \emph{Lyapunov}) function for
  the Markov kernel $P$ if $V$ is integrable with respect to
  $P^n(\bq, \cdot)$ for every $\bq \in \H$ and $n \in \mathbb{N}$, and
  satisfies the following inequality
  \begin{align}\label{ineq:Lyap}
    P^nV(\bq)  \leq \kappa_V^n V(\bq) +K_V \quad 
    \mbox{ for all } \bq \in \H \mbox{ and } n \in \mathbb{N},
  \end{align}
  for some $\kappa_V \in [0,1)$ and $K_V > 0$. 
\end{Def}

With this definition in hand the main result of this section is as follows:
\begin{Prop}\label{prop:FL}
  Impose \cref{A123}, \cref{ass:higher:reg:C} and \cref{B123} and
  suppose that $T \in \RR^+$ satisfies
  \begin{align}
    \label{cond:t:Lyap}
    T \leq  \min \left\{ 
         \frac{1}{[2(1 + \lambda_1^{1 - 2\gamma} \Ua)]^{1/2}}, 
    \frac{\Ub^{1/2}}{2\sqrt{6} (1 + \lambda_1^{1 - 2\gamma} \Ua)} \right\},
  \end{align}
  where $\Ua$ and $\Ub$ are defined as in \eqref{Hess:bound},
  \eqref{dissip:new}, respectively and $\lambda_1$ is the largest
  eigenvalue of $\cC$. Then, the functions
  \begin{align}\label{eq:FL:1}
        V_{1,i}(\bq) = \nga{\bq}^i, \quad i \in \mathbb{N},
  \end{align}
  and 
  \begin{align}\label{eq:FL:2}
        V_{2, \eta}(\bq) = \exp(\eta \nga{\bq}^2),
  \end{align}
with $\eta \in \mathbb{R}^+$ satisfying
 \begin{align}
    \eta < \left[ c \Tr(\cC^{1 - 2 \gamma}) 
    \left( \Ub^{-1} +  T^2 \right) \right]^{-1},
    \label{cond:eta}
\end{align}
for a suitable absolute constant $c \in \mathbb{R}^+$, are Lyapunov
functions for the Markov kernel $P$ defined in
\eqref{eq:PEHMC:kernel:def}.
\end{Prop}
\begin{proof}
  We start by showing that $V_{1,2}(\bq) = \nga{\bq}^2$ is a Lyapunov
  function for $P$. First, notice
  $\frac{d}{dt}\nga{\bq_t}^2 = 2 \langle \bq_t, \bv_t
  \rangle_{\gamma}$ so that
  \begin{align}
    \label{est:norm:q}
    \nga{\bq_T}^2 = \nga{\bq_0}^2 
            + 2 \int_0^T \langle \bq_s, \bv_s \rangle_{\gamma} ds.
  \end{align}
  Moreover, from \eqref{eq:dynamics:xxx}
  \begin{align}\label{eq:der:q:v}
    \frac{d}{ds} \langle \bq_s, \bv_s \rangle_{\gamma} 
    = \nga{\bv_s}^2 - \nga{\bq_s}^2 
       - \langle \bq_s, \cC \nabla U (\bq_s) \rangle_{\gamma}.
  \end{align}
  Hence, using \cref{B123}, (B2),
  \begin{align}\label{est:q:v}
    \langle \bq_s, \bv_s \rangle_{\gamma} 
    &= \langle \bq_0, \bv_0 \rangle_{\gamma} 
      + \int_0^s \left[ \nga{\bv_\tau}^2 - \nga{\bq_\tau}^2 
        - \langle \bq_\tau, \cC \nabla U (\bq_\tau) \rangle_{\gamma} 
          \right] d\tau 
         \nonumber\\
    &\leq \langle \bq_0, \bv_0 \rangle_{\gamma} 
     + \int_0^s \left[ \nga{\bv_\tau}^2 - \Ub \nga{\bq_\tau}^2 
               + \Uc \right] d\tau,
  \end{align}
  for any $s \geq 0$.  Using \eqref{est:q:v} in \eqref{est:norm:q}, we
  obtain
  \begin{align}
  \label{bound:qt}
    \nga{\bq_T}^2 
      \leq \nga{\bq_0}^2 
        + 2T \langle \bq_0, \bv_0 \rangle_{\gamma} 
        + 2 \int_0^T \int_0^s \left[ \nga{\bv_\tau}^2 
             - \Ub \nga{\bq_\tau}^2 + \Uc \right] d\tau ds.
  \end{align}
	
  From \cref{prop:basic:b:apriori}, \eqref{sup:vs} and hypothesis
  \eqref{cond:t:Lyap}, it follows that
  \begin{align*}
    \nga{\bv_\tau} 
       \leq \frac{7}{4}\nga{\bv_0} 
         + \frac{3}{2} (1 + \lambda_1^{1 - 2\gamma} \Ua) \tau \nga{\bq_0} 
         + \frac{3}{2} \lambda_1^{1 - 2 \gamma} \Uz \tau,
  \end{align*}
  so that
  \begin{align}
    \label{bound:vtau}
    \nga{\bv_\tau}^2 
      \leq \frac{49}{8} \nga{\bv_0}^2 
           + 9(1 + \lambda_1^{1 - 2\gamma} \Ua)^2 \tau^2 \nga{\bq_0}^2 
           + 9 (\lambda_1^{1 - 2\gamma} \Uz)^2 \tau^2,
  \end{align}
  which holds for any $\tau \geq 0$.
  Moreover, from \eqref{sup:qs} and using hypothesis
  \eqref{cond:t:Lyap} again, we obtain that
  \begin{align*}
     \nga{\bq_\tau - (\bq_0 + \tau \bv_0)} 
    \leq \frac{\nga{\bq_0}}{2} + \frac{\tau}{2} \nga{\bv_0} 
         + \lambda_1^{1 - 2\gamma} \Uz \tau^2,
  \end{align*}
  so that
  \begin{align*}
    \nga{\bq_\tau} 
    \geq \frac{\nga{\bq_0}}{2} 
         - \frac{3}{2} \tau \nga{\bv_0} 
         - \lambda_1^{1 - 2\gamma} \Uz \tau^2
  \end{align*}
  and, consequently,
  \begin{align*}
    2 \nga{\bq_\tau}^2  
    \geq \frac{\nga{\bq_0}^2}{4} 
        -  9 \tau^2 \nga{\bv_0}^2 
        - 4 (\lambda_1^{1 - 2\gamma} \Uz)^2 \tau^4.
  \end{align*}
  Thus, from \eqref{ineq:K:L} and \eqref{cond:t:Lyap}, it follows that
  \begin{align}
    - 2 \Ub \nga{\bq_\tau}^2  
    &\leq -\frac{\Ub}{4} \nga{\bq_0}^2 + 9 \Ub \tau^2 \nga{\bv_0}^2 
        + 4 \Ub (\lambda_1^{1 - 2\gamma} \Uz)^2 \tau^4 
      \nonumber \\
    &\leq -\frac{\Ub}{4} \nga{\bq_0}^2
         + 9 (1 + \lambda_1^{1- 2 \gamma} \Ua)  \tau^2 \nga{\bv_0}^2 
         + 4 (1 + \lambda_1^{1- 2 \gamma} \Ua) (\lambda_1^{1 - 2\gamma} \Uz)^2 \tau^4 
      \nonumber \\
    & \leq -\frac{\Ub}{4} \nga{\bq_0}^2 + \frac{9}{2} \nga{\bv_0}^2 
      + 2  (\lambda_1^{1 - 2\gamma} \Uz)^2 \tau^2,
      \label{bound:qtau}
  \end{align}
  for any $\tau \geq 0$. Using \eqref{bound:vtau} and
  \eqref{bound:qtau} in \eqref{bound:qt}, yields
  \begin{align}
    \label{bound:qt:2}
    \nga{\bq_T}^2 
       \leq& \left( 1 + \frac{3}{2} (1 + \lambda_1^{1 - 2\gamma} \Ua)^2 T^4 
            - \frac{\Ub}{8} T^2 \right)\nga{\bq_0}^2 
         \notag\\
           &\qquad + 2T \langle \bq_0, \bv_0 \rangle_{\gamma}
             + \frac{67}{8} T^2 \nga{\bv_0}^2 
             + \frac{5}{3} (\lambda_1^{1 - 2\gamma} \Uz)^2 T^4 + \Uc T^2.
  \end{align}
  By hypothesis \eqref{cond:t:Lyap}, we have that
  $3(1 + \lambda_1^{1 - 2\gamma} \Ua)^2 T^4/2 \leq \Ub T^2/16$. Thus,
  \begin{align}
	\label{est:term:qt}
	     1 + \frac{3}{2} (1 + \lambda_1^{1 - 2\gamma} \Ua)^2 T^4 - \frac{\Ub}{8} T^2 
	     \leq 1 - \frac{\Ub}{16}T^2
	     \leq e^{- \frac{\Ub T^2}{16}},
  \end{align}
  where we used the fact that $1 - x \leq e^{-x}$, for every
  $x \geq 0$.  Using \eqref{est:term:qt} in \eqref{bound:qt:2} and
  taking expected values on both sides of the resulting inequality,
  and noting that, by symmetry
  $\E\langle \bq_0, \bv_0 \rangle_{\gamma} = 0$ we obtain
  \begin{align}
	\label{quad:Lyap}
	P V_{1,2}(\bq_0) = \bE \nga{\bq_T}^2 
	\leq  e^{- \frac{\Ub T^2}{16}} \nga{\bq_0}^2 
	+ \left( \frac{67}{8} \Tr(\cC^{1 - 2 \gamma}) 
        + \frac{5}{3} (\lambda_1^{1 - 2\gamma} \Uz)^2 T^2 + \Uc \right)T^2.
  \end{align}
  Hence, after iterating on the result in \eqref{quad:Lyap} $n$ times,
  we have
  \begin{align}
    P^n V_{1,2}(\bq_0) =& \bE \nga{Q_n(\bq_0)}^2 
    \notag\\
       \leq&  e^{- \frac{n \Ub T^2}{16}} \nga{\bq_0}^2 
           + \left( \frac{67}{8} \Tr(\cC^{1 - 2 \gamma}) 
           + \frac{5}{3} (\lambda_1^{1 - 2\gamma} \Uz)^2 T^2 + \Uc \right) T^2  
                   \sum_{j=0}^{n-1}  e^{- \frac{j\Ub t^2}{16}}.
                   \label{quad:Lyap:n}
  \end{align}
  Notice that
  \begin{align*}
    T^2  \sum_{j=0}^{n-1}  e^{- \frac{j\Ub T^2}{16}} 
      \leq \frac{T^2}{1 - e^{- \frac{\Ub t^2}{16}}} 
      \leq \frac{48}{\Ub},
  \end{align*}
  where in the last inequality we used that
  $x/(1 - e^{-x}) \leq e \leq 3$, for every $0 \leq x \leq 1$. Thus,
 \begin{align*}
   P^n V_{1,2}(\bq_0)
   \leq  e^{- \frac{n \Ub T^2}{16}} \nga{\bq_0}^2 
      + \left( \frac{67}{8} \Tr(\cC^{1 - 2 \gamma}) 
             + \frac{5}{3} (\lambda_1^{1 - 2\gamma} \Uz)^2 T^2 
             + \Uc \right) \frac{48}{\Ub},
 \end{align*}
 which shows \eqref{def:Lyap} for $V_{1,2}$. 
	
 We turn now to establish \eqref{def:Lyap} in the general case of
 $V_{1,i}$, for any $i \in \mathbb{N}$.  Here, invoking Young's
 inequality to estimate the term
 $ 2T \langle \bq_0, \bv_0 \rangle_\gamma$ in \eqref{bound:qt:2} as
 \begin{align*}
   2T \langle \bq_0, \bv_0 \rangle_{\gamma} 
   \leq \frac{\Ub T^2}{32} \nga{\bq_0}^2 + \frac{32}{\Ub} \nga{\bv_0}^2,
 \end{align*}
 and using again that $3(1 + \lambda_1^{1 - 2\gamma} \Ua)^2 T^4/2 \leq \Ub T^2/16$, 
 it follows from \eqref{bound:qt:2} that
 \begin{align}
 \label{bound:qt:3}
   \nga{\bq_T}^2 \leq 
    \left( 1 - \frac{\Ub T^2}{32} \right)\nga{\bq_0}^2 
     + \left( \frac{67}{8} T^2 + \frac{32}{\Ub} \right) \nga{\bv_0}^2 
   + \frac{5}{3} (\lambda_1^{1 - 2\gamma} \Uz)^2 T^4 + \Uc T^2.
 \end{align}
 
 Invoking the basic inequalities $1 - x \leq e^{-x}$ and
 $(x + y)^{1/2} \leq x^{1/2} + y^{1/2}$, valid for every $x,y \geq 0$,
 we obtain, for any $i \geq 1$,
 \begin{align}\label{V1i:Lyap}
  \nga{\bq_T}^i 
  &\leq e^{-\frac{\Ub T^2 i}{64}} \nga{\bq_0}^i  
    + C \sum_{j = 1}^i  \left( e^{-\frac{\Ub T^2 }{64}}\nga{\bq_0} \right)^j 
      \left( \nga{\bv_0}^{i-j} + 1 \right) \notag \\
  &\leq  e^{-\frac{\Ub T^2 i}{65}} \nga{\bq_0}^i 
      + \tilde{C} \left( \nga{\bv_0}^i + 1 \right),
 \end{align}
 where in the second inequality we invoked Young's inequality to
 estimate each term inside the sum, and with $C$ and $\tilde C$ being
 positive constants depending on $i, \lambda_1, \gamma, T, \Uz, \Ub$
 and $\Uc$. Since $\bv_0 \sim \mN(0, \cC)$, by Fernique's theorem
 (see, e.g., \cite[Theorem 2.7]{DaZa2014}) we have that
 $\bE \nga{\bv_0}^i < \infty$ for every $i \in \mathbb{N}$. Therefore,
 we conclude the result for $V_{1,i}$ after taking expected values in
 \eqref{V1i:Lyap} and iterating $n$ times on the resulting inequality.

 Finally, let us show \eqref{def:Lyap} for $V_{2,\eta}$ as in
 \eqref{eq:FL:2}.  Multiplying by $\eta$, taking the exponential and
 expected value on both sides of \eqref{bound:qt:3}, it follows that
 \begin{multline}\label{ineq:exp:Lyap}
   P V_2(\bq_0) = \bE \exp \left( \eta \nga{\bq_T}^2 \right) \\
      \leq \exp\left( \eta \left( 1 - \frac{\Ub T^2}{32} \right)  \nga{\bq_0}^2 \right) 
           \exp \left( \frac{5}{3} \eta (\lambda_1^{1 - 2 \gamma} \Uz)^2 T^4
             + \eta \Uc T^2 \right) 
        \bE \exp \left[\eta \left( \frac{32}{\Ub} + \frac{67}{8}  T^2 \right)
              \nga{\bv_0}^2 \right].
  \end{multline}
  Recalling $\bv_0 \sim \mathcal{N}(0, \cC)$ and the assumption
  $\eta < \left[ 2 \Tr(\cC^{1 - 2 \gamma}) \left( \frac{32}{\Ub} +
      \frac{67}{8} T^2 \right) \right]^{-1}$,
  we have, again by Fernique's theorem \cite[Proposition
  2.17]{DaZa2014}, and \cref{lem:ineq:C} that
  \begin{align}
    \label{est:Gaussian}
    \bE \exp \left[ \eta \left(  \frac{32}{\Ub} + \frac{67}{8} T^2 \right) 
                    \nga{\bv_0}^2 \right] 
      \leq \left[ 1 - 2 \eta \left( \frac{32}{\Ub} 
             + \frac{67}{8} T^2  \right) \Tr(\cC^{1 - 2 \gamma}) \right]^{-1/2}.
  \end{align}
  Thus, denoting $\tilde \kappa_2 = 1 - \Ub T^2/32$ and 
  \begin{align*}
    R = \exp \left( \frac{5}{3} \eta (\lambda_1^{1 - 2 \gamma} \Uz)^2 T^4 
                   + \eta \Uc T^2 \right) 
        \left[ 1 - 2 \eta \left( \frac{32}{\Ub} + \frac{67}{8} T^2  \right) 
                 \Tr(\cC^{1 - 2 \gamma}) \right]^{-1/2},
  \end{align*}
  we obtain from \eqref{ineq:exp:Lyap} and \eqref{est:Gaussian} that 
  \begin{align}\label{ineq:exp:Lyap:2}
	P V_{2,\eta}(\bq_0) 
	&\leq R \exp \left( \eta \tilde \kappa_2 \nga{\bq_0}^2 \right) 
          = R \exp \left( \eta \nga{\bq_0}^2 \right)^{\tilde \kappa_2} \notag \\
	&\leq \tilde \kappa_2 V_2(\bq_0)
              + R^{\frac{1}{1- \tilde \kappa_2}} (1 - \tilde \kappa_2)
           = \tilde \kappa_2 V_2(\bq_0)
                + R^{\frac{32}{\Ub T^2}} \frac{\Ub  T^2}{32} \notag \\
	&\leq e^{- \frac{\Ub T^2}{32}} V_2(\bq_0)
                 + R^{\frac{32}{\Ub T^2}} \frac{\Ub  T^2}{32}
  \end{align}
  where the second estimate follows by Young's inequality. We conclude
  \eqref{def:Lyap} for $V_{2,\eta}$ after using \eqref{ineq:exp:Lyap:2} $n$
  times iteratively.  The proof is now complete.
\end{proof}

\section{Pointwise contractivity bounds for 
  the Markovian dynamics}
\label{sec:Ptwise:contract:MD}

This section details two pointwise contractivity bounds for the
Markovian dynamics of the PHMC chain \eqref{eq:PEHMC:kernel:def} in a
suitably tuned Wasserstein-Kantorovich metric. These bounds provide
crucial ingredients needed for the weak Harris theorem, namely the so
called `$\rhoep$-contractivity' and `$\rhoep$-smallness' conditions,
which, together with the Lyapunov structure identified in
\cref{prop:FL}, form the core of the proof of \cref{thm:weak:harris}.

Our contraction results are given with respect to an underlying metric
$\rho: \H_\gamma \times \H_\gamma \to [0,1]$ defined as
\begin{align}
  \label{eq:rho:def}
  \rho(\bq, \btq) := \frac{ \nga{ \bq - \btq  }}{\varepsilon} \wedge 1,
\end{align}
where $\gamma$ is given in \cref{B123}.  On the other hand,
$\varepsilon > 0$ is a tuning parameter which specifies the small
scales in our problem and is determined by \eqref{def:smc} in such a
fashion as to produce a contraction in \eqref{eq:rho:contract}.
Recall that the Wasserstein distance on the space of probability
measures on $\H_\gamma$ induced by $\rho$ is given as in
\eqref{eq:Wass:Def} with $\tilde{\rho}$ replaced by $\rho$, and
denoted by $\Wass_\rho$.

The first result yielding `$\rhoep$-contractivity'
(cf. \cite[Definition 4.6]{hairer2011asymptotic}) is given as follows:
\begin{Prop}
  \label{prop:local_contr}
  Suppose \cref{A123}, \cref{ass:higher:reg:C} and \cref{B123} are
  satisfied and choose an integration time $T >0$ and
  $N \in \mathbb{N}$ maintaining the condition
  \eqref{cond:t:N:contr}. Fix any $\varepsilon > 0$ defining the
  associated metric $\rho$ as in \eqref{eq:rho:def}. Then, for every
  $n \in \mathbb{N}$ and for every $\bq_0, \btq_0 \in \H_{\gamma}$
  such that $\rho( \bq_0, \btq_0) < 1$, we have
\begin{align}
  \Wass_{\rhoep}(P^n(\bq_0, \cdot), P^n(\btq_0, \cdot)) \leq 
  \smc \rhoep(\bq_0, \btq_0)
  \label{eq:rho:contract}
\end{align} 
where recall that $P^n$ is $n$ steps of the PHMC kernel
\eqref{eq:n:step:PHMC:kernel} and $\Wass_\rho$ is the
Wasserstein distance, as in \eqref{ineq:main:result}, associated with $\rho$.
Here
\begin{align}
\label{def:smc}
  \smc = \smc(n)
  := \smb(n)   + \frac{ 2 \sqrt{2} \lambda_N^{-\frac{1}{2} + \gamma}
                              (1 + \lambda_1^{1 - 2 \gamma}\Ua)\varepsilon }
                               {T (1 - \sma^2)^{1/2}}
  =\smb(n)   + \frac{  \sqrt{2} \lambda_N^{-\frac{1}{2} + \gamma}
                              \alpha \varepsilon }
                               {2 T (1 - \sma^2)^{1/2}},
\end {align}
where 
\begin{align}
\label{def:sma:smb}
    \smb(n) := 4 \sqrt{2}(1 + \lambda_1^{1 - 2 \gamma} \Ua) \sma^n 
            = \sqrt{2} \alpha \sma^n,
    \quad 
    \sma:= 1 - \frac{T^2}{12},
 \end{align}
 $T >0$ is the integration time in \eqref{eq:PEHMC:kernel:def}, $\Ua$
 is the Lipschitz constant of $DU$ as in \eqref{Hess:bound} and
 $\lambda_1$ is the largest eigenvalue of $\cC$ and, in regards to
 $\alpha$, recall \eqref{eq:alpha:lip:def}.
\end{Prop}

\begin{Rmk}
  If $N \in \mathbb{N}$ is the smallest natural number for which the
  corresponding condition in \eqref{cond:t:N:contr} holds, i.e.
  \begin{align*}
  N = \min \left\{ n \in \mathbb{N} \,:\, \lambda_{n+1}^{1 - 2 \gamma} \leq \frac{1}{4 \Ua} \right\},
\end{align*}
then $\smc$ from \eqref{def:smc} above can be given in the more
explicit form
\begin{align*}
  \smc = \smc(n)
  := \smb(n)   + \frac{ 4 \sqrt{2} \Ua^{1/2}
                              (1 + \lambda_1^{1 - 2 \gamma}\Ua)\varepsilon }
                               {T (1 - \sma^2)^{1/2}}
  = \smb(n)   + \frac{ \sqrt{2} \Ua^{1/2} \alpha \varepsilon }
                               {T (1 - \sma^2)^{1/2}},
\end {align*}
with $\smb$ defined exactly as in \eqref{def:sma:smb} above. 
\end{Rmk}

Our second main result corresponding to `$\rho$-smallness'
(cf. \cite[Definition 4.4]{hairer2011asymptotic}) is given as:
\begin{Prop}\label{prop:smallness}
  Assume the same hypotheses from \cref{prop:local_contr}. Let
  $M \geq 0$ and take
  \begin{align*}
    A = \left\{ \bq \in \H_{\gamma} \,: \, \nga{\bq} \leq M \right\}.
  \end{align*}
  Then, for every $n \in \mathbb{N}$ and every $\varepsilon > 0$ we have
  for the corresponding $\rho$ defined by \eqref{eq:rho:def} that
  \begin{align}\label{Wass:small}
    \Wass_{\rhoep} (P^n(\bq_0, \cdot), P^n(\btq_0, \cdot)) \leq 1 - \smd
  \end{align}
  for every $\bq_0, \btq_0 \in A$, where
  \begin{align*}
    \smd = \smd(n) 
      &:= \frac{1}{2} \exp \left( - 
                   \frac{ 256 \Ua (1 + \lambda_1^{1 - 2 \gamma} \Ua)^2 M^2}
                        {T^2 (1 - \sma^2)} 
                          \right) 
         - \frac{2 M \smb(n) }{\varepsilon}  \\
      &= \frac{1}{2} \exp \left( - 
                   \frac{ 16 \Ua \alpha^2 M^2}
                        {T^2 (1 - \sma^2)} 
                          \right) 
         - \frac{2 M \smb(n) }{\varepsilon} ,
  \end{align*}
  with $\sma$ and $\smb$ as defined in \eqref{def:sma:smb}, and $\alpha$ as defined in \eqref{eq:alpha:lip:def}.
\end{Prop}

Before proceeding with the proofs of \cref{prop:local_contr} and
\cref{prop:smallness}, we introduce some further preliminary
terminology and general background.  Set an integration time $T > 0$
in the definition of the transition kernel $P$ of the PHMC chain,
\eqref{eq:PEHMC:kernel:def}. For each $n \in \mathbb{N}$, let
$\H^{\otimes n}$ denote the space given as the product of $n$ copies
of $\H$. Moreover, given a sequence $\{\bv_0^{(j)}\}_{ j \in \NN}$ of
i.i.d. draws from $\mathcal{N}(0,\cC)$, we denote by
$\npV_0^{(n)} = (\bv_0^{(1)}, \ldots, \bv_0^{(n)})$ the noise path for
the first $n \geq 1$ steps, as in \eqref{eq:noise:path:notation}. We
then have $\npV_0^{(n)} \sim \mN(0,\cC)^{\otimes n}$, with
$\mN(0,\cC)^{\otimes n}$ denoting the product of $n$ independent
copies of $\mN(0, \cC)$.

For simplicity of notation, we set from now on
\begin{align*}
    \sigma := \mN(0, \cC), \quad \sigma_n:= \mN(0,\cC)^{\otimes n}.
\end{align*}
For every $\bq_0, \btq_0 \in \H_\gamma$, with $\gamma$ as in
\eqref{eq:Tr:Def:b}, \eqref{Hess:bound}, and $N \in \mathbb{N}$ as in
\cref{prop:FP:Type:contract}, we consider
$\tQ_1(\btq_0, \bq_0): \H \to \H$ to be the random
variable defined as
\begin{align*}
\widetilde Q_1(\bq_0, \btq_0)(\bv_0^{(1)}) =
   \bq_T(\btq_0, \bv_0^{(1)} + T^{-1} \Pi_N (\bq_0 - \btq_0)) 
\end{align*}
where $\bv_0^{(1)} \sim \sigma$.  Iteratively we
define, for $n \geq 2$, the random variables
$\tQ_n(\bq_0, \btq_0):$  $\H^{\otimes n}$ $\to \H$ as
\begin{align}
  \widetilde{Q}_n(\bq_0, \btq_0)(\npV_0^{(n)}) 
  := q_T(\widetilde{Q}_{n-1}(\bq_0, \btq_0)( \npV_0^{(n-1)}), \bv_0^{(n)} 
  + \cS_n(\npV_0^{(n-1)})),
  \label{eq:shift:RV}
\end{align}
where $\npV_0^{(n)} \sim \sigma_n$, and
\begin{align}
\label{def:Phin}
    \cS_n(\npV_0^{(n-1)})
  := T^{-1}\Pi_N [Q_{n-1}(\bq_0)(\npV_0^{(n-1)})
                          -  \widetilde{Q}_{n-1}(\bq_0, \btq_0)(\npV_0^{(n-1)})].
\end{align}
We therefore obtain the shifted noise path
\begin{align}
  \tilde\npV_0^{(n)}
  = ( \bv_0^1 + \cS_1, \bv_0^{(2)} + \cS_2(\npV_0^{(1)}),
         \ldots,  \bv_0^{(n)} + \cS_n(\npV_0^{(n-1)})),
  \quad \mbox{ where } \cS_1 = T^{-1} \Pi_N (\bq_0 - \btq_0).
  \label{eq:shift:Noise}
\end{align}
Let $\tsig_n := \text{Law}(\tilde \npV_0^{(n)})$. In order to simplify
notation, let us denote
\begin{align}
\label{def:bPhi}
 \bcS_n(\npV_0^{(n)}) = (\cS_1, \cS_2(\npV_0^{(1)}), \ldots, \cS_n(\npV_0^{(n-1)}))
\end{align}
and
\begin{align}
\label{def:bPsi}
 \bcR_n(\npV_0^{(n)}) = \npV_0^{(n)} + \bcS_n(\npV_0^{(n)}),
\end{align}
so that $\tilde\npV_0^{(n)} = \bcR_n(\npV_0^{(n)})$. Thus, $\tsig_n$
is the push-forward of $\sigma_n$ by the mapping
$\bcR_n: \H^{\otimes n} \to \H^{\otimes n}$, i.e.
$\tilde \sigma_n = \bcR_n^\ast \sigma_n$.
Now put, for every $n \in \mathbb{N}$ and $A \in \mB(\H)$,
\begin{align}
  \widetilde{P}^n(\bq_0, \btq_0, A)
  = \widetilde Q_n(\bq_0, \btq_0)^\ast \sigma_n(A)
  = \sigma_n( \widetilde{Q}_n(\bq_0, \btq_0)^{-1}(A)).
  \label{eq:Shifted:MK}
\end{align}
Notice that $\widetilde{P}^n(\bq_0, \btq_0, \cdot)$ can be equivalently written as
\begin{align}
  \widetilde{P}^n(\bq_0, \btq_0, A)
      = Q_n(\btq_0)^\ast (\bcR_n^\ast \sigma_n)(A)
     = Q_n(\btq_0)^\ast \tilde \sigma_n (A).
    \label{eq:tPn}
\end{align}

With these notations in place we have the following estimate 
which we will use several times below in establishing \cref{prop:local_contr}, \cref{prop:smallness}. The proof follows immediately from \cref{prop:FP:Type:contract}
and \cref{rmk:equiv:norms}.
\begin{Lem}\label{lem:contract}
  We are maintaining the same hypotheses as in \cref{prop:local_contr}. Then,
  starting from any $\bq_0, \btq_0 \in \H_{\gamma}$ we have that for all
  $n \geq 1$,
  \begin{align*}
    \nga{Q_n(\bq_0)(\npV_0^{(n)}) - \tQ_n(\bq_0, \btq_0)(\npV_0^{(n)})}
         \leq \smb \nga{\bq_0 - \btq_0} \quad \mbox{ for every }
    \npV_0^{(n)} \in \H^{\otimes n},
  \end{align*}
   where $Q_n$ and $\tQ_n$ are defined as in
   \eqref{eq:n:step:chain:Noise:p} and \eqref{eq:shift:RV},
   respectively, and $\smb$ is as in \eqref{def:sma:smb}. Therefore,
  \begin{align}
    \E \nga{Q_n(\bq_0) - \tilde{Q}_n(\bq_0, \btq_0)} \leq
    \smb \nga{\bq_0 - \btq_0}.
    \label{eq:nudged:ave:contr}
  \end{align}  
\end{Lem}

We also recall additional notions of distances in the space of Borel
probability measures on a given complete metric space $(X,d)$, denoted
$\Pr(X)$, with the associated Borel $\sigma$-algebra denoted as
$\mB(X)$. Namely, the \textit{total variation} distance is defined as
\begin{align}
  \tv{\nu - \tilde{\nu}} 
  := \sup_{A \in \mathcal{B}(X)} | \nu(A) - \tilde{\nu}(A)|
  \label{def:tv:dist}
\end{align}
for any $\nu, \tilde{\nu} \in \Pr(X)$. On the other hand when
$\tilde{\nu} \ll \nu$, i.e. when $\tilde{\nu}$ is absolutely
continuous with respect to $\nu$, the \textit{Kullback-Leibler
  Divergence} is defined as
\begin{align}
  \KL (\tilde{\nu} | \nu) :=
  \int_{X} \log\left( \frac{d \tilde{\nu}}{d \nu}(\npV) \right) 
           d \tilde{\nu}(d \npV).
  \label{eq:KL:div}
\end{align}
Recall that for the trivial metric
\begin{align*}
  \rho_0(\bq, \btq) :=
  \begin{cases}
    1&\text{ if } \bq \not =  \btq\\
    0&\text{ if } \bq = \btq,
  \end{cases}
\end{align*}
the associated Wasserstein distance $\Wass_{\rho_0}$ coincides with
the total variation distance. On the other hand, Pinsker's inequality
(see e.g. \cite{Tsybakov2009}) states that
\begin{align}
  \tv{\nu - \tilde{\nu} } \leq \sqrt{\frac{1}{2} \KL (\tilde{\nu} | \nu)},
  \label{ineq:TV:KL:1}
\end{align}
for any $\nu, \tilde \nu \in \Pr(X)$, $\tilde{\nu} \ll \nu$. Moreover,
as showed e.g. in \cite[Appendix]{butkovsky2018generalized},
\begin{align}
    \tv{\nu - \tilde{\nu}} \leq 1 
  - \frac{1}{2} \exp\left( - \KL(\tilde \nu | \nu ) \right)
    \label{ineq:TV:KL:2}
\end{align}
for all $\nu, \tilde \nu \in \Pr(X)$, $\tilde{\nu} \ll \nu$.

\begin{proof}[Proof of \cref{prop:local_contr}]
  Fix any $\bq_0, \btq_0 \in \H_{\gamma}$ such that
  $\rhoep( \bq_0, \btq_0) < 1$. Then, recalling the notation
  \eqref{eq:Shifted:MK} and using that $\rho$ is a metric on $\H$ we have
\begin{align}
        \label{Wass:triang:ineq}
  \Wass_{\rhoep}(P^n(\bq_0, \cdot), P^n(\btq_0, \cdot))
  \leq \Wass_{\rhoep}(P^n(\bq_0, \cdot), \widetilde{P}^n(\bq_0,\btq_0, \cdot)) 
  +  \Wass_{\rhoep}( \widetilde{P}^n(\bq_0,\btq_0, \cdot), P^n(\btq_0, \cdot)).
\end{align}
Notice that 
\begin{align}
  \Wass_{\rhoep}(P^n(\bq_0, \cdot), \widetilde{P}^n(\bq_0,\btq_0, \cdot)) 
  \leq& \bE \rhoep(Q_n(\bq_0), \tQ_n(\bq_0, \btq_0))
  \leq \frac{1}{\varepsilon}
          \bE \nga{Q_n(\bq_0) - \tQ_n(\bq_0, \btq_0)}
  \notag\\
  \leq& \frac{\smb}{\varepsilon} \nga{\bq_0 - \btq_0}
  = \smb \rhoep(\bq_0,  \btq_0),
  \label{Wass:P:tP}
\end{align}
where the last inequality follows from Lemma \ref{lem:contract}.  

For the second term in \eqref{Wass:triang:ineq}, it follows from the
coupling lemma (see e.g. \cite[Lemma 1.2.24]{kuksin2012mathematics})
and the fact that $\rho \leq 1$ that
\begin{align}
  \Wass_{\rhoep}( \widetilde{P}^n(\bq_0,\btq_0, \cdot), P^n(\btq_0, \cdot)) 
  \leq  \tv{\widetilde{P}^n(\bq_0,\btq_0, \cdot) - P^n(\btq_0, \cdot)}. 
  \label{Wass:tP:P:0a}
\end{align}
From \eqref{def:Pn:Qn} and \eqref{eq:tPn}, we have
\begin{align*}
  \tv{\widetilde{P}^n(\bq_0,\btq_0, \cdot) - P^n(\btq_0, \cdot)} 
  = \tv{Q_n(\btq_0)^\ast \tsig_n - Q_n(\btq_0)^\ast \sigma_n}.
\end{align*}
Moreover, from the definition of the total variation distance in
\eqref{def:tv:dist} and inequality \eqref{ineq:TV:KL:1}, we infer
\begin{align}
  \tv{Q_n(\btq_0)^\ast \tsig_n - Q_n(\btq_0)^\ast \sigma_n}
  \leq \tv{\tsig_n - \sigma_n}
  \leq \sqrt{\frac{1}{2} \KL (\tsig_n | \sigma_n)}.
    \label{Wass:tP:P:0}
\end{align}
As a consequence of Girsanov's Theorem, we obtain
\begin{align}
  \label{Girsanov}
  \frac{d \sigma_n}{d \tsig_n}(\bcR_n(\npV)) 
  = \exp \left( \frac{1}{2} |\cC^{-1/2} \npV|^2 
              - \frac{1}{2} |\cC^{-1/2} \bcR_n(\npV)|^2 \right) 
  \quad \text{ for any }  \npV \in \H_{1/2}^{\otimes n},
\end{align}
with $\bcR_n$ as defined in \eqref{def:bPsi}. Thus,
\begin{align}
  \KL(\tsig_n | \sigma_n) 
  =& \int \log \left( \frac{d \tsig_n}{d \sigma_n}(\npV) \right) \tsig_n(d\npV) 
  = - \int \log \left( \frac{d \sigma_n}{d \tsig_n}(\npV) \right) \tsig_n(d\npV)
  \notag\\
  =& - \int \log \left( \frac{d \sigma_n}{d \tsig_n}(\bcR_n(\npV)) \right) 
          \sigma_n(d \npV) 
  = \int \left( - \frac{1}{2} |\npV|_{1/2}^2 
         + \frac{1}{2} |\bcR_n(\npV)|_{1/2}^2 \right) \sigma_n(d \npV) 
  \notag\\
  =& \int \left( \langle \bcS_n(\npV), \npV \rangle_{1/2} 
             + \frac{1}{2} |\bcS_n(\npV)|_{1/2}^2 \right) \sigma_n(d \npV)
  = \frac{1}{2} \int |\bcS_n(\npV)|_{1/2}^2 \sigma_n(d \npV) 
  \notag\\
  =& \frac{1}{2} \sum_{j=1}^n \bE |\cS_j(\cdot)|_{1/2}^2.
     \label{ineq:KL}
\end{align}
Here note that, taking $\npV = (\bv_1, \ldots, \bv_n)$
and $\npV^j = (\bv_1, \ldots, \bv_j)$ for $j \leq n$
we have
\begin{align*}
  \int \langle \bcS_n(\npV), \npV \rangle_{1/2} \sigma_n(d \npV)
  =& \sum_{j = 1}^n \int \langle \cS_j(\npV^{j-1}), \bv_j \rangle_{1/2} \sigma_n(d \npV)\\
  =& \sum_{j = 1}^n \int 
         \int \langle \cS_j(\npV^{j-1}), \bv_j \rangle_{1/2} \sigma(d \bv_j)
             \sigma_{j-1}(d\npV^{j-1})
  = 0,
\end{align*}
which justifies dropping this term in \eqref{ineq:KL}.  Now, from the
definition of $\cS_j$ in \eqref{def:Phin}, \eqref{eq:Inv:Poincare}
in \cref{lem:ineq:C} and \eqref{equiv:norms:b} it follows that
\begin{align*}
  |\cS_j(\npV_0^{j-1})|_{1/2}^2 
  \leq& \lambda_N^{-1 + 2 \gamma} \nga{\cS_j(\npV_0^{j-1})}^2 
  \leq \lambda_N^{-1 + 2 \gamma} \normga{\cS_j(\npV_0^{j-1})}^2 
  \leq T^{-2} \lambda_N^{-1 + 2 \gamma} \sma^{2(j-1)} \normga{\bq_0 - \btq_0}^2 \\ 
  \leq& T^{-2} \lambda_N^{-1 + 2 \gamma} \sma^{2(j-1)} 2 \alpha^2 \nga{\bq_0 - \btq_0}^2,
\end{align*}
for each $j \geq 1$, with $\alpha$ as defined in
\eqref{eq:alpha:lip:def}. Therefore,
\begin{align}
  \KL(\tsig_n | \sigma_n) 
  \leq \frac{\lambda_N^{-1 + 2 \gamma} \alpha^2}{T^2} \nga{\bq_0 - \btq_0}^2 \sum_{j=1}^n \sma^{2(j-1)} 
  \leq  \frac{\lambda_N^{-1 + 2 \gamma} \alpha^2}{T^2 (1 - \sma^2)}  \nga{\bq_0 - \btq_0}^2,
          \label{eq:D:KL:diff}
\end{align}
so that, combining this observation with \eqref{Wass:tP:P:0a}-\eqref{Wass:tP:P:0}, and our standing
assumption that $\rho(\bq_0, \btq_0) < 1$,
\begin{align}
\label{Wass:tP:P}
  \Wass_{\rhoep}(\widetilde{P}^n(\bq_0,\btq_0, \cdot), P^n(\btq_0, \cdot)) 
    \leq \frac{\lambda_N^{-\frac{1}{2} + \gamma} \alpha }
              {\sqrt{2} T (1 - \sma^2)^{1/2}} \nga{\bq_0 - \btq_0} 
     = \frac{\lambda_N^{-\frac{1}{2} + \gamma} \alpha \varepsilon }{\sqrt{2} T (1 - \sma^2)^{1/2}} 
                        \rhoep(\bq_0, \btq_0).
\end{align}
We therefore conclude \eqref{eq:rho:contract} from
\eqref{Wass:triang:ineq}, \eqref{Wass:P:tP} and \eqref{Wass:tP:P},
completing the proof of \cref{prop:local_contr}.
\end{proof}

\begin{proof}[Proof of \cref{prop:smallness}]
  We proceed similarly as in the proof of Proposition
  \ref{prop:local_contr} starting with the splitting
  \eqref{Wass:triang:ineq}. Fix any $\bq_0, \btq_0 \in A$. The first
  term after inequality \eqref{Wass:triang:ineq} is estimated
  exactly as in \eqref{Wass:P:tP}, so that 
  \begin{align*}
  \Wass_{\rhoep}(P^n(\bq_0, \cdot), \widetilde{P}^n(\bq_0,\btq_0, \cdot)) 
    \leq \frac{\smb}{\varepsilon} \nga{\bq_0 - \btq_0}
    \leq \frac{2 M \smb }{\varepsilon}.
  \end{align*}
  The second term in  \eqref{Wass:triang:ineq} is estimated by using \eqref{ineq:TV:KL:2} and \eqref{eq:D:KL:diff} as 
  \begin{align*}
    \Wass_{\rhoep}(\widetilde{P}^n(\bq_0,\btq_0, \cdot), P^n(\btq_0, \cdot)) 
     \leq& \tv{\tsig_n - \sigma_n} 
     \leq 1 - \frac{1}{2} \exp \left( - \KL(\tsig_n | \sigma_n) \right) \\
     \leq& 1 - \frac{1}{2} \exp \left( - \frac{\lambda_N^{-1 + 2 \gamma} \alpha^2}{T^2 (1 - \sma^2)}  
              \nga{\bq_0 - \btq_0}^2 \right),
 \end{align*}
 with $\alpha$ as defined in \eqref{eq:alpha:lip:def}. Hence, together
 with \eqref{Wass:triang:ineq} and using that $\bq_0, \btq_0 \in A$,
 we conclude \eqref{Wass:small}.
\end{proof}

\section{Main Result}
\label{sec:Proof:Main}

Having obtained in the previous sections a Foster-Lyapunov structure
\eqref{ineq:Lyap} together with the smallness and contractivity
properties \eqref{eq:rho:contract}-\eqref{Wass:small} for the Markov
kernel $P$ in \eqref{eq:HMC:kernel:overview}, we are now ready to
proceed with the proof of our main result. As pointed out in the
introduction, the spectral gap \eqref{Wass:trhoe} below follows as a
consequence of the weak Harris theorem given the aforementioned
properties.

We provide a self-contained presentation of the weak Harris approach
in this section both for completeness and in order to make some of the
constants in the proof more explicit. We start by noticing that it is
enough to show \eqref{Wass:trhoe} for $\nu_1, \nu_2$ being Dirac
measures, say concentrated at points $\bq_0, \btq_0 \in \H_\gamma$.
The proof is then split into three possible cases for such points:
$\rho(\bq_0, \btq_0) < 1$ (`close to each other');
$\rho(\bq_0, \btq_0) = 1$ with $V(\bq_0) + V(\btq_0) > 4 K_V$ (`far
from the origin'); and $\rho(\bq_0, \btq_0) = 1$ with
$V(\bq_0) + V(\btq_0) \leq 4 K_V$ (`close to the origin'). The first
case follows from the contraction result in \cref{prop:local_contr}
together with the Lyapunov structure from \cref{prop:FL}. The second
case follows entirely from the Lyapunov property. Lastly, the third
case follows by invoking the smallness result in \cref{prop:smallness}
as well as the Lyapunov structure. Finally, the second part of our
main result, namely \eqref{ineq:obs:2}-\eqref{main:thm:CLT}, follows
essentially from the spectral gap \eqref{Wass:trhoe} by invoking
\cref{lem:Kantor:dual}, \cref{prop:CLT:LLN} and \cref{prop:Lip:2:obs},
which are all proved in detail in \cref{sec:LLN:CLT}.

\begin{Thm}\label{thm:weak:harris}
  Fix $\gamma \in [0,1/2)$. Suppose \cref{A123},
  \cref{ass:higher:reg:C}, \cref{B123} and
  \cref{ass:integrability:cond} are satisfied and choose an
  integration time $T >0$ such that
  \begin{align}
    T \leq  \min \left\{ \frac{1}{[2(1 + \lambda_1^{1 - 2\gamma}\Ua)]^{1/2}},
    \frac{\Ub^{1/2}}{2\sqrt{6} (1 + \lambda_1^{1 - 2\gamma} \Ua)} \right\}.
	\label{eq:time:restrict:basic}
  \end{align}
  Here the constants $L_1, L_2$ are as in \eqref{Hess:bound} and
  \eqref{dissip:new} and $\lambda_1$ is the largest eigenvalue of the
  covariance operator $\cC$ defined as in \cref{A123}.  Let
  $V: \H_\gamma \to \mathbb{R}^+$ be a Lyapunov function for the
  Markov kernel $P$ defined in \eqref{eq:PEHMC:kernel:def} of the form
  \eqref{eq:FL:1} or \eqref{eq:FL:2}. Then, there exists
  $\varepsilon > 0$, $C_1 > 0$ and $C_2 > 0$ such that, for every
  $\nu_1, \nu_2 \in \Pr(\H)$ with support included in $\H_{\gamma}$,
  \begin{align}\label{Wass:trhoe}
    \Wass_{\trhoe} (\nu_1 P^n, \nu_2 P^n) \leq
    C_1 e^{-C_2 n} \Wass_{\trhoe}(\nu_1, \nu_2) 
    \quad \mbox{ for all } n \in \mathbb{N},
  \end{align}
  where $\trhoe: \H_{\gamma} \times \H_{\gamma} \to \RR^+$ is the
  distance-like function given by
  \begin{align*}
    \trhoe(\bq, \tilde \bq) =
    \sqrt{\rhoep(\bq,\tilde \bq) (1 + V(\bq) + V(\tilde \bq))}
    \quad \mbox{ for all } \bq, \tilde \bq \in \H_{\gamma}, 
  \end{align*}
  with $\rhoep$ as defined in \eqref{eq:rho:def}.

  Moreover, with respect to $\mu$ defined in \eqref{eq:mu}, i.e. the
  invariant measure for $P$ (cf. \cref{eq:invariance}), the following
  results hold: for any observable $\Phi: \H_\gamma \to \RR$ such that
  \begin{align}\label{main:thm:Lip:const}
    L_\Phi :=        \sup_{q \in \H_\gamma} 
          \frac{\max\{ 
          2   |\Phi(\bq)|, \sqrt{\varepsilon} |D \Phi(\bq)|_{\mL(\H_\gamma)}
          \}}{\sqrt{1+ V(\bq)}} < \infty,
  \end{align}
  with $|\cdot|_{\mL(\H_\gamma)}$ denoting the standard operator
  norm of a linear functional on $\H_\gamma$, we have
  \begin{align}
          \left| P^n \Phi(\bq) - \int \Phi(\bq') \mu(dq') \right|
          \leq L_\Phi C_1 e^{-n C_2} \int \sqrt{1 + V(\bq) + V(\bq')} \mu(d \bq'),
          \label{ineq:obs:2}
  \end{align}
  for every $n \in \mathbb{N}$ and $\bq \in \H_\gamma$.  On the other
  hand, taking $\{Q_k(\bq_0)\}_{k \geq 0}$ to be any process associated
  to $\{P^k(\bq_0, \cdot)\}_{k \geq 0}$ as in \eqref{def:Pn:Qn},
  we have, for any measurable observable maintaining \eqref{main:thm:Lip:const},
  that
  \begin{align}\label{main:thm:LLN}
    \lim_{n \to \infty} \frac{1}{n} \sum_{k=1}^n \Phi(Q_k(\bq)) =  \int \Phi(\bq')
          \mu(d\bq'), \quad \text{ almost surely},
  \end{align}
  for all $\bq \in \H_\gamma$. Furthermore,
  \begin{align}\label{main:thm:CLT}
          \sqrt{n} \left[ \frac{1}{n} \sum_{k=1}^n \Phi(Q_k(\bq))
                              -  \int \Phi(\bq') \mu (d \bq')) \right]
          \Rightarrow \mN(0, \sigma^2(\Phi))
          \quad  \mbox{ as } n \to \infty,
  \end{align}
  for all $\bq \in \H_\gamma$, i.e. the expression in the left-hand
  side of \eqref{main:thm:CLT} converges weakly to a real-valued
  gaussian random variable with mean zero and covariance
  $\sigma^2(\Phi)$, where $\sigma^2(\Phi)$ is specified explicitly as
  \eqref{eq:CLT:sig:def} below, with $\mu^*$ replaced by $\mu$.
\end{Thm}	
\begin{proof}
  We claim it suffices to show that there exists $\varepsilon > 0$,
  $C_1 > 0$ and $C_2 > 0$ such that
  \begin{align}\label{Wass:points}
    \Wass_{\trhoe}(P^n(\bq_0,\cdot), P^n(\tilde \bq_0,\cdot))
    \leq C_1 e^{-C_2 n} \trhoe(\bq_0, \tilde \bq_0)
    \quad \mbox{ for all } \bq_0, \tilde \bq_0 \in \H_{\gamma}
    \mbox{ and } n \in \mathbb{N}.
  \end{align}
  Indeed, since $\trhoe$ is lower-semicontinuous and non-negative, it
  follows from \cite[Theorem 4.8]{villani2008optimal} that
  \begin{align*}
    \Wass_{\trhoe}(\nu_1 P^n, \nu_2 P^n ) 
    \leq \int \Wass_{\trhoe}(P^n(\bq_0,\cdot), P^n( \tilde \bq_0, \cdot)) 
                 \Gamma (d \bq_0, d \tilde \bq_0) 
       \quad \mbox{ for all } 
       \Gamma \in \Co(\nu_1 , \nu_2) \mbox{ and } n \in \mathbb{N}.
  \end{align*}
  Clearly, if $\nu_1$ and $\nu_2$ have supports included in
  $\H_{\gamma}$, then $\Gamma \in \Co(\nu_1 , \nu_2)$ has support
  included in $\H_\gamma \times \H_\gamma$. Hence, if
  \eqref{Wass:points} holds then
  \begin{align}\label{Wass:points:2}
    \Wass_{\trhoe}(\nu_1 P^n, \nu_2 P^n) 
    \leq  C_1 e^{-C_2 n} \int \trhoe(\bq_0, \tilde \bq_0) 
         \Gamma(d \bq_0, d \tilde \bq_0) 
    \quad \mbox{ for all } \Gamma \in \Co(\nu_1, \nu_2) 
    \mbox{ and } n \in \mathbb{N},
  \end{align}
  which implies \eqref{Wass:trhoe}.
	
  In order to show \eqref{Wass:points}, we consider an auxiliary
  metric defined as
  \begin{align*}
    \trhob(q, \tq)
    = \sqrt{\rho(\bq,\tilde \bq) (1 + \beta V(\bq) + \beta V(\tilde \bq)},
          \quad \mbox{ for all } \bq, \tilde \bq \in \H_{\gamma},
  \end{align*}
  with the additional parameter $\beta > 0$ to be appropriately chosen
  below; cf. \eqref{cond:eps:beta}. Notice that $\trhoe$ and $\trhob$
  are equivalent. Indeed,
  \begin{align}\label{equiv:trhoe:trhob}
    \left( \min\{1,\beta\}\right)^{1/2} \trhoe(\bq,\tilde \bq) 
    \leq \trhob(\bq, \tilde \bq) 
    \leq \left( \max\{1,\beta\} \right)^{1/2} \trhoe(\bq, \tilde \bq), 
    \quad \mbox{ for all } \bq, \tilde \bq \in \H_{\gamma}.
  \end{align}
  We now show that 
  \begin{align}\label{Wass:trhob}
    \Wass_{\trhob}(P^n(\bq_0, \cdot), P^n(\tilde \bq_0, \cdot)) 
    \leq \sme(n) \trhob(\bq_0, \tilde \bq_0) 
    \quad \mbox{ for all } n \geq 1 
    \mbox{ and } \bq_0, \tilde \bq_0 \in \H_{\gamma},
  \end{align}
  such that, for suitably chosen $\varepsilon > 0$, $\beta > 0$, and
  for $n_0 \in \mathbb{N}$ sufficiently large we have
  $\sme(n) < 1$ for every $n \geq n_0$.  We then subsequently use this
  bound to establish \eqref{Wass:points} as in \eqref{def:C1:C2} below.

  The analysis leading to \eqref{Wass:trhob} is split into three cases:
  
  \noindent\textit{\underline{Case 1:} Suppose that
  $\rhoep(\bq_0, \tilde \bq_0) < 1$, so that
  $\rhoep(\bq_0, \tilde \bq_0) = \nga{\bq_0 - \tilde \bq_0}
  \varepsilon^{-1}$.}  

  By H\"older's inequality, we obtain
  \begin{align}\label{Wass:Holder}
    \Wass_{\trhob}(P^n(\bq_0, \cdot), P^n(\tilde \bq_0, \cdot))^2
    &\leq \inf_{\Gamma \in \Co(\delta_{\bq_0} P^n,\delta_{\tilde \bq_0} P^n)}
      \left\{ \!
      \left( \int \rhoep(\bq, \tilde \bq) \Gamma(d\bq, d\tilde \bq) \right)
      \! \!
      \left( \int (1 + \beta V(\bq) + \beta V(\tilde \bq)) 
             \Gamma(d\bq, d\tilde \bq) \right) 
      \! \right\} \nonumber\\
    &= \left( 1 + \beta P^n V(\bq_0) + \beta P^n V(\tilde \bq_0)\right) 
      \Wass_{\rhoep}(P^n(\bq_0, \cdot), P^n(\tilde \bq_0, \cdot)).
  \end{align}
  From \cref{prop:FL}, and \cref{prop:local_contr}, it
  follows that
  \begin{align}\label{Wass:case:1}
    \Wass_{\trhob}(P^n(\bq_0, \cdot), P^n(\tilde \bq_0, \cdot))^2
    &\leq \left( 1 + \beta \kappa_V^n V(\bq_0) 
                   + \beta \kappa_V^n V(\tilde \bq_0) 
                   + 2 \beta K_V \right) 
          \smc \rhoep(\bq_0, \tilde \bq_0)\nonumber \\
    &\leq \left( 1 + \beta V(\bq_0) 
                   + \beta V(\tilde \bq_0) 
                   + 2 \beta K_V \right) 
          \smc \rhoep(\bq_0, \tilde \bq_0) \nonumber \\
    &\leq  \smc (1 + 2 \beta K_V) 
           \left( 1 + \beta V(\bq_0) 
                    + \beta V(\tilde \bq_0) \right) 
      \rhoep(\bq_0, \tilde \bq_0) \nonumber \\
    &= \smc (1 + 2 \beta K_V) \left( \trhob(\bq_0, \tilde \bq_0) \right)^2.
  \end{align}
  
	\noindent\textit{\underline{Case 2:} Suppose that $\rhoep(\bq_0, \tilde \bq_0) = 1$ and $V(\bq_0) + V(\tilde \bq_0) > 4 K_V$.}
	
	Since $\rhoep(\cdot, \cdot) \leq 1$ and again invoking \cref{prop:FL} we obtain
	\begin{align}\label{Wass:case:2}
		\Wass_{\trhob}(P^n(\bq_0, \cdot), P^n(\tilde \bq_0, \cdot))^2 
		&\leq 1 + \beta P^n V(\bq_0) + \beta P^n V(\tilde \bq_0) \notag \\
		&\leq 1 + \beta \kappa_V^n V(\bq_0) + \beta \kappa_V^n V(\tilde \bq_0) + 2 \beta K_V \notag \\
		&= \frac{1 + 2 \beta K_V}{1 + 3 \beta K_V} (1 + 3 \beta K_V) + \kappa_V^n \beta ( V(\bq_0) + V(\tilde \bq_0)) \notag \\
		&\leq \max \left\{ \frac{1 + 2 \beta K_V}{1 + 3 \beta K_V}, 4 \kappa_V^n \right\} \left( 1 + 3 \beta K_V + \frac{\beta}{4} (V(\bq_0) + V(\tilde \bq_0)) \right) \notag \\
		&< \max \left\{ \frac{1 + 2 \beta K_V}{1 + 3 \beta K_V}, 4 \kappa_V^n \right\} \left( 1 + \beta V(\bq_0) + \beta V(\tilde \bq_0) \right) \notag\\
		&= \max \left\{ \frac{1 + 2 \beta K_V}{1 + 3 \beta K_V}, 4 \kappa_V^n \right\} \left( \trhob(\bq_0, \tilde \bq_0) \right)^2.
	\end{align}
	
	\noindent\textit{\underline{Case 3:} Suppose that
          $\rhoep(\bq_0, \tilde \bq_0) = 1$ and
          $V(\bq_0) + V(\tilde \bq_0) \leq 4 K_V$. }
	
	We proceed as in \eqref{Wass:Holder}, but now use Proposition
        \ref{prop:smallness} to estimate the term
        $\Wass_{\rhoep}(P^n(\bq_0, \cdot), P^n(\tilde \bq_0,
        \cdot))$. First, let $M_V > 0 $ be such that
	\begin{align*}
          \left\{ \bq \in \H_{\gamma} \,:\, V(\bq) \leq 4 K_V \right\} 
             = \left\{ \bq \in \H_{\gamma} \,:\, \nga{\bq} \leq M_V \right\}.
	\end{align*}
	Notice that the specific definition of $M_V$ depends on the
        choice of Lyapunov function $V$ (which defines the constant $K_V$,
        cf. \eqref{eq:FL:1}-\eqref{eq:FL:2}).  Thus, for any
        $\bq_0, \tilde \bq_0 \in \left\{ \bq \in \H_{\gamma} \,:\,
          V(\bq) \leq 4 K_V \right\}$
        from Proposition \ref{prop:smallness}, it follows that
	\begin{align*}
          \Wass_{\rhoep}(P^n(\bq_0, \cdot), P^n(\tilde \bq_0, \cdot)) \leq 1 - \smd,
	\end{align*}
	where
        \begin{align}\label{def:smd}
          \smd = \smd(n) 
              := \frac{1}{2} 
                \exp \left( - \frac{ 16 \Ua \alpha^2 M_V^2}{T^2 (1 - \sma^2)} \right) 
                             - \frac{2 M_V \smb(n)}{\varepsilon},
	\end{align}
	with $\sma$ and $\smb$ as defined in \eqref{def:sma:smb} and
        $\alpha = 4 (1 + \lambda_1^{1 - 2 \gamma} \Ua)$
        (cf. \eqref{eq:alpha:lip:def}). Hence,
	\begin{align}\label{Wass:case:3}
		\Wass_{\trhob}(P^n(\bq_0, \cdot), P^n(\tilde \bq_0, \cdot))^2 
		&\leq (1 - \smd) 
                    \left( 1 + \beta \kappa_V^n 
                       \left( V(\bq_0) + V(\tilde \bq_0) \right) + 2 \beta K_V\right) \nonumber \\
		&\leq (1 - \smd) (1 + 2 (1 + 2 \kappa_V^n) \beta K_V) \nonumber \\
		&\leq (1 - \smd)  (1 + 2 (1 + 2 \kappa_V^n) \beta K_V) 
                         \left( \trhob(\bq_0, \tilde \bq_0) \right)^2.
	\end{align}
	
	From \eqref{Wass:case:1}, \eqref{Wass:case:2} and
        \eqref{Wass:case:3}, we now obtain the bound \eqref{Wass:trhob} with
        $\sme = \sme(n)$ defined as
	\begin{align}
           \biggl(
          \max \biggl\{ (1 + 2 \beta K_V) \smc(n),
          \max
          \left\{\frac{1 + 2\beta K_V}{1 + 3 \beta K_V}, 4
            \kappa_V^n\right\},
          (1 - \smd(n)) (1 + 2 (1 + 2 \kappa_V^n) \beta K_V)
            \biggr\} \biggr)^{1/2}.
            \label{def:C3}
	\end{align}
	We claim that if we now choose $\varepsilon > 0 $, $\beta > 0$
        satisfying
	\begin{align}\label{cond:eps:beta}
	\varepsilon \leq  \frac{T(1 - \sma^2)^{1/2}}{8 \sqrt{2} \alpha \Ua^{1/2}} 
	\quad \mbox{ and } \quad
	\beta \leq \frac{1}{12 K_V} 
          \exp \left(  -  \frac{ 16 \Ua \alpha^2 M_V^2}{T^2 (1 - \sma^2)} \right),
	\end{align}
	and $n_0 \in \mathbb{N}$ satisfying 
	\begin{align}\label{cond:n0}
          \sma^{n_0} \leq 
              \min \left\{  \frac{1}{4 \sqrt{2} \alpha} , \frac{\varepsilon}{8 \sqrt{2} \alpha M_V} 
                  \exp \left( -  \frac{ 16 \Ua \alpha^2 M_V^2}{T^2 (1 - \sma^2)}  \right) \right\} 
		\quad \mbox{ and } \quad 
          \kappa_V^{n_0} \leq \frac{1}{8}, 
	\end{align}
	then indeed we have 
        \begin{align}\label{def:sme}
          \sme(n) \leq \smeb \leq
          \left(\max \left\{ \frac{1 + 2 \beta K_V}{1 + 3 \beta K_V}, 
                1 - \frac{1}{16} \exp \left( -  \frac{ 32 \Ua \alpha^2 M_V^2}
                        {T^2 (1 - \sma^2)} \right) \right\}  \right)^{1/2} < 1 
          \quad \mbox{ for all } n \geq n_0,
	\end{align}
	as we desired in the estimate \eqref{Wass:trhob}.

        To see this bound in \eqref{def:sme} observe that since
        $\sma^{n_0} \leq (4 \sqrt{2} \alpha)^{-1}$ and $\varepsilon$
        satisfies the first inequality in \eqref{cond:eps:beta}, then
        it follows from the definitions of $\smb$ and $\smc$ in
        \eqref{def:sma:smb} and \eqref{def:smc}, respectively, that
	\begin{align}\label{ineq:smc}
	\smb(n) \leq \frac{1}{4} \quad \mbox{ and } 
          \quad \smc(n) \leq \frac{3}{8} \quad \mbox{ for all } n \geq n_0.
	\end{align}
	From \eqref{cond:eps:beta}, we have in particular that
        $\beta \leq (12 K_V)^{-1}$. Together with \eqref{ineq:smc},
        this yields
	\begin{align}\label{ineq:first:C3}
	(1 + 2 \beta K_V) \smc(n) \leq \frac{1}{2} 
          \quad \mbox{ for all } n \geq n_0.
	\end{align}
	Moreover, since $\kappa_V^{n_0} \leq 1/8$, then
	\begin{align}\label{ineq:second:C3}
	\max \left\{\frac{1 + 2\beta K_V}{1 + 3 \beta K_V}, 4 \kappa_V^n \right\} 
          \leq \max \left\{\frac{1 + 2\beta K_V}{1 + 3 \beta K_V}, \frac{1}{2} \right\} 
          = \frac{1 + 2\beta K_V}{1 + 3 \beta K_V}.
	\end{align}
        Also, from the definition of $\smb$ in \eqref{def:sma:smb} and
        the first condition in \eqref{cond:n0}, it follows that
        $\smd$, defined in \eqref{def:smd}, satisfies
	\begin{align}\label{ineq:smd}
	\smd(n) \geq \frac{1}{4} \exp \left(  -  \frac{ 16 \Ua \alpha^2 M_V^2}{T^2 (1 - \sma^2)} \right) 
          \quad \mbox{ for all } n \geq n_0.
	\end{align}
	Thus, with condition \eqref{cond:eps:beta} on $\beta$, we obtain
	\begin{align}\label{ineq:third:C3}
	(1 - \smd(n))(1 + 3 \beta K_V) \leq 1 
          - \frac{1}{16} \exp \left( -  \frac{ 32 \Ua \alpha^2 M_V^2}{T^2 (1 - \sma^2)} \right) 
          \quad \mbox{ for all } n \geq n_0. 
	\end{align}
	Combining now \eqref{def:C3}, \eqref{ineq:first:C3},
        \eqref{ineq:second:C3} and \eqref{ineq:third:C3} we now 
        conclude \eqref{def:sme}.

	We turn now to show that \eqref{Wass:trhob} implies
        \eqref{Wass:points} and, consequently,
        \eqref{Wass:trhoe}. First note that, by the same arguments as
        in \eqref{Wass:points}-\eqref{Wass:points:2} we have that
        \eqref{Wass:trhob} implies
        $\Wass_{\trhob}(\nu_1 P^n, \nu_2 P^n) \leq \sme
        \Wass_{\trhob}(\nu_1 , \nu_2)$
        for all $n \geq n_0$ and $\nu_1 , \nu_2 \in \Pr(\H_\gamma)$
        with support included in $\H_\gamma$.  Now, for any
        $n \in \mathbb{N}$, we can write $n = m n_0 + k$, for some
        $m, k \in \mathbb{N}$ with $k \leq n_0 - 1$. Thus,
	\begin{align*}
		\Wass_{\trhob} (P^n (\bq_0, \cdot), P^n(\tilde \bq_0, \cdot))
		&= \Wass_{\trhob}(P^{m n_0 + k}(\bq_0, \cdot), P^{m n_0 + k}(\tilde \bq_0, \cdot)) 
		\leq \smeb^m \Wass_{\trhob} (P^k (\bq_0, \cdot), P^k (\tilde \bq_0, \cdot))  \nonumber \\
		& \leq \smeb^m  \sme(k) \trhob(\bq_0, \tilde \bq_0) 
		\leq \smeb^{\frac{n}{n_0} -1} \sme(n_0 - 1) \trhob(\bq_0, \tilde \bq_0),
	\end{align*}
	where in the last inequality we used that $\sme$ is a
        non-increasing function of $n$. Moreover, from the equivalence
        between $\trhoe$ and $\trhob$ in \eqref{equiv:trhoe:trhob}, we
        obtain
	\begin{align*}
          \Wass_{\trhoe}(P^n(\bq_0, \cdot), P^n(\tilde \bq_0, \cdot))
          &\leq \left( \frac{\max\{1, \beta\}}{\min\{1, \beta\}} \right)^{1/2} \smeb^{\frac{n}{n_0} -1}  
              \sme(n_0 - 1) \trhoe (\bq_0, \tilde \bq_0) \nonumber \\
          &\leq  \left( \frac{\max\{1, \beta\}}{\min\{1, \beta\}}\right)^{1/2} 
            \frac{ \sme(n_0 - 1)}{\smeb} \exp\left( n \log \left(\smeb^{\frac{1}{n_0}} \right) \right) \trhoe (\bq_0, \tilde \bq_0) 
            \quad \mbox{ for all } n \in \mathbb{N}.
	\end{align*}
	Therefore, with the constants
	\begin{align}\label{def:C1:C2}
		C_1 := \left( \frac{\max\{1, \beta\}}
                              {\min\{1, \beta\}} \right)^{1/2} 
                     \frac{ \sme(n_0 - 1)}{\smeb}
		\quad \mbox{ and } \quad 
		C_2 := - \log \left( \smeb^{\frac{1}{n_0}} \right),
	\end{align}
	\eqref{Wass:points} and consequently \eqref{Wass:trhoe} are
        now established.

        Finally, the second part of the proof, namely
        \eqref{ineq:obs:2}-\eqref{main:thm:CLT} under assumption
        \eqref{main:thm:Lip:const}, follow as a direct consequence of
        \cref{lem:Kantor:dual} and \cref{prop:CLT:LLN} combined with
        \cref{prop:Lip:2:obs}.
\end{proof}

\section{Implications for the Finite Dimensional Setting}
\label{sec:finite:dim:HMC}

The approach given above can be modified in a straightforward fashion
to provide a novel proof of the ergodicity of the exact HMC algorithm
in finite dimensions.  We detail this connection in this section. We
abuse notation and use the same terminology for the analogous
constants and operators from the infinite-dimensional case introduced
in the previous sections.

We take our phase space to be $\H = \RR^k$, $k \in \mathbb{N}$,
endowed with the Euclidean inner product and norm, which are denoted
by $\langle \cdot, \cdot \rangle$ and $|\cdot|$,
respectively. Similarly to \eqref{eq:mu} above we fix a target
probability measure of the form
\begin{align}\label{eq:mu:FD}
	\mu(d\bq ) \propto \exp( - U(\bq)) \mu_0(d\bq) 
  \quad \text{ with } \mu_0 = \mathcal{N}(0,\cC), 
\end{align}
where $\cC$ is a symmetric strictly positive-definite covariance
matrix.  Here we aim to sample from $\mu$ using the dynamics
\begin{align}
  \frac{d \bq}{d t} = \mM^{-1} \bp
  \quad
  \frac{d \bp}{d t} = -\cC^{-1} \bq - DU(\bq)
  \label{eq:FD:H:dynam}
\end{align}
corresponding to the Hamiltonian
\begin{align}
  H(\bq,\bp)
  = \langle \cC^{-1} \bq, \bq  \rangle + U(\bq) + \langle \mM^{-1} \bp, \bp \rangle,
  \label{eq:FD:Ham}
\end{align}
where $\mM$ is a user-specified `mass matrix' which we suppose to be
symmetric and strictly positive definite; and $U: \RR^k \to \RR$ is a
$C^2$ potential function. Let us denote by $\lM$ and $\LM$ the
smallest and largest eigenvalues of $\mM$. Analogously, let $\lC$ and
$\LC$ be the smallest and largest eigenvalues of $\cC$.

We impose the following conditions on the potential function $U$ (cf. \cref{B123}
above):

\begin{assumption}
  \label{ass:potential:FD}
  \begin{itemize}
  \item[(F1)] There exists a constant $\Ua \geq 0$ such that
    \begin{align} \label{Hess:bound:FD}
             | D^2 U(\bff) | \leq \Ua
             \quad 
            \text{ for any } \bff \in \RR^k.
    \end{align}
  \item[(F2)] There exist constants $\Ub > 0$ and $\Uc \ge 0$ such that
              \begin{align} \label{dissip:FD}
            |\mM^{-1/2} \cC^{-1/2} \bff|^2 + \langle \bff,\mM^{-1} D U(\bff) \rangle \ge \Ub |\mM^{-1/2} \cC^{-1/2} \bff|^2 - \Uc
            \quad  
            \text{ for any } \bff \in \RR^k.
          \end{align}
  \end{itemize}
\end{assumption}
Note that under \eqref{Hess:bound:FD}, $U$ is globally Lipschitz so
that \eqref{eq:FD:H:dynam} yields a well defined dynamical system on
$C^1(\RR, \RR^k)$ as above in \cref{prop:well:posed}. Furthermore,
similarly as in \cref{rmk:B123:Consequences}, we have:
\begin{enumerate}[(i)]
    \item From \eqref{Hess:bound:FD}, it follows that
        \begin{align}\label{bound:DU:FD}
            |DU(\bff)| \leq \Ua |\bff| + \Uz \quad \mbox{ for every } \bff \in \RR^k.
        \end{align}
        where $\Uz = |DU(0)|$.
      \item If $|DU(\bff)| \leq \Ud |\bff| + \Ue$ for some
        $\Ud \in [0, \lM (\LM \LC)^{-1})$ and $\Ue \geq 0$, then
        \eqref{dissip:FD} follows.
    \item Assumptions $(F1)$ and $(F2)$ imply that 
        \begin{align}\label{ineq:tUa:tUb}
            \Ub \leq 1 + \LM \LC \lM^{-1} \Ua.
        \end{align}
\end{enumerate}

Fixing an integration time $T > 0$, and under the given conditions on
$\cC, \mM$ and $U$ in \eqref{eq:FD:H:dynam} we have a well-defined Feller
Markov transition kernel defined as
\begin{align}
  P(\bq_0, A) = \Prb( q_T(\bq_0, \bp_0) \in A)
  \label{eq:finite:dim:TK_1}
\end{align}
for any $\bq_0 \in \RR^k$ and any Borel set $A \subset \RR^k$, where
\begin{align}
  \bp_0 \sim N(0,\mM).
  \label{eq:finite:dim:TK_2}
\end{align}
Here, following previous notation, $q_T(\bq_0, \bp_0)$ is the solution of
\eqref{eq:FD:H:dynam} at time $T$ starting from the initial position $\bq_0 \in \RR^k$
and momentum $\bp_0 \in \RR^k$.  The $n$-fold iteration of the kernel $P$ is denoted as $P^n$. 

As in \cref{thm:weak:harris}, we measure the convergence of $P^n$
using a suitable Wasserstein distance. In this case, we take
\begin{align}
  \tilde\rho(\bq, \tilde \bq) = \sqrt{\rhoep( \bq, \tilde \bq) (1 + V(\bq) + V(\tilde \bq))} 
  \quad \text{ where } \quad
  \rhoep( \bq, \tilde \bq) 
  = \frac{| \bq - \tilde \bq |}{\varepsilon} \wedge 1
  \label{eq:semi:met:FD}
\end{align}
and $V$ is a Foster-Lyapunov function defined as either
$V(\bq) = V_{1,i}(\bq) = |\bq|^i$, $i \in \NN$, or as
$V(\bq) = V_{2,\eta}(\bq) = \exp(\eta |\bq|^2)$, with $\eta > 0$
satisfying
\begin{align}\label{cond:eta:FD}
  \eta < \left[ 2 \Tr(\mM) \left( \frac{67}{8}T^2 
  + \frac{32}{\Ub (\LM \LC)^{-1}} \right) \lM^{-2} \right]^{-1}.
\end{align}

We then consider the corresponding Wasserstein distance
$\Wass_{\tilde\rho}$ and prove the theorem below concerning the exact
HMC kernel $P$.

\begin{Thm}\label{thm:main:FD}
  Consider the Markov kernel $P$ defined as
  \eqref{eq:finite:dim:TK_1}, \eqref{eq:finite:dim:TK_2} from the
  dynamics \eqref{eq:FD:H:dynam}.  We suppose that $\mM$ and $\cC$ in
  \eqref{eq:FD:H:dynam} are both symmetric and strictly positive
  definite and we assume that the potential function $U$ satisfies
  \cref{ass:potential:FD}. In addition, we impose the following 
  condition on the integration time $T >0$:
  \begin{align}\label{fin:dim:cond:T}
    T \leq \min \left\{ \frac{1}{\left[ 2 \lM^{-1}(\lC^{-1} + \Ua) \right]^{1/2}}, \frac{\Ub^{1/2} (\LM \LC)^{-1/2}}{2 \sqrt{6} \lM^{-1} (\lC^{-1} + \Ua)} \right\},
  \end{align}
  where $\lM$ and $\LM$ denote the smallest and largest eigenvalues of
  $\mM$, while $\lC$ and $\LC$ denote the smallest and largest
  eigenvalues of $\cC$, respectively.
 
  Then $P$ has a unique ergodic invariant measure given by $\mu$ in
  \eqref{eq:mu:FD}.  Moreover, $P$ satisfies the following spectral
  gap condition with respect to the Wasserstein distance
  $\Wass_{\tilde \rho}$ associated to $\tilde \rho$ defined in
  \eqref{eq:semi:met:FD}: For all $\nu_1, \nu_2$ Borel probability
  measures on $\RR^k$,
  \begin{align}
    \Wass_{\trhoe}(\nu_1 P^n, \nu_2 P^m) \leq
       C_1 e^{- C_2 n} \Wass_{\trhoe}(\nu_1, \nu_2) \quad \mbox{ for all } n \in \mathbb{N},
    \label{eq:Wass:conv:FD}
  \end{align}
  where the constants $C_1, C_2, \varepsilon >0$ are independent of
  $\nu_1, \nu_2$ and $k$, and can be given explicitly as depending exclusively on $\Ua$, $\Ub$, $\Uc$, $T$, $\mM$ and $\cC$.
\end{Thm}  
\begin{Rmk}
  Similarly as in \cref{thm:weak:harris}, we can also show that
  \eqref{eq:Wass:conv:FD} implies a convergence result with respect to
  suitable observables as in \eqref{ineq:obs:2}, as well as a strong
  law of large numbers and a central limit theorem analogous to
  \eqref{main:thm:LLN}-\eqref{main:thm:CLT}.
\end{Rmk}
\begin{proof}
  The proof follows very similar steps to the results from Sections
  \ref{sec:apriori}, \ref{sec:f:l:bnd}, \ref{sec:Ptwise:contract:MD}
  and \ref{sec:Proof:Main}, so we only point out the main differences.

From \eqref{eq:FD:H:dynam}, it follows that
\begin{align*}
    \frac{d^2 \bq}{dt^2} = - \mM^{-1}\cC^{-1}\bq - \mM^{-1} D U(\bq),
\end{align*}
so that, after integrating with respect to $t \in [0,T]$ twice, we have
\begin{align}\label{eq:qt:FD}
    \bq_t - (\bq_0 + t \mM^{-1} \bp_0) = - \int_0^t \int_0^s \left( \mM^{-1} \cC^{-1} \bq_{\tau} + \mM^{-1} D U(\bq_{\tau}) \right) d\tau ds
\end{align}
Using that 
\begin{align*}
  | \mM^{-1} \bff| \leq \lM^{-1} |\bff| \quad \mbox{ and }
  \quad |\cC^{-1}\bff| \leq \lC^{-1} |\bff| \quad  \mbox{ for every }\bff \in \RR^k,
\end{align*}
together with \eqref{bound:DU:FD} and the condition
$T \leq [\lM^{-1}(\lC^{-1} + \Ua)]^{-1/2}$, one obtains, analogously
to \eqref{sup:qs} and \eqref{sup:vs},
\begin{align}\label{bound:q:FD}
  \sup_{t \in [0,T]} |\bq_t - (\bq_0 + t \mM^{-1} \bp_0)| 
  \leq \lM^{-1} (\lC^{-1} + \Ua) 
  T^2 \max \left\{ |\bq_0|, |\bq_0 + T \mM^{-1}\bp_0| \right\} + \lM^{-1} \Uz T^2
\end{align}
and
\begin{align}
  \sup_{t \in [0,T]} |\bp_t - \bp_0|
  \leq& (\lC^{-1} + \Ua)t
  \left[ 1 + \lM^{-1}(\lC^{-1} + \Ua)t^2\right]
        \max \left\{ |\bq_0|, |\bq_0 + T \mM^{-1}\bp_0| \right\}
  \notag\\
      &+ \Uz t \left[ 1 + \lM^{-1}(\lC^{-1} + \Ua) t^2 \right].
        \label{bound:p:FD}
\end{align}
Moreover, analogously to \eqref{eq:Hgamma:lip:bnd}, we obtain that for
every
$(\bq_0, \bp_0), (\tilde \bq_0, \tilde \bp_0) \in \RR^k \times \RR^k$,
\begin{multline}
    \sup_{t \in [0,T]} |\bq_t(\bq_0, \bp_0) - \bq_t(\tilde \bq_0, \tilde \bp_0) 
        - [(\bq_0 - \tilde \bq_0) + t \mM^{-1}(\bp_0 - \tilde \bp_0)] | \nonumber \\
    \leq  \lM^{-1} (\lC^{-1} + \Ua) T^2 \max 
          \left\{ |\bq_0 - \tilde \bq_0|, |
            \bq_0 - \tilde \bq_0 + t \mM^{-1} (\bp_0 - \tilde \bp_0)| \right\}.
\end{multline}
In particular, if $\tilde \bp_0 = \bp_0 + \mM (\bq_0 - \tilde \bq_0) T^{-1}$ then
\begin{align}\label{ineq:contr:FD}
    \sup_{t \in [0,T]} |\bq_t(\bq_0, \bp_0) - \bq_t(\tilde \bq_0, \tilde \bp_0)| 
  \leq \lM^{-1} (\lC^{-1} + \Ua) T^2 |\bq_0 - \tilde \bq_0| 
  \leq \frac{1}{2} |\bq_0 - \tilde \bq_0|.
\end{align}

We also show that $V(\bq) = |\bq|^i$ with $i \geq 1$ or
$V(\bq) = \exp(\eta |\bq|^2)$, with $\eta > 0$ satisfying
\eqref{cond:eta:FD}, all verify a Foster-Lyapunov structure as in
\cref{def:Lyap}. The proof follows as in Proposition \ref{prop:FL},
with the difference starting from \eqref{eq:der:q:v}, which is now
written as
\begin{align}\label{eq:der:q:p}
    \frac{d}{ds} \langle \bq_s, \mM^{-1} \bp_s \rangle 
  = |\mM^{-1} \bp_s|^2 - |\mM^{-1/2} \cC^{-1/2} \bq_s|^2 
  - \langle \bq_s, \mM^{-1} D U (\bq_s) \rangle.
\end{align}
Using now $(F2)$ from \cref{ass:potential:FD} and the inequalities
\begin{align*}
    |\mM^{-1/2} \bff| \geq \LM^{-1/2}|\bff| 
  \quad \mbox{ and } \quad |\cC^{-1/2} \bff| \geq \LC^{-1/2} |\bff| 
  \quad \mbox{ for all } \bff \in \RR^k,
\end{align*}
we obtain from \eqref{eq:der:q:p} that
\begin{align}\label{ineq:qT:FD}
    |\bq_T|^2 \leq |\bq_0|^2 + 2 T \langle \bq_0, \mM^{-1} \bp_0 \rangle 
  + 2 \int_0^T \int_0^s \left[ \lM^{-2} |\bp_\tau|^2 - \Ub (\LM \LC)^{-1} |\bq_\tau|^2 
  + \Uc \right] d\tau ds.
\end{align}
Then, with \eqref{ineq:tUa:tUb}, the a priori bounds
\eqref{bound:q:FD}-\eqref{bound:p:FD} and the fact that
$2 \lM^{-1} (\lC^{-1} + \Ua) T^2 \leq 1$ from hypothesis
\eqref{fin:dim:cond:T}, we arrive at
\begin{align}
    |\bq_T|^2 
    \leq& \left( 1 + \frac{3}{2} \lM^{-2} (\lC^{-1} + \Ua)^2 T^4 
        - \frac{\Ub}{8} (\LM \LC)^{-1} T^2 \right) |\bq_0|^2  
        +  2T \langle \bq_0, \mM^{-1} \bp_0 \rangle  
  \notag\\
        &+  \frac{67}{8} \lM^{-2} T^2 |\bp_0|^2  
        + \frac{3}{2} \Uz^2 \lM^{-2} T^4 
        + \frac{\Uz^2}{6} \lM^{-2} T^4 + \Uc T^2.
\label{ineq:qT:fin:dim}
\end{align}
From the second condition in hypothesis \eqref{fin:dim:cond:T} it
follows that
$(3/2)\lM^{-2} (\lC^{-1} + \Ua)^2 T^4 \leq (\Ub/16) (\LM \LC)^{-1}
T^2$,
so that after taking expected values in \eqref{ineq:qT:fin:dim} we
obtain
\begin{align*}
    \bE |\bq_T|^2 
    \leq \exp \left( - \frac{\Ub}{16} (\LM \LC)^{-1} T^2 \right) |\bq_0|^2 
        + \left( \frac{67}{8} \lM^{-2} \Tr(\mM)  
           + \frac{5}{3} \lM^{-2} \Uz^2 T^2  + \Uc \right) T^2.
\end{align*}
Now proceeding analogously as in
\eqref{quad:Lyap:n}-\eqref{ineq:exp:Lyap:2}, we obtain that for
$V: \RR^k \to \RR$ given either as $V(\bq) = |\bq|^i$, $i \in \NN$, or
$V(\bq) = \exp(\eta|\bq|^2)$, with $\eta > 0$ satisfying
\eqref{cond:eta:FD}, there exist constants $\kappa_V \in [0,1)$ and
$K_V > 0$ such that
\begin{align}\label{ineq:FL:FD}
    P^n V(\bq_0) \leq \kappa_V^n V(\bq_0) +  K_V \quad \mbox{ for all } \bq_0 \in \RR^k, \mbox{ for all } n \in \mathbb{N},
\end{align}
i.e. these are Lyapunov functions for $P$.

Let $(\RR^k)^n$ denote the product of $n$ copies of $\RR^k$ and let
$\mN(0, \mM)^{\otimes n}$ denote the product of $n$ copies of
$\mN(0, \mM)$. Analogously to Section \ref{sec:Ptwise:contract:MD},
given $\bq_0 \in \RR^k$ and a sequence
$\{\bp_0^{(j)}\}_{j \in \mathbb{N}}$ of i.i.d. draws from
$\mathcal{N}(0, \mM)$, we denote
$\bP_0^{(n)} = (\bp_0^{(1)}, \ldots, \bp_0^{(n)})$, for all
$n \in \mathbb{N}$, and take $Q_n(\bq_0, \cdot): (\RR^k)^n \to \RR^k$,
according to
\begin{align*}
   Q_1(\bq_0, \bp_0^{(1)}) = \bq_T(\bq_0, \bp_0^{(1)}), 
  \quad 
  Q_n(\bq_0,\bP_0^{(n)}) = \bq_T(Q_{n-1}(\bq_0,\bP_0^{(n-1)}),\bp_0^{(n)}) 
  \quad \mbox{ for all } n \geq 2.
\end{align*}
Similarly given any $\bq_0, \tilde{\bq}_0 \in \RR^k$ we take
$\tQ_n(\bq_0, \tilde \bq_0, \cdot): (\RR^k)^n \to \RR^k$ to be the
random variables starting from
\begin{align*}
  \tQ_1(\bq_0, \tilde \bq_0, \bp_0^{(1)}) 
  = \bq_T(\tilde \bq_0, \bp_0^{(1)} +  T^{-1} \mM(\bq_0 - \tilde \bq_0)),
\end{align*}
then defined for each integer $n \geq 2$ as
\begin{align*}
    \tQ_n(\bq_0, \tilde \bq_0, \bP_0^{(n)}) 
  = \bq_T(\tQ_{n-1}(\bq_0, \tilde \bq_0, \bP_0^{(n-1)}), \bp_0^{(n)} + \cS_n(\bP_0^{(n-1)})) 
\end{align*}
with
\begin{align}\label{def:Phin:FD}
    \cS_n(\bP_0^{(n-1)}) 
  = T^{-1} \mM \left[ Q_{n-1}(\bq_0, \bP_0^{(n-1)}) - \tQ_{n-1} (\bq_0, \tilde \bq_0, \bP_0^{(n-1)}) \right] 
  \quad \mbox{ for all } n \geq 2.
\end{align}
We also denote 
\begin{align*}
  \bcS_n(\bP_0^{(n)}) = (\cS_1, \cS_2(\bP_0^{(1)}), \ldots, \cS_n(\bP_0^{(n-1)})), 
  \quad \mbox{ with } \cS_1 = T^{-1}(\bq_0 - \tilde \bq_0),
\end{align*}
and $\bPsi_n(\bP_0^{(n)}) = \bP_0^{(n)} + \bcS_n(\bP_0^{(n)})$. Thus, by
using inequality \eqref{ineq:contr:FD} $n$ times iteratively, we
obtain that
\begin{align}\label{ineq:contr:FD:2}
    |Q_n(\bq_0,\bP_0^{(n)}) - \tQ_n(\bq_0, \tilde \bq_0,\bP_0^{(n)})| 
       \leq \frac{1}{2^n} |\bq_0 - \tilde \bq_0|,
\end{align}
for all  $\bP_0^{(n)} \in (\RR^k)^n$. 

Let $\sigma_n = \text{Law} (\bP_0^{(n)}) = \mN(0,\mM)^{\otimes n}$ and
$\tilde \sigma_n = \text{Law}(\bPsi_n(\bP_0^{(n)})) = \bPsi_n^\ast
\nu_n$.
Analogously as in Proposition \ref{prop:local_contr} and Proposition
\ref{prop:smallness}, we obtain that the distance-like function $\rho$
defined in \eqref{eq:semi:met:FD} satisfies contractivity and
smallness properties with respect to the Markov operator $P^n$ for $n$
sufficiently large. Here, the main difference lies in the estimate of
Kullback-Leibler Divergence $\KL (\tilde{\sigma}_n | \sigma_n)$,
\eqref{eq:KL:div}.  Proceeding similarly as in
\eqref{Girsanov}-\eqref{ineq:KL}, we arrive at
\begin{align*}
  \KL(\tilde{\sigma}_n | \sigma_n) 
  \leq \frac{1}{2} \sum_{j=1}^n \bE |\mM^{-1/2} \cS_j (\cdot)|^2.
 \end{align*}
 Using \eqref{ineq:contr:FD:2}, it follows that for every
 $j \in \{1, \ldots, n\}$ and $\bP_0^{(j-1)} \in (\RR^k)^{(j-1)}$
\begin{align*}
	|\mM^{-1/2}\cS_j(\bP_0^{(j-1)})|^2 
	\leq& \lM^{-1} |\cS_j(\bP_0^{(j-1)})|^2 
	\leq \lM^{-1} \LM^2 T^{-2} |Q_{j-1} (\bq_0)(\bP_0^{(j-1)}) 
             - \tQ_{j-1} (\bq_0, \tilde \bq_0)(\bP_0^{(j-1)})|^2  \\
	\leq& \frac{\lM^{-1} \LM^2 T^{-2}}{2^{(j-1)2}} |\bq_0 - \tilde \bq_0|^2,
\end{align*}
where in the second inequality we used that $|\mM \cdot|^2 \leq \LM^2 |\cdot|^2$.
Hence,
\begin{align}\label{ineq:KL:FD}
	\KL(\tilde{\sigma}_n | \sigma_n) 
	\leq \frac{\lM^{-1} \LM^2 T^{-2}}{2} |\bq_0 - \tilde \bq_0|^2 \sum_{j=1}^n \frac{1}{2^{(j-1)2}} 
	\leq \frac{4 \LM^2}{\lM T^2} |\bq_0 - \tilde \bq_0|^2.
\end{align}
By using \eqref{ineq:KL:FD}, one obtains analogously as in Proposition \ref{prop:local_contr} that for every $n \in \mathbb{N}$ and for every $\bq_0, \tilde \bq_0 \in \RR^k$ such that $\rhoep(\bq_0, \tilde \bq_0) < 1$, we have
\begin{align}\label{ineq:contr:Wass:FD}
	\Wass_{\rhoep}(P^n(\bq_0, \cdot), P^n(\tilde \bq_0, \cdot)) \leq \left( \frac{1}{2^n} + \frac{\sqrt{2} \LM \varepsilon}{\lM^{1/2} T } \right) \rhoep(\bq_0, \tilde \bq_0).
\end{align} 
Moreover, analogously as in Proposition \ref{prop:smallness}, we
obtain that, given $M \geq 0$, for every
$\bq_0, \tilde \bq_0 \in A := \{ \bq \in \RR^k \,:\, |\bq| \leq M \}$,
it holds:
\begin{align}\label{ineq:small:FD}
  \Wass_{\rhoep} (P^n(\bq_0, \cdot), P^n(\tilde \bq_0, \cdot)) 
  \leq 1 - \frac{1}{2} \exp \left( -\frac{16 \LM^2 M^2}{\lM T^2} \right) 
  + \frac{M}{2^{n-1} \varepsilon}.
\end{align}

The remaining portion of the proof now follows as for Theorem
\ref{thm:weak:harris}, by combining \eqref{ineq:FL:FD},
\eqref{ineq:contr:Wass:FD} and \eqref{ineq:small:FD}.
\end{proof}

\begin{Rmk}\label{rmk:FD}
  From condition the \eqref{fin:dim:cond:T} on the integration time $T$,
  we see how the upper bound could potentially degenerate to zero in
  case the eigenvalues of $\cC$ and/or the eigenvalues of $\mM$ decay
  to zero as the dimension of $\mathbb{R}^k$ increases. Moreover, if
  the eigenvalues of $\mM$ decrease to zero (i.e. $\lM \to 0$) or
  increase to infinity (i.e. $\LM \to \infty$) with respect to $k$,
  then, for fixed $n$, $\varepsilon$ and $T$, the upper bound in
  \eqref{ineq:contr:Wass:FD} increases to infinity, and the first two
  terms in the upper bound in \eqref{ineq:small:FD} increase to
  $1$. This would imply that the convergence rate in
  \eqref{eq:Wass:conv:FD}, which is directly proportional to the upper
  bounds in \eqref{ineq:contr:Wass:FD}-\eqref{ineq:small:FD} and
  inversely proportional to $T$, would become `slower' as $k$
  increases. In other words, the number $n$ of iterations necessary
  for the distance between $\nu_1 P^n$ and $\nu_2 P^n$ to decay within
  a given $\delta > 0$ would increase with the dimension $k$. This
  type of behavior is commonly known as the `curse of dimensionality'.
    
  A natural choice for the mass matrix $\mM$ to avoid such unwanted
  behavior is given by $\mM = \cC^{-1}$ -- this is the idea behind
  preconditioning in \cite{BePiSaSt2011} which leads us to consider
  \eqref{eq:dynamics:xxx} in the infinite dimensional formulation. In
  this preconditioned case, one could use that $\lM = \LC^{-1}$ and
  $\LM = \lC^{-1}$ directly in \eqref{fin:dim:cond:T} to obtain
    \begin{align}\label{cond:T:rough}
        T \leq \min \left\{ \frac{1}{[ 2 \LC (\lC^{-1} + \Ua) ]^{1/2}},  
       \frac{\Ub^{1/2} (\lC^{-1} \LC )^{-1/2}}
           {2 \sqrt{6} \LC (\lC^{-1} + \Ua)}  \right\},
    \end{align}
    where the upper bound actually still degenerates to zero in case
    $\lC \to 0$ as $k \to \infty$ (corresponding to the trace-class
    assumption on $\cC$ in the infinite-dimensional case). However,
    the inequalities that lead to the condition on $T$ as in
    \eqref{cond:T:rough} would in fact be a rough overestimate in this
    case. Indeed, for $\mM = \cC^{-1}$, the term
    $\mM^{-1}\cC^{-1} \bq_{\tau}$ in \eqref{eq:qt:FD} is simply equal
    to $\bq_\tau$ and thus we no longer estimate from above by
    $\lM^{-1}\lC^{-1} |\bq_\tau|$ as in \eqref{bound:q:FD}. Similarly,
    the term $|\mM^{-1/2} \cC^{-1/2} \bq_s|^2$ in \eqref{eq:der:q:p}
    is simply $|\bq_s|^2$ and thus no longer estimated from below by
    $(\LM \LC)^{-1} |\bq_\tau|^2$ as in \eqref{ineq:qT:FD}. With these
    changes, T is required to satisfy instead
    \begin{align*}
        T \leq \min \left\{ \frac{1}{[2 (1 + \LC \Ua)]^{1/2}},  
                \frac{\Ub^{1/2}}{2 \sqrt{6} (1 + \LC \Ua)}  \right\},
    \end{align*}
    which is consistent with condition \eqref{eq:time:restrict:basic}
    for $\LC = \lambda_1$ (when $\gamma = 0$), and thus independent of
    $k$ when $\LC$ is uniformly bounded with respect to $k$.
    
    On the other hand, replacing $\LM$ with $\lC^{-1}$ in
    \eqref{ineq:contr:Wass:FD} and \eqref{ineq:small:FD}, we see that
    the same unwanted behavior is not removed here when $\lC \to 0$ as
    $k \to \infty$; i.e. the convergence rate would still degenerate
    with the dimension $k$. This emphasizes the need for considering
    `shifts' in the momentum (or velocity) paths for the modified
    process $\tQ_n(\bq_0, \btq_0, \cdot)$,
    $\bq_0, \btq_0 \in \mathbb{R}^k$, that are restricted to a fixed
    number of directions in $\mathbb{R}^k$, for every $k$, as done in
    \eqref{def:Phin} through the projection operator $\Pi_N$, with $N$
    sufficiently large but fixed (cf. \eqref{def:Phin:FD}).
\end{Rmk}

\section{Application for the Bayesian estimation 
  of divergence free
  Flows from a passive scalar}
  \label{sec:bayes:AD}

In this section we establish some results concerning the degree of
applicability of \cref{thm:weak:harris} to the PDE inverse problem
of estimating a divergence free flow from a passive scalar as we
described above in the introduction,  cf. \eqref{eq:ad:eqn},
\eqref{eq:gen:linear:form}, \eqref{eq:post:passive:scal}.

For this purpose, according to the conditions required in \cref{B123},
we wish to establish suitable bounds on $U$, $DU$ and $D^2U$.  Of
course such bounds are expected to depend crucially on the form of the
observation operator $\Obs$.  Here, adopting the notations
$ U^{\xi} = \langle DU, \xi \rangle$ and
$U^{\xi,\txi} = \langle D^2U\xi, \txi \rangle$ for directional
derivatives of $U$ with respect to vectors $\xi, \txi$ in the phase space, we
have that
\begin{align}
  U^{\xi} (\bq) 
  = -2\langle \GP (\data - \Obs(\theta(\bq))), 
                      \GP \Obs(\psi^\xi(\bq)) \rangle 
  \label{eq:DD:U:xi}
\end{align}
and
\begin{align}
     U^{\xi, \txi} (\bq) 
   = 2\langle 
      \GP \Obs (\psi^{\txi}(\bq)), 
      \GP \Obs(\psi^\xi(\bq)) 
    \rangle 
     -2\langle \GP (\data - \Obs(\theta(\bq))), 
                       \GP \Obs(\psi^{\xi, \txi}(\bq)) \rangle
  \label{eq:DD:U:xi:txi}
\end{align}
where $\psi^\xi(\bq) = \psi^\xi(t;\bq)$ obeys 
\begin{align}
  \pd_t \psi^\xi + \bq \cdot \nabla \psi^\xi 
  = \kappa \Delta \psi^\xi
     - \xi \cdot \nabla \theta(\bq),
  \quad \psi^\xi(0; \bq) =0
  \label{eq:grad:the:xi}
\end{align}
and $\psi^{\xi, \txi}(\bq) = \psi^{\xi, \txi}(t; \bq)$ satisfies
\begin{align}
  \pd_t \psi^{\xi, \txi} + \bq \cdot \nabla \psi^{\xi, \txi}
  = \kappa \Delta\psi^{\xi, \txi}
     - \txi \cdot \nabla \psi^\xi
     - \xi \cdot \nabla \psi^{\txi},
  \quad \psi^{\xi, \txi} (0; \bq) = 0,
  \label{eq:grad:the:xitxi}
\end{align}
for any suitable $\xi, \txi$.

\subsubsection{Mathematical Setting of the Advection Diffusion
  Equation, Associated Bounds}
\label{sec:AD:Math}
In order to place \eqref{eq:post:passive:scal} in a rigorous
functional setting we adapt some results from
\cite{borggaard2018Consistency, borggaard2018bayesian}.  In view of
\eqref{eq:grad:the:xi}, \eqref{eq:grad:the:xitxi} we consider a
slightly more general version of \eqref{eq:ad:eqn} where we include an
external forcing term $f: [0,T] \times \TT^2 \rightarrow \RR$, namely,
\begin{align}
  \pd_t \phi + \bq \cdot \nabla \phi  = \kappa \Delta \phi + f, 
 \quad \phi(0) = \phi_0.
  \label{eq:ad:forced}
\end{align}
Specially, we need to estimate terms appearing in the gradient and
Hessian of $U$ involving solutions of \eqref{eq:ad:forced} with
certain forcing terms; cf. \eqref{eq:grad:the:xi},
\eqref{eq:grad:the:xitxi} below.

We adopt the notation $H^s(\TT^2)$ for the Sobolev space of periodic
functions with $s \geq 0$ derivatives in $L^2$.  Here we denote
$\FD^s = (- \Delta)^{s/2}$.  Thus, the associated $H^s(\TT^2)$ norms are given by $\| \cdot \|_s = \| \FD^s \cdot \|_0$
where $\| \cdot \|_0$ is the usual $L^2(\TT^2)$ norm.  We also make
use of the negative Sobolev spaces $H^{-s}(\TT^2)$ for $s \geq 0$ defined via
duality relative to $L^2(\TT^2)$ with the norms reading as
\begin{align}
  \| f \|_{-s} = \sup_{\|\xi\|_{s} = 1} \langle f, \xi \rangle
  \label{eq:neg:sob:space}
\end{align}
where $\langle \cdot, \cdot \rangle$ is the usual duality pairing so
that $\langle f, \xi \rangle = \int_{\TT^2} f \xi dx$ when
$f \in L^2(\TT^2)$. All other norms are denoted as $\| \cdot \|_X$
where $X$ is the associated space i.e. $L^\infty$.  We abuse notation
and use the same naming convention $H^s(\TT^2)$ and associated norm
$\| \cdot \|_s$ for periodic, divergence free vector fields with $s$
derivatives in $L^2(\TT^2)$.

We have the following proposition adapted from
\cite{borggaard2018Consistency}:
\begin{Prop}[Well-Posedness and Continuity of the solution map for
  \eqref{eq:ad:forced}]
  \label{def:adr_weak}
  \mbox{}
\begin{itemize}
\item[(i)] Fix any $s \geq 0$ and suppose that $\bq \in H^{s}(\TT^2)$,
  $\phi_0 \in H^{s}(\TT^2) \cap L^\infty(\TT^2)$ and
  $f \in L^2_{loc}([0,\infty);$ $H^{s -1}(\TT^2))$.  Then there exists a
  unique $\phi = \phi(\bq, \phi_0, f)$ such that
  \begin{align}
    \phi &\in L^2_{loc}([0,\infty); H^{s+1}(\TT^2)) 
        \cap L^\infty([0,\infty); H^{s}(\TT^2)),
     \quad
    \frac{\partial \phi}{\partial t} 
           \in L^2_{loc}([0,\infty); H^{s-1}(\TT^2))
           \label{eq:AD:gen:reg:1}
  \end{align}
  so that in particular\footnote{See e.g. \cite[Lemma 3.1.2]{Temam2001}.}
  \begin{align*}
  \phi \in C([0,\infty); H^{s}(\TT^2))
  \end{align*}
  and where $\phi$ solves~\eqref{eq:ad:forced} at least weakly.
  Additionally $\phi$ maintains the bounds
  \begin{align}
    &\frac{d}{dt} \| \phi \|^2_0 + 2\kappa \| \phi \|_1^2 = 2 \int
    f \phi dx,
        \label{eq:L2:balance}\\
    &\sup_{t \in [0,t^*]} \| \phi(t) \|_{L^\infty}
        \leq \| \phi_0 \|_{L^\infty} + \int_0^{t^*} \|f\|_{L^\infty} dt,
     \text{ for any } t^* > 0.
      \label{eq:Linfty:bnd}
    \end{align}
    When $s > 0$ we have
    \begin{align}
   \frac{d}{dt} \| \phi \|^2_s + \kappa \| \phi \|_{s+1}^2 
        \leq c \| \phi \|^2_s \| \bq\|^{a}_s  + 2 \int \FD^s f   \FD^s
      \phi \, dx
        \label{eq:Hs:balance}
  \end{align}
  where the contants $c = c(\kappa, s)$, $a = a(\kappa, s)$ are
  independent of $\bq$.
\item[(ii)] Let $\phi^{(j)} =
  \phi(\bq_j,\phi_{0,j},f_j)$ for $j = 1,2$ be two solutions of \eqref{eq:ad:forced}
  corresponding to data $\bq_j,\phi_{0,j},f_j$ satisfying the conditions
  in part (i).   Then, taking
  $\psi = \phi^{(1)} - \phi^{(2)}$, $\bp = \bq_1 - \bq_2$, we have
  \begin{align}
    \frac{d}{dt} \|\psi  \|_0^2 +\kappa \|\psi  \|_1^2
    \leq  c \| \bp\|_0^2 \|\phi^{(1)}\|_{L^\infty}^2
          + c\| f_1 -  f_2\|^2_{-1}
    \label{eq:L2:cont:est:AD}
  \end{align}
  with $c = c(\kappa)$ independent of $\bq_1, \bq_2$. Furthermore, in
  the case when $s > 0$ we have
  \begin{align}
     \frac{d}{dt}\|\psi  \|_s^2 +\kappa \|\psi  \|_{s+1}^2
    \leq c \|\psi  \|_s^2 \|\bq_1\|_s^{a}
    + c \|\bp\|_s^{2} \|  \phi^{(2)} \|_{s+1}^2 
    + c \|f_1 - f_2\|_{s-1}^2,
        \label{eq:Hs:cont:est:AD}
  \end{align}
  where the constants $c = c(\kappa, s)$, $a = a(\kappa, s)$ are again independent of
  $\bq_1, \bq_2$.
\end{itemize}
\end{Prop}

\cref{def:adr_weak} immediately yields quantitate bounds on
derivatives of $\theta(\bq)$ in its advecting flow $\bq$ which solve
\eqref{eq:grad:the:xi}, \eqref{eq:grad:the:xitxi}.  In turn these
bounds provide the quantitative foundation for the estimates on $DU$
and $D^2U$ below in \cref{prop:DU:DsqU:Bnds},
\cref{cor:DU:DsqU:Bnd:w:C}.

\begin{Prop}
  \label{prop:grad:xi:apriori}
  Fix any $s > 0$ and $\theta_0 \in L^\infty(\TT^2) \cap H^s(\TT^2)$.
  Then the map from $H^s(\TT^2)$ to $C([0,\infty); H^s(\TT^2))$ that
  associates to each $\bq \in H^s(\TT^2)$ the corresponding solution
  $\theta(\bq) := \theta(\cdot; \bq, \theta_0)$ of \eqref{eq:ad:eqn}
  is a $C^2$ function.  Denote $\psi^\xi(\bq)$ and
  $\psi^{\xi,\txi}(\bq)$ as the directional derivatives of $\theta$ in
  the directions $\xi, \txi \in H^s(\TT^2)$.  Then $\psi^{\xi}(\bq)$
  and $\psi^{\xi,\txi}(\bq)$ obey \eqref{eq:grad:the:xi} and
  \eqref{eq:grad:the:xitxi}, respectively, with regularity
  \eqref{eq:AD:gen:reg:1} in the sense of
  \cref{def:adr_weak}. Furthermore,
  \begin{itemize}
  \item[(i)] For any $\bq, \xi \in H^s(\TT^2)$, $t^* > 0$ we have
\begin{align}
  \sup_{t\leq t^*} \| \psi^\xi(t;\bq)\|_0^2 
  + \int_0^{t^*}\| \psi^\xi(t;\bq)\|_1^2 dt
  \leq c t^* \|\xi\|_0^2 
  \label{eq:psi:xi:Bnd:L2}
\end{align}
and
\begin{align}
    \sup_{t\leq t^*} \| \psi^\xi(t;\bq)\|_0^2 
  + \int_0^{t^*}\| \psi^\xi(t;\bq)\|_1^2 dt
  \leq c \|\xi\|_s^2 
  \label{eq:psi:xi:Bnd:L2:b}
\end{align}
where $c = c(\|\theta_0\|_{L^\infty},\kappa)$ is independent of $\bq$,
$\xi$ and $t^*$. Furthermore,
\begin{align}
  \sup_{t \leq t^*} \| \psi^\xi(t; \bq)\|_s^2   
  +  \int_0^{t^*} \| \psi^\xi(t; \bq)\|_{s+1}^2 dt
  \leq  c \|\xi\|_s^2  \exp( c t^* \|\bq\|^{a}_s) 
  \label{eq:psi:xi:Bnd:Hs}
\end{align}
where the constant $c = c(s, \|\theta_0\|_s, \kappa)$ is independent of $\bq$,
$\xi$ and $t^* > 0$; and $a$ is precisely the constant from \eqref{eq:Hs:balance}.
\item[(ii)] On the other hand, given any
  $\bq, \xi, \txi \in H^s(\TT^2)$, $t^* > 0$
\begin{align}
  \sup_{t \leq t^*} \| \psi^{\xi,\txi}(t;\bq)\|_0^2 
  + \int_0^{t^*}\| \psi^{\xi,\txi}(t;\bq)\|_1^2 dt
  \leq c(\|\xi\|_s^4 + \|\txi\|_s^4)
  \label{eq:psi:xi:txiBnd:L2}
\end{align}
where $c= c(s,\|\theta_0\|_{L^\infty}, \|\theta_0\|_s, \kappa)$ is
independent of $\bq, \xi, \txi$ and $t^*$. Moreover,
\begin{align}
  \sup_{t \leq t^*} \| \psi^{\xi,\txi}(t; \bq)\|_s^2 
  \leq c(\|\xi\|^4_s + \|\txi\|_s^4)
  \exp( t^* c \| \bq\|^{a}_s )
  \label{eq:psi:xi:txiBnd:H2}
\end{align}
for a constant
$c = c(s,\|\theta_0\|_{L^\infty}, \|\theta_0\|_s, \kappa)$ independent
of $\bq, \xi, \txi$ and $t^* > 0$.
\end{itemize}
\end{Prop}

Before turning to the details of the proof let us recall some useful
inequalities. Firstly the Sobolev embedding theorem in dimension
$d = 2$ is given as
\begin{align}
   \| g \|_{L^p} \leq c \| g \|_{H^{r}}
   \quad \text{ for any } r \geq 1 - \frac{2}{p}, \text{ with } 2 \leq p < \infty,
   \label{eq:sob:embedding}
\end{align}
for any $g: \TT^2 \to \RR$ in $H^r(\TT^2)$, where the
universal constant $c$ depends only on $p$ and $r$.  We also make use
of the Leibniz-Kato-Ponce inequality which takes the general form
\begin{align}
  \| \Lambda^{r} (fg) \|_{L^m} \leq 
  C ( \| \Lambda^{r} f\|_{L^{p_1}} \| g \|_{L^{q_1}} +
  \| f\|_{L^{p_2}} \| \Lambda^{r}  g \|_{L^{q_2}} 
  )
  \label{eq:Kato:Ponce}
\end{align}
valid for any $r \geq 0$, $1 < m < \infty$ and
$1 < p_i, q_i \leq \infty$ with $m^{-1} = p^{-1}_j + q^{-1}_j$ for
$j = 1,2$ and where $C$ is a positive constant depending only on $r, m, p_1, q_1, p_2, q_2$.

\begin{proof}
  The claimed regularity for $\psi^\xi$, $\psi^{\xi,\txi}$ follows from
  \cref{def:adr_weak} and the forthcoming formal estimates leading to
  \eqref{eq:psi:xi:Bnd:L2}--\eqref{eq:psi:xi:txiBnd:H2} which can be
  justified in the context of an appropriate regularization
  scheme. We begin by showing \eqref{eq:psi:xi:Bnd:L2}. From
  \eqref{eq:L2:balance}, namely multiplying \eqref{eq:grad:the:xi} by
  $\psi^\xi$ and integrating we have
\begin{align}
  \frac{1}{2}\frac{d}{dt} \| \psi^\xi\|_0^2  
  + \kappa \| \nabla \psi^\xi\|_0^2
   = - \int_{\TT^2} \xi \cdot \nabla \theta(\bq) \psi^\xi dx.
   \label{eq:energy:psixi}
\end{align}  
Integrating by parts and using that $\xi$ is divergence free
\begin{align}
 \left| \int_{\TT^2} \xi \cdot \nabla \theta(\bq) \psi^\xi dx \right|
  =  \left| \int_{\TT^2} \xi \cdot \nabla \psi^\xi \theta(\bq) dx
  \right|
  \leq \|\theta(\bq) \|_{L^\infty} \|\nabla \psi^\xi\|_0 \|\xi\|_0
  \label{est:int:advec}
\end{align}
Invoking the Maximum principle as in \eqref{eq:Linfty:bnd} we obtain
that
\begin{align}
   \| \theta(t;\bq) \|_{L^\infty} \leq \| \theta_0\|_{L^\infty} \quad
  \text{ for any } t \geq 0,
  \label{eq:Max:Prin:Th}
\end{align}
and hence
\begin{align*}
  \frac{d}{dt} \| \psi^\xi\|_0^2 + \kappa \| \nabla \psi^\xi\|_0^2 
  \leq c\|\xi\|_0^2.
\end{align*}
This immediately implies the first estimate \eqref{eq:psi:xi:Bnd:L2}. For showing \eqref{eq:psi:xi:Bnd:L2:b}, we estimate \eqref{est:int:advec} differently, namely
\begin{align*}
 \left| \int_{\TT^2} \xi \cdot \nabla \theta(\bq) \psi^\xi dx \right|
  \leq \|\xi\|_p \|\nabla \theta(\bq) \|_0 \|\psi^\xi\|_q 
\end{align*}
with $1 < p,q < \infty$ such that $\frac{1}{p} + \frac{1}{q} = \frac{1}{2}$. With the Sobolev inequality \eqref{eq:sob:embedding} and noting that $q \to 2$ when $p
\to \infty$ we can find $p$ and $q$ in this
range such that
\begin{align*}
 \left| \int_{\TT^2} \xi \cdot \nabla \theta(\bq) \psi^\xi dx \right|
  & \leq \|\xi\|_s \|\nabla \theta(\bq) \|_0 \|\nabla \psi^\xi\|_0 \\
  & \leq \frac{\kappa}{2} \|\nabla \psi^\xi\|_0^2 + c \|\xi\|_s^2 \|\nabla \theta(\bq) \|_0^2, 
\end{align*}
which in combination with \eqref{eq:energy:psixi} yields
\begin{align}
    \frac{d}{dt} \| \psi^\xi\|_0^2
  + \kappa \| \nabla \psi^\xi\|_0^2 \leq c\|\xi\|_{s}^2 \|\nabla \theta(\bq)\|_0^2.
  \label{ineq:energy:psixi}
\end{align}
Integrating \eqref{eq:L2:balance} for $f = 0$ with respect to time, we have
\begin{align}
  \sup_{s \leq t^*} \|\theta(\bq)\|^2_0 + \kappa \int_0^{t^*} \|\nabla \theta(\bq)\|^2_0 dt
  \leq \|\theta_0\|^2_0
  \label{eq:stupid:L2:bnd}
\end{align}
Hence from \eqref{ineq:energy:psixi} and \eqref{eq:stupid:L2:bnd}, it follows that
\begin{align*}
     \sup_{t\leq t^*} \|\psi^\xi\|_0^2 
  + \kappa \int_0^{t^*}\| \nabla \psi^\xi\|_0^2 dt
  \leq c \|\xi\|_s^2 \int_0^{t^*} \|\nabla \theta(\bq)\|_0^2 dt \leq c \|\theta_0\|_0^2 \|\xi\|_s^2,
\end{align*}
finishing the proof of \eqref{eq:psi:xi:Bnd:L2:b}.

Turning to $H^s(\TT^2)$ estimates we refer to \eqref{eq:Hs:balance}
which translates to
\begin{align}
\frac{d}{dt} \| \psi^\xi \|^2_s + \kappa \| \nabla \psi^\xi \|_s^2 
        \leq c\| \psi^\xi \|^2_s \| \bq\|^{a}_s  
  - 2 \int \FD^s (\xi \cdot \nabla \theta(\bq))  \FD^s \psi^\xi \, dx.
  \label{eq:Hs:inequal:psi:xi}
\end{align}
Invoking H\"older's inequality and the Leibniz bound
\eqref{eq:Kato:Ponce} we estimate 
\begin{align}
  \left | \int \FD^s (\xi \cdot \nabla \theta(\bq))  
                        \FD^s \psi^\xi  \, dx \right|
  \leq c \| \FD^s \psi^\xi\|_{L^p}
       ( \| \FD^s \xi \|_0 \|\FD^1 \theta(\bq)\|_{L^q} 
        +  \| \xi \|_{L^q} \|\FD^{s+1} \theta(\bq)\|_0)
    \label{eq:K:P:App:1}
\end{align}
valid whenever $1 < p, q < \infty$ and maintains
$1 - \tfrac{1}{p} = \tfrac{1}{2} + \tfrac{1}{q}$ i.e. $q = 2p/(p -2)$.  Again with the Sobolev
inequality \eqref{eq:sob:embedding} and noting that $q \to 2$ when $p
\to \infty$ we can find $p$ and $q$ in this
range such that
\begin{align}
  \left | \int \FD^s (\xi \cdot \nabla \theta(\bq))  
                        \FD^s \psi^\xi  \, dx \right|
  &\leq c \| \FD^{s+1} \psi^\xi\|_0 
             \| \FD^s \xi \|_0 \|\FD^{s+1} \theta(\bq)\|_0
    \notag\\
  &\leq \frac{\kappa}{4} \| \FD^{s+1} \psi^\xi\|_0^2 +
             c\| \FD^s \xi \|_0^2 \|\FD^{s+1} \theta(\bq)\|_0^2.
    \label{eq:K:P:App:2}
\end{align}
Combining this bound with \eqref{eq:Hs:inequal:psi:xi} 
yields the inequality
\begin{align}
  \frac{d}{dt} \| \psi^\xi \|^2_s+ \frac{\kappa}{2} \| \nabla \psi^\xi \|_s^2 
  \leq c\| \psi^\xi \|^2_s \| \bq\|^{a}_s  
         + c \| \xi \|_s^2 \|\theta(\bq)\|_{s+1}^2
  \label{eq:psi:xi:HS:est:1}
\end{align}
so that with the Gronwall inequality we obtain
\begin{align*}
  \sup_{r \leq t^*} \| \psi^\xi \|^2_s
  \leq \|\xi\|_s^2 \exp(c t^* \|\bq\|_s^{a})
           \int_0^{t^*} \|\theta(\bq)\|_{s+1}^2  dt
\end{align*}  
A second application of \eqref{eq:Hs:balance}, this time with $f = 0$, yields
\begin{align}
  \kappa \int_0^{t^*} \|\theta(\bq)\|_{s+1}^2 dt
  &\leq c t^* \| \bq\|^{a}_s  \sup_{t \leq t^*}\|\theta(\bq)\|_{s}^2 
  \leq c t^* \| \bq\|^{a}_s   \exp( ct^*\|\bq\|_s^{a}) \|\theta_0\|_s^2
    \notag\\
  &\leq c \exp( c t^*\|\bq\|_s^{a}) \|\theta_0\|_s^2.
    \label{eq:psi:xi:HS:est:2} 
\end{align}
Combining the previous two bounds we find, for any $t^* \geq 0$,
\begin{align}
  \sup_{t \leq t^*}\| \psi^\xi\|_s^2  
  \leq  c \exp(c t^* \|\bq\|_s^{a}) \|\xi\|_s^2 \|\theta_0\|_s^2.
  \label{eq:psi:xi:HS:est:3}
\end{align}
Integrating \eqref{eq:psi:xi:HS:est:1} in time and invoking
\eqref{eq:psi:xi:HS:est:2}, \eqref{eq:psi:xi:HS:est:3}
\begin{align*}
 \kappa \int_0^{t^*} \| \nabla \psi^\xi \|_s^2 dt  \leq
  c t^* \sup_{t \leq t^*} \| \psi^\xi \|^2_s \| \bq\|^{a}_s   
         + c \| \xi \|_s^2 \int_0^{t^*}\|\theta(\bq)\|_{s+1}^2 dt
  \leq c \|\theta_0\|_s^2 \|\xi\|_s^2 \exp(ct^* \|\bq\|_s^{a})
\end{align*}
and hence we now obtain \eqref{eq:psi:xi:Bnd:Hs}.

We next provide estimates for $\psi^{\xi, \txi}$.  As before we begin
by addressing the $L^2$ case, namely \eqref{eq:psi:xi:txiBnd:L2}. 
We take the inner product in $L^2$ of \eqref{eq:grad:the:xitxi} with $\psi^{\xi, \txi}$ and integrate
to obtain, as in \eqref{eq:L2:balance},
\begin{align}
    \frac{1}{2}\frac{d}{dt} \| \psi^{\xi, \txi}\|_0^2 
  + \kappa \| \nabla \psi^{\xi, \txi}\|_0^2 =
       - \int \txi \cdot \nabla \psi^\xi \psi^{\xi, \txi}
       - \int \xi \cdot \nabla \psi^{\txi} \psi^{\xi, \txi}
  := I.
  \label{eq:L2:xitxi:en:bal}
\end{align}
Integrating by parts and using H\"older's inequality the right hand side is estimated as
\begin{align*}
  |I| \leq
    (\|\xi\|_{L^p} + \| \txi\|_{L^p})
   (  \|\psi^{\xi}\|_{L^q} + \| \psi^{\txi} \|_{L^q}) 
                    \|\nabla\psi^{\xi, \txi}\|_0
\end{align*}                    
for $p^{-1} + q^{-1} =2^{-1}$.  Choosing $p$, $q$ appropriately and then
applying the Sobolev embedding, \eqref{eq:sob:embedding}, we find
\begin{align*}
  |I| \leq
    (\|\xi\|_s + \| \txi\|_s)
   (  \|\psi^{\xi}\|_{1} + \| \psi^{\txi} \|_{1}) 
                    \|\nabla\psi^{\xi, \txi}\|_0 
   \leq c(\|\xi\|_s^2 + \| \txi\|_s^2)  (  \|\psi^{\xi}\|_1^2 + \| \psi^{\txi} \|_1^2) 
         + \frac{\kappa}{2} \|\nabla\psi^{\xi, \txi}\|_0^2.
\end{align*}
Hence, using this bound with \eqref{eq:L2:xitxi:en:bal} and
then applying \eqref{eq:psi:xi:Bnd:L2} we infer
\eqref{eq:psi:xi:txiBnd:L2}.

We turn finally to the $H^s(\TT^2)$ estimates for
$\psi^{\xi, \txi}$.    Here \eqref{eq:Hs:balance} becomes 
\begin{align}
\frac{d}{dt} \| \psi^{\xi, \txi} \|^2_s + \kappa \| \nabla \psi^{\xi, \txi} \|_s^2 
        \leq c\| \psi^{\xi, \txi} \|^2_s \| \bq\|^{a}_s  
  - 2 \int 
  \FD^s (\txi \cdot \nabla \psi^\xi +\xi \cdot \nabla \psi^{\txi})  
  \FD^s \psi^{\xi, \txi}  \, dx.
  \label{eq:psi:xi:txi:Hs:bnd:1}
\end{align}
Estimating the last term above in a similar fashion in 
\eqref{eq:K:P:App:2} above leads to 
\begin{align}
  &\left|
  \int  \FD^s (\txi \cdot \nabla \psi^\xi +\xi \cdot \nabla \psi^{\txi})  
           \FD^s \psi^{\xi, \txi}  \, dx
  \right|
     \leq
   \frac{\kappa}{2} \| \psi^{\xi, \txi} \|_{s+1}^2 
    + c(\|\xi\|^2_s + \|\txi\|_s^2)
        (\|\psi^\xi\|_{s+1}^2 + \| \psi^{\txi}\|_{s+1}^2).
  \label{eq:psi:xi:txi:Hs:bnd:2}
\end{align}
Combining the previous two bounds \eqref{eq:psi:xi:txi:Hs:bnd:1},
\eqref{eq:psi:xi:txi:Hs:bnd:2} and then making use of Gronwall inequality and \eqref{eq:psi:xi:Bnd:Hs} we obtain
\begin{align}
 \sup_{r \leq t^*} \| \psi^{\xi, \txi} \|^2
  \leq& c(\|\xi\|^2_s + \|\txi\|_s^2)
        \exp( ct^* \| \bq\|^{a}_s ) 
  \int_0^{t^*} (\|\psi^\xi\|_{s+1}^2 + \|\psi^{\txi}\|_{s+1}^2) dt
        \notag\\
  \leq& c(\|\xi\|^4_s + \|\txi\|_s^4) \exp( c t^* \| \bq\|^{a}_s ),
        \notag
\end{align}
which establishes the final bound, \eqref{eq:psi:xi:txiBnd:H2}, completing the proof.
\end{proof}

\subsubsection{Bounds on the Potential $U$ and its Derivatives}
\label{sec:d:d2:pot:bnds}

With these preliminary bounds on \eqref{eq:ad:forced} and hence
\cref{prop:grad:xi:apriori} in hand we turn to provide estimates for $U$
defined as in \eqref{eq:post:passive:scal}. Recall that we seek to
determine the extent to which \cref{B123},
\cref{ass:integrability:cond} applies for the certain classes of
potential $U$ which arise in this example, namely
\eqref{eq:pot:passive:scal:g:n} subject to conditions on the
observation operator \eqref{eq:gen:linear:form}. Of course, since $U$
is positive, \cref{ass:integrability:cond} holds
regardless of our assumptions on $\Obs$.

Regarding the assumptions on $\Obs$ we consider the following three
situations.  Fix an observation time window $t^* > 0$.  Firstly, in the
case of spectral observations or that of spatial (volumetric) averages,
we may suppose
\begin{align}
  |\Obs( \phi)|  \leq c_0 \sup_{t \leq t^*} \| \phi(t)\|_0
  \label{eq:obs:cond:spec}
\end{align}
for $\phi \in C([0,t^*]; L^2(\TT^2))$.  On the other hand, the case of
pointwise spatial-temporal measurement yields the condition:
\begin{align}
  |\Obs( \phi)| \leq c_0 \sup_{t \leq t^*} \|\phi(t)\|_{L^\infty}
  \label{eq:obs:cond:node}
\end{align}
for $\phi \in C([0,t^*] \times \TT^2)$. Finally, for estimates
involving gradients or other derivatives of $\phi$ we assume that,
for some $s > 0$,
\begin{align}
  |\Obs( \phi)| \leq c_0 \sup_{t \leq t^*} \|\phi(t)\|_{H^s}
  \label{eq:obs:cond:grad}
\end{align}
valid for $\phi \in C([0,t^*]; H^s(\TT^2))$.

Let us begin with estimates on $DU$ and $D^2U$ in negative Sobolev
space which in turn yield the conditions in \cref{B123} on the
$\H_\gamma$ spaces, \eqref{def:Hgamma}, defined relative to a
covariance operator $\cC$ of the Gaussian prior $\mu_0$ in
\eqref{eq:post:passive:scal}. 

\begin{Prop}
\label{prop:DU:DsqU:Bnds}
Let $U$ be defined as in \eqref{eq:pot:passive:scal:g:n}
for a fixed $\data \in \RR^m$ and $\Gamma$ a symmetric
strictly positive definite matrix.
\begin{itemize}
\item[(i)] When $\Obs$ satisfies \eqref{eq:obs:cond:spec}, $U$ is
  twice Fr\'echet differentiable in $H^{s'}(\TT^2)$ for any $s' > 0$.
  In this case for any $s' \geq 0$
  \begin{align}
    \| D U(\bq)\|_{-s'} \leq M_1 < \infty
    \label{eq:Grad:bnd:spec}
  \end{align}
  for a constant
  $M_1 = M_1(s', \kappa, t^*, \theta_0, c_0, \data, \Gamma)$ which is
  independent of $\bq$.\footnote{Note furthermore that $M_1$ is
    independent of $t^*$ in the case when $s' > 0$, cf. 
    \eqref{eq:psi:xi:Bnd:L2}, \eqref{eq:psi:xi:Bnd:L2:b}.}
  Furthermore, assuming now that $s' > 0$ we have
  \begin{align}
    \| D^2 U(\bq)\|_{\mL_2(H^{s'}(\TT^2))} \leq M_2 < \infty,
    \label{eq:Hess:bnd:spec}
  \end{align}
  where $\|\cdot\|_{\mL_2(H^{s'}(\TT^2))}$ denotes the standard
  operator norm of a real-valued bilinear operator on
  $H^{s'}(\TT^2) \times H^{s'}(\TT^2)$ (see \eqref{eq:C:gamma:Hess:U}), and
  $M_2 = M_2(s', \kappa, \theta_0, c_0, \data, \Gamma)$ is a constant
  independent of $\bq$.
\item[(ii)] In the case \eqref{eq:obs:cond:grad} for $\Obs$ we have
  once again that $U$ is twice Fr\'echet differentiable in
  $H^{s}(\TT^2)$ for the given value of $s > 0$ in
  \eqref{eq:obs:cond:grad}. Here, for any $s' \geq s$,
  \begin{align}
    \| D U(\bq)\|_{-s'} \leq M \exp(c \|\bq\|_{s'}^{a})
    \label{eq:Grad:bnd:point}
  \end{align}
  and
  \begin{align}
   \| D^2 U(\bq)\|_{\mL_2(H^{s'}(\TT^2))}  \leq M \exp(c \|\bq\|_{s'}^{a})
    \label{eq:Hess:bnd:point}
\end{align}
where $c = c(s', \kappa, t^*, \theta_0, c_0, \data, \Gamma)$,
$M = M(s', \kappa, t^*, \theta_0, c_0, \data, \Gamma)$ are independent of $\bq$
and $a> 0$ is precisely the constant appearing in
\eqref{eq:Hs:balance}.
\item[(iii)] Finally under the assumption that $\Obs$ obeys
  \eqref{eq:obs:cond:node}, $U$ is twice Fr\'echet differentiable in
  $H^{s'}(\TT^2))$ for any $s' > 1$.  In this case, when $s' > 1$, we
  again have the bounds \eqref{eq:Grad:bnd:point},
  \eqref{eq:Hess:bnd:point}.
\end{itemize}
\end{Prop}

\begin{proof}
We start with the proof of \eqref{eq:Grad:bnd:spec}.
Notice that, referring back to \eqref{eq:DD:U:xi} 
and using the condition \eqref{eq:obs:cond:spec}, we have
\begin{align*}
    |U^{\xi} (\bq)| \leq c (1 + \sup_{t \leq t^*} \| \theta(\bq)\|_0) \cdot
    \sup_{t \leq t^*} \| \psi^\xi(\bq) \|_0,
\end{align*}
for any $\bq \in L^2(\TT^2)$, $\xi \in H^{s'}(\TT^2)$ and
$c = c(\GP, \data, c_0)$.  Observe that for any $s' \geq 0$ we have
\begin{align}
   \| D U(\bq) \|_{-s'}
  = \sup_{\|\xi \|_{s'} = 1} |U ^\xi(\bq)|.
  \label{eq:C:gamma:grad:U}
\end{align}
Thus, invoking the bounds \eqref{eq:stupid:L2:bnd},
\eqref{eq:psi:xi:Bnd:L2} when $s' = 0$ or \eqref{eq:psi:xi:Bnd:L2:b}
for the case $s' > 0$, we obtain \eqref{eq:Grad:bnd:spec}.

We turn next to the proof of \eqref{eq:Hess:bnd:spec}.  In this case, working
from \eqref{eq:DD:U:xi:txi} and again making use of the condition
\eqref{eq:obs:cond:spec},
\begin{align*}
  |U^{\xi, \txi} (\bq)| 
      \leq c \sup_{t \leq t^*} \| \psi^{\txi}(\bq)\|_0 
     \cdot \sup_{t \leq t^*} \| \psi^{\xi}(\bq) \|_0 
      + c (1 +\sup_{t \leq t^*} \| \theta(\bq)\|_0 )
      \cdot \sup_{t \leq t^*} \| \psi^{\xi, \txi} (\bq) \|_0
\end{align*}
for any $\xi, \txi \in H^{s'}(\TT^2)$, where $c = c(\GP, \data, c_0)$.  Here
using 
\begin{align}
  \| D^2 U(\bq) \|_{\mathcal{L}_2(H^{s'}(\TT^2))}
   = \sup_{\|\xi \|_{s'} = \|\txi\|_{s'} = 1} |U ^{\xi, \txi}(\bq)|
  \label{eq:C:gamma:Hess:U}
\end{align}
and the bounds \eqref{eq:stupid:L2:bnd},
\eqref{eq:psi:xi:Bnd:L2:b}, \eqref{eq:psi:xi:txiBnd:L2}, the desired
estimate \eqref{eq:Hess:bnd:spec} now follows.

We next address \eqref{eq:Grad:bnd:point}, \eqref{eq:Hess:bnd:point}.
Here \eqref{eq:DD:U:xi} and \eqref{eq:obs:cond:grad} result in
\begin{align}
  |U^{\xi} (\bq)|
        \leq c (1 + \sup_{t \leq t^*} \| \theta(\bq)\|_s) \cdot
  \sup_{t \leq t^*} \| \psi^\xi(\bq) \|_s
  \label{eq:U:xi:Hs:Obs:bnd}
\end{align}
and similarly, with \eqref{eq:DD:U:xi:txi},
\begin{align}
  |U^{\xi,\txi} (\bq)|
  \leq c \sup_{t \leq t^*} \| \psi^{\txi}(\bq)\|_{s} 
     \cdot \sup_{t \leq t^*} \| \psi^{\xi}(\bq) \|_{s} 
      + c (1 +\sup_{t \leq t^*} \| \theta(\bq)\|_{s} )
  \cdot \sup_{t \leq t^*} \| \psi^{\xi, \txi} (\bq) \|_{s}
    \label{eq:U:txi:Hs:Obs:bnd}
\end{align}
for any $\xi,\txi \in H^{s'}(\TT^2)$, $s' \geq s$. Thus, invoking
\eqref{eq:Hs:balance} (with $f \equiv 0$), \eqref{eq:psi:xi:Bnd:Hs},
\eqref{eq:psi:xi:txiBnd:H2} with
\eqref{eq:C:gamma:grad:U}-\eqref{eq:U:txi:Hs:Obs:bnd}, we obtain
\eqref{eq:Grad:bnd:point}, \eqref{eq:Hess:bnd:point} establishing the
second item.

Regarding the final item (iii) observe that \eqref{eq:DD:U:xi},
\eqref{eq:DD:U:xi:txi} and the Sobolev embedding of
$H^s(\TT^2) \subset L^\infty(\TT^2)$ when $s > 1$ we obtain bounds
as in \eqref{eq:U:xi:Hs:Obs:bnd}, \eqref{eq:U:txi:Hs:Obs:bnd}
under \eqref{eq:obs:cond:node} for any $s > 1$.  We therefore
conclude this final item arguing as in the previous case.  The proof
is now complete.  
\end{proof}

Drawing upon \cref{prop:DU:DsqU:Bnds} we now draw certain conclusions
on the scope of applicability of \cref{B123} to
\eqref{eq:post:passive:scal}.  For this purpose suppose $\cC$ is a
symmetric, positive, trace class operator on $L^2(\TT^2)$.  Following
the notations introduced above in \eqref{def:Hgamma} we consider the
fractional powers of $\cC$ and associated spaces $\H_\gamma$ with norm
$|\bq |_{\gamma} = \| \cC^{-\gamma} \bq\|_{0}$ for $\gamma \geq 0$, so
that in particular we have the notation $|\bq| = \|\bq\|_{0}$.  We
have the following corollary:

\begin{Cor} \label{cor:DU:DsqU:Bnd:w:C} Let $\cC$ be a symmetric,
  positive, trace class operator on $L^2(\TT^2)$.  Assume that
  for some $s > 0$, and some $\rt \in (0,1/2)$ there is
  a constant $c_1$ such that
  \begin{align}
    \| \bq\|_s  \leq  c_1 |\bq |_\gamma =  c_1 \|\cC^{-\rt} \bq\|_0
    \quad \mbox{ for all } \bq \in \H_\gamma,
    \label{eq:C:reg:cond:1}
  \end{align}
  so that $\H_\gamma \subset H^s(\TT^2)$.
  \begin{itemize}
  \item[(i)] Under the spectral observation assumption,
    \eqref{eq:obs:cond:spec}, \cref{B123} and
    \cref{ass:integrability:cond} hold for $U$ and the given $\cC$.
    Additionally, if for this value of $\rt$, $\cC^{1-2\gamma}$
    is trace class in the sense of \eqref{eq:Tr:Def:b}, so that
    \cref{ass:higher:reg:C} holds, then
    \cref{thm:weak:harris} applies to \eqref{eq:post:passive:scal}.
  \item[(ii)] Under \eqref{eq:obs:cond:grad}, assuming that
    \eqref{eq:C:reg:cond:1} holds for the value of $s > 0$ in
    \eqref{eq:obs:cond:grad} we have that
    \begin{align}
      |DU(\bq)|_{-\gamma} \leq  M \exp( c |\bq|_{\gamma}^a)
      \label{eq:tran:to:gamma:sp:DU}
    \end{align}
    and that
    \begin{align}
      \|\cC^\gamma D^2U(\bq) \cC^\gamma\|_{\mL_2(\H_0)}
      \leq M \exp( c |\bq|_{\gamma}^a)
            \label{eq:tran:to:gamma:sp:DU:b}
    \end{align}
    where $\|\cdot\|_{\mL(\H_0)}$ here denotes the standard operator
    norm of a real-valued bilinear operator on $\H_0 \times \H_0$, and
    again the constants
    $c = c(s', \kappa, t^*, \theta_0, c_0, c_1 \data, \Gamma)$,
    $M = M(s', \kappa, t^*, \theta_0, c_0, c_1, \data, \Gamma)$ are
    independent of $\bq$ and $a> 0$ is as in \eqref{eq:Hs:balance}.
  \item[(iii)] In the case \eqref{eq:obs:cond:node}, if
    \eqref{eq:C:reg:cond:1} holds for some $s > 1$ then we again have
    the bounds \eqref{eq:tran:to:gamma:sp:DU},
    \eqref{eq:tran:to:gamma:sp:DU:b} for the corresponding values of
    $\gamma$.
  \end{itemize}
\end{Cor}

\begin{proof}
  Regarding the first item we proceed to establish the conditions
  \eqref{Hess:bound} and \eqref{dissip:new}. Observe that under
  \eqref{eq:C:reg:cond:1}
  \begin{align}
    c_1^2 \| D^2U (\bq) \|_{\mathcal{L}_2(H^s(\TT^2))} \geq
        \|\cC^\gamma D^2U(\bq) \cC^\gamma \|_{\mathcal{L}_2(\H_0)} 
    \label{eq:obv:upp:bnd:1}
  \end{align}
  so that with \eqref{eq:Hess:bnd:spec} we infer \eqref{Hess:bound}.
  For \eqref{dissip:new} we demonstrate the stronger condition
  \eqref{eq:B3a}.  Again, due to \eqref{eq:C:reg:cond:1} we have
  \begin{align}
    c_1 \| D U(\bq)\|_{-s} \geq | D U(\bq) |_{-\gamma}
        \label{eq:obv:upp:bnd:2}
  \end{align}
  so that \eqref{eq:B3a} follows from \eqref{eq:Grad:bnd:spec}.

  Regarding the second and third items we simply apply
  \eqref{eq:obv:upp:bnd:1}, \eqref{eq:obv:upp:bnd:2}
  now in combination with \eqref{eq:Grad:bnd:point}
  and \eqref{eq:Hess:bnd:point}.    The proof is complete.
\end{proof}

\begin{Rmk}
  \label{rmk:crux:matter}
  Let $A$ be the Stokes operator in dimension $2$ with periodic boundary
  conditions.  Of course for any given $s > 0$ the condition
  \eqref{eq:C:reg:cond:1} is fulfilled when $\cC = (A)^{-\kappa/2}$ for
  any $\kappa$ such that $\kappa \geq s/\gamma$.  Here note, in regards
  to \cref{ass:higher:reg:C},  $\cC = (A)^{-\kappa/2}$ has the
  eigenvalues $\lambda_j \approx |j|^{\kappa/2}$.  Thus
  \eqref{eq:Tr:Def:b} entails the additional requirement 
  $\kappa > 2/(1 - 2 \gamma)$.

  Note however that the examples considered in
  \cite{borggaard2018bayesian} involved a covariance $\cC$ with
  exponentially decaying spectrum so that \eqref{eq:C:reg:cond:1}
  applies for any $s \geq 0$ and \eqref{eq:Tr:Def:b} for any
  $0 \leq \gamma < 1/2$.
\end{Rmk}

\begin{Rmk}[Improved bounds in the time independent case]

  We expect that improved, $\bq$-independent bounds on
  (\ref{eq:grad:the:xi}) and (\ref{eq:grad:the:xitxi}) can be achieved
  through more sophisticated parabolic regularity techniques.  In turn
  this could improve bounds obtainable for $DU$ and $D^2U$ in the case
  of point observations \eqref{eq:obs:cond:node}.  Whatever the
  mechanism, we note that the numerical results in
  \cite{borggaard2018bayesian} suggest good mixing occurs for the
  Hamiltonian Monte Carlo algorithm in this case of point observations
  notwithstanding the fact that our current results do not cover this
  situation.
  
  In this connection it is notable that a global bound on $DU$ and
  $D^2U$ and hence the conditions for \eqref{thm:weak:harris} can be
  achieved for point observations in the time-stationary analogue of
  \eqref{eq:ad:forced} thanks to \cite{Berestycki_etal_09}.  Let
  \begin{align}
    \bq \cdot \nabla \theta =  \kappa \Delta \theta + f
    \label{eq:AD:stat}
  \end{align}
  on $\TT^2$ for a given fixed $f: \TT^2 \to \RR$, $\kappa >0$.  We
  can consider, similarly to above, the statistical inversion problem
  of recovering a divergence free $\bq$ from the sparse observation of
  the resulting solution $\theta: \TT^2 \to \RR$.  In this case,
  following the Bayesian approach we again obtain a posterior measure
  of the form \eqref{eq:post:passive:scal} with $U$ given analogously
  to \eqref{eq:pot:passive:scal:g:n} in the case of Gaussian
  observation noise.

  As previously the task of estimating $DU$ and $D^2U$ entails
  suitable estimates for
  \begin{align*}
    \bq \cdot \nabla \psi^\xi 
  =  \kappa \Delta \psi^\xi 
     - \xi \cdot \nabla \theta(\bq),
  \end{align*}
  and
  \begin{align*}
    \bq \cdot \nabla \psi^{\xi, \txi}
  =  \kappa \Delta\psi^{\xi, \txi}
     - \txi \cdot \nabla \psi^\xi
     - \xi \cdot \nabla \psi^{\txi}.
  \end{align*}
  over suitable directions $\xi, \txi$.

  Suppose that $\phi$ obeys
  \begin{align}
    \bq \cdot \nabla \phi =  \kappa \Delta \phi + g
    \label{eq:gen:AD:stat:eqn}
  \end{align}
  for some $\bq: \TT^2 \to \RR^2$, divergence free and
  $g:\TT^2 \to \RR$.  According to \cite[Lemma
  1.3]{Berestycki_etal_09} we have that\footnote{The result
    \cite{Berestycki_etal_09} is stated for \eqref{eq:AD:stat}
    supplemented with Dirichlet boundary conditions but pursuing the
    proof it is clear that this bound also applies in the spatially
    periodic case.}
  \begin{align}
    \|\phi\|_{L^\infty} \leq c \| g \|_{L^p}
    \label{eq:int:bnd}
  \end{align}
  for any $p >1$ where crucially the constant $c = c(p, \kappa)$ is
  independent of $\bq$.  Applying \eqref{eq:int:bnd} and carrying out
  other standard manipulations we have that
  \begin{align}
    \| \theta(\bq) \|_{L^\infty}^2 + \|\nabla \theta(\bq)\|^2_{0}
    \leq c \|f\|_{0}^2
    \label{eq:th:base:bnd}
  \end{align}
  for $c = c(\kappa)$ independent of $\bq$.  As such a second application of
  \eqref{eq:int:bnd}, Sobolev embedding, \eqref{eq:sob:embedding},
  and \eqref{eq:th:base:bnd}
  yields
  \begin{align}
    \| \psi^{\xi} \|_{L^\infty}
    \leq c \| \xi \|_{s} \|f\|_{0}
        \label{eq:phi:stat:xi:2}
  \end{align}
  for any $s > 0$ where the constant $c = c(s, \kappa)$ is again
  independent of $\bq$.  Moreover, using that $\bq$ is divergence free
  and \eqref{eq:th:base:bnd}
  \begin{align}
     \|\nabla \psi^{\xi} \|_{0} \leq c \|\xi\|_{0} \| f\|_{0}
    \label{eq:phi:stat:xi:3}
  \end{align}
  with $c = c(s,\kappa)$ independent of $\bq$.  Finally
  \eqref{eq:int:bnd} followed by 
  \begin{align}
    \| \psi^{\xi, \txi} \|_{L^\infty}
    \leq c (\|\xi\|_{s}^2 + \| \txi \|^2_{s}).
    \label{eq:phi:stat:xi:4}
  \end{align}
  for any $s > 0$ where $c = c(s, \kappa)$ does not depend on $\bq$.
  Thus, arguing as in \cref{prop:grad:xi:apriori} but making use of
  \eqref{eq:phi:stat:xi:2}, \eqref{eq:phi:stat:xi:4} we can therefore
  conclude that whenever
  \begin{align*}
    |\Obs(\phi)| \leq c_0 \|\phi\|_{L^\infty},
  \end{align*}
  bounds as in \eqref{eq:Grad:bnd:spec},
  \eqref{eq:Hess:bnd:spec} must hold.
\end{Rmk}

\section{Outlook}

This work provides an illustration of the power and efficacy of the
weak Harris theorem as a tool for the analysis of mixing in
infinite-dimensional MCMC methods. Specifically our work addresses a
Hilbert space version from \cite{BePiSaSt2011} of the Hamiltonian
Monte Carlo method.  Notwithstanding recent progress in this setting
of infinite dimensional MCMC algorithms, the understanding of mixing
rates and the relatedly optimal choice of algorithmic parameters
remains in its infancy.  Let us therefore point out a number of
interesting questions remaining to be studied which we plan to address
in future work.

One immediate avenue concerns the analysis of numerically discretized
versions of the HMC algorithm \eqref{eq:PEHMC:kernel:def} which must
be used in practice. Here the Metropolization step, which is used to
correct for the bias introduced by the discretization of
\eqref{eq:dynamics:xxx}, must be accounted for. In a similar vein it
would be useful to have error bounds between the adjusted and
unadjusted versions of the algorithm.

It is also worth noting that there are a number of variations on the
infinite dimensional HMC algorithm from \cite{BePiSaSt2011} now
available in the literature whose mixing properties are poorly
understood, particularly as we regard these different algorithms in
comparative perspective. For example we note the Second-Order Langevin
Hamiltonian (SOLHMC) methods in \cite{ottobre2016} and the Riemannian
(geometric) HMC approach developed in \cite{BeByLiGi2017,
  BeGiShiFaStu2017}.

Although the above analysis is a nontrivial first step towards a
better understanding of \eqref{eq:HMC:kernel:overview} one may
nevertheless view the time step condition
\eqref{eq:time:restrict:basic} as restricting the scope of our
analysis to a perturbation of the linear Gaussian case;
cf. \cref{rmk:taking:stock:lip_1}.  It is notable that similar small
time step condition also appears in all the other recent studies of
the HMC algorithm that we are aware
of\cite{DurmusEtAl2017,LivingstoneEtAl2019,BoEbZi2018, BoEbPHMC}.  We
conjecture that for many problems of interest this restriction on $T$
may be far from optimal from the point of view of mixing rates.  Indeed
this bound on $T$ \eqref{eq:time:restrict:basic} turns on our
treatment of the Lyapunov structure in \cref{prop:FL} and on the
nudging scheme \cref{prop:FP:Type:contract} which could presumably be
improved with a more delicate treatment of the Hamiltonian dynamics
\eqref{eq:dynamics:xxx}.  As a starting point it would be of great
interest to find some simple settings in finite dimensions where this
could be carried out.

As already noted above in the introduction, a primary motivation for
considering infinite dimensional MCMC methods concerns the Bayesian
approach to PDE inverse problems. While several large scale numerical
studies have been carried out for some specific problems a more
systematic gallery of examples on which the performance of algorithms
have been experimentally tested would be desirable. Here our results
presented in \cref{sec:bayes:AD} show that analysis of conditions on
the potential $U$ in \eqref{eq:dynamics:xxx} as arising from the
Bayesian approach to PDE inverse problems can be quite involved.
Indeed, in the case of the advection-diffusion problem we consider
here, it is not clear that we can obtain a global Hessian bound on $U$
for interesting classes of observations, such as space-time point
observations. Thus it would be useful to develop an analysis that only
requires that $U$ is locally Lipschitz. More broadly, further examples
of PDE inverse problems as found in e.g. \cite{stuart2010inverse}
should be analytically studied in this context to obtain a broader
sense of the variety of relevant conditions on $U$.

\setcounter{equation}{0}
\appendix

\section{Consequences for convergence of observables}
\label{sec:LLN:CLT}

Let $P$ be a Markov kernel on a Polish space $\Hp$ and take
$\{Q_n(\bq_0)\}_{n \geq 1, \bq_0 \in \Hp}$ to be the Markov process
associated with $P$ starting from $\bq_0 \in \Hp$.  Suppose that $\mu_*$
is an invariant measure for $P$.  In addition to quantifying various
abstract notions of distance, i.e. the Wasserstein metric, between the
measures $\mu P^n$ and $\mu_*$, we are typically interested in
estimating
\begin{align}
  \left| P^n\Phi(\bq_0) - \int \Phi(\bq') \mu_*(d\bq') \right|
  \label{eq:obs:mean:conv:c1}
\end{align}
and also
\begin{align}
  \left| \frac{1}{n}\sum\limits_{k =1}^n\Phi(Q_k(\bq_0))
  - \int \Phi(\bq') \mu_*(d\bq') \right|
    \label{eq:obs:LLN:conv:c1}
\end{align}
for concrete observables $\Phi: \Hp \to \RR$ and starting from any initial
$\bq_0 \in \Hp$.

Typically, contraction bounds as in \eqref{Wass:trhoe} and
\eqref{eq:Wass:conv:FD} which we demonstrated above can be used to
establish estimates for quantities like \eqref{eq:obs:mean:conv:c1},
\eqref{eq:obs:LLN:conv:c1}.  Indeed, if the $\tilde{\rho}$ appearing
in the bounds \eqref{Wass:trhoe} and \eqref{eq:Wass:conv:FD} was
actually a metric then the Kantorovich-Wasserstein duality would
immediately imply bounds for \eqref{eq:obs:mean:conv:c1}.  Moreover, a
number of results in the literature, e.g.  \cite{hairer2008spectral,
  komorowski2012fluctuations, komorowski2012central,
  kuksin2012mathematics, kulik2017ergodic}, yield a law of large
numbers, central limit theorems type convergence results from
Wasserstein contraction bounds as desired in
\eqref{eq:obs:LLN:conv:c1}. This appendix proceeds to show that useful
bounds for \eqref{eq:obs:mean:conv:c1}, \eqref{eq:obs:LLN:conv:c1} can
still be achieved in our setting without presuming that the
underlying distance $\tilde{\rho}$ is a metric. Notwithstanding
the significant literature on such convergence results we expect
our approach here to be of novel interest even when the underlying
distance is a metric.

In order to proceed, let us recall a few basic definitions:
\begin{Def}
  \label{def:DLF}
  We say that $\dlf: \Hp \times \Hp \to \RR^+$ is a
  \emph{distance-like function} if $\dlf$ is symmetric,
  lower-semicontinuous and it holds that $\dlf(\bq,\btq) = 0$ if and
  only if $\bq = \btq$.  We define $\Wass_\dlf: \Pr(\Hp) \times \Pr(\Hp) \to \RR^+ \cup \{+\infty\}$ to be the following Wasserstein-like extension of $\dlf$ to $\Pr(\Hp) \times \Pr(\Hp)$:
\begin{align*}
  \Wass_\dlf( \nu_1, \nu_2) 
      = \inf_{\Gamma \in \Co(\nu_1, \nu_2)} 
  \int_{\Hp \times \Hp} \dlf(\bq, \btq) \Gamma(d \bq, d \btq),
\end{align*}
where $\Co(\nu_1, \nu_2)$ is the set of all couplings of
$\nu_1, \nu_2 \in Pr(\Hp)$.\footnote{The mapping $\Wass_\dlf$ is also called the `optimal transport cost functional' in the optimal transport literature; see, e.g., \cite{villani2008optimal}.}
\end{Def}
\noindent Relative to a given distance-like function $\dlf$ we define
$\dlf$-Lipschitz in the obvious way as:
\begin{Def}
  Given a distance-like function $\dlf: \Hp \times \Hp \to \RR^+$, we
  say that $\Phi:\Hp \to \RR$ is $\dlf$-Lipschitz with Lipschitz
  constant $L_{\Phi} > 0$ if
\begin{align*}
 |\Phi(\bq) - \Phi(\bq')| \leq L_{\Phi} \dlf(\bq,\bq') 
\end{align*}
for any $\bq,\bq' \in \Hp$.  We denote the set of $\dlf$-Lipschitz
functions as $\mbox{Lip}_{\dlf}$.
\end{Def}
\noindent In order to verify that an observable $\Phi$ is
$\dlf$-Lipschitz for the class of distance like functions employed
above, see \cref{prop:Lip:2:obs} below.

Results for \eqref{eq:obs:mean:conv:c1} can be drawn by using the following proposition.
\begin{Prop}\label{lem:Kantor:dual}
  Let $\dlf: \Hp \times \Hp \to \RR^+$ be a distance-like function as
  in \cref{def:DLF}. Then, for every $\nu_1, \nu_2 \in Pr(\Hp)$ and every
  $\dlf$-Lipschitz function $\Phi: \Hp \to \RR$,
  \begin{align}
    \Wass_{\dlf}(\nu_1, \nu_2) 
    \geq  \frac{1}{L_{\Phi}}
    \left| \int \Phi(\bq ) \nu_1(d\bq) - \int \Phi(\bq') \nu_2(d\bq') \right|,
    \label{eq:hyper:KWD:1}
  \end{align}
  where $L_{\Phi}$ is the Lipschitz constant associated with $\Phi$.
  In particular, for any Markov kernel $P$, 
  \begin{align}
    \left| P^n\Phi(\bq_0) - \int \Phi(\bq) \nu(d\bq) \right|
    \leq L_{\Phi} \Wass_\dlf(P^n(\bq_0, \cdot ), \nu),
    \label{eq:hyper:KWD:2}
  \end{align}
  valid for any measure $\nu \in Pr(\Hp)$, $\bq_0 \in \Hp$ and $\dlf$-Lipschitz function $\Phi$.
\end{Prop}
\begin{proof}
Fix $\nu_1, \nu_2 \in \Pr(\Hp)$ and let $\Gamma \in \Co(\nu_1, \nu_2)$. Note that
\begin{align}
    \left|  \int \Phi(\bq ) \nu_1(d\bq) - \int \Phi(\bq') \nu_2(d\bq') \right| 
    \leq \int \left| \Phi(\bq) - \Phi(\bq') \right| \Gamma(d\bq, d\bq') \leq L_\Phi \int l(\bq, \bq') \Gamma(d\bq, d\bq'). 
    \label{eq:hyper:KWD:proof}
\end{align}
Inequality \eqref{eq:hyper:KWD:1} then follows by taking the infimum in \eqref{eq:hyper:KWD:proof} over all $\Gamma \in \Co(\nu_1, \nu_2)$.
\end{proof}

We next present a first version of the strong law of large numbers
(SLLN) relevant for certain classes of mixing Markov processes.
Note that this first result does not require a spectral gap condition
but see \cref{prop:CLT:LLN} below where we additionally establish
criteria for a central limit theorem under the stronger assumption of a
spectral gap.
\begin{Prop}
  \label{prop:gen:LLN}
  Suppose that $P$ is a markov kernel with a unique invariant measure
  $\mu_*$. We denote the associated Markov process as
  $\{Q_k(\bq_0)\}_{k \geq 0, q_0 \in \Hp}$. Let $\dlf$ be a distance-like
  function and introduce the notation
  \begin{align}
    G(\bq_0) := \sum_{k =0}^\infty \Wass_\dlf(P^k(\bq_0, \cdot), \mu_*).
    \label{eq:summation:mix:cond:1}
  \end{align}
  Then, for any $\bq_0 \in \Hp$ such that
  \begin{align}
    G(\bq_0) + \sup_{n \geq 1} \E [G(Q_n(\bq_0))^2] < \infty
    \label{eq:Mix:Mom:Cond}
  \end{align}
  and such that, for some $\bar{\bq} \in \Hp$,
  \begin{align}
    \sup_{n \geq 1} \E [\dlf(Q_n(\bq_0), \bar{\bq})^2] < \infty,
    \label{eq:Mix:Mom:Cond:2}
  \end{align}
  we have that, for each $\Phi \in \mbox{Lip}_{\dlf}$,
  \begin{align}
    \lim_{n \to \infty} \left| \frac{1}{n}\sum\limits_{k =1}^n\Phi(Q^k(\bq_0))
    - \int \Phi(\bq') \mu_*(d\bq') \right| = 0,
    \label{eq:obs:LLN:conclusion}
 \end{align}
 almost surely.\footnote{Note that under \eqref{eq:Mix:Mom:Cond} every
   $\mbox{Lip}_{\dlf} \subset L^1(\mu_*)$ so that
   $\int \Phi(\bq') \mu_*(d\bq')$ is a well defined, finite quantity.}
\end{Prop}

\begin{Rmk}
  The scope of applicability of \cref{lem:Kantor:dual},
  \cref{prop:gen:LLN} reaches beyond \cref{prop:CLT:LLN} below
  which is more specialized to our setting.   See, for
  example, the sub-geometric rates of convergence in the Wasserstein
  distance given in \cite{butkovsky2014subgeometric,
    durmus2015quantitative, durmus2016subgeometric}.  
\end{Rmk}

\begin{proof}
Take $\{\mathcal{F}_n\}_{n \geq 1}$ to be the filtration associated with the Markov
process $\{Q_k(\bq_0)\}_{k \geq 0, q_0 \in \Hp}$.  Given any
$\Phi \in \mbox{Lip}_{\dlf}$, we define  
\begin{align}
  M^{\Phi}_n :=
  \sum_{k = 0}^\infty \left(\E( \BPhi(Q_k(\bq_0)  |\mathcal{F}_n)
  -\E(\BPhi(Q_k(\bq_0) )\right)
      \label{eq:MG:EA:decomposition:0}
\end{align}
where
\begin{align}
  \BPhi(\bq_0) := \Phi(\bq_0)  - \int \Phi(\btq) \mu_*(d\btq)
  \label{eq:mean:zero:Obs}
\end{align}  
Invoking the Markov property,
\begin{align}
  M^{\Phi}_n
  = \sum_{k = 0}^n \BPhi(Q_k(\bq_0))
    + \sum_{k =0}^\infty\left( P^{k+1}\BPhi(Q_n(\bq_0)) - P^k\BPhi(\bq_0) \right),
      \label{eq:MG:EA:decomposition:1}
\end{align}
so that, rearranging, we have
\begin{align}
  \frac{1}{n} \sum_{k = 0}^n \Phi(Q_k(\bq_0)) - \int \Phi(\btq) \mu_*(d\btq)
  &= \frac{1}{n} \sum_{k=0}^\infty
        \left(  P^k\BPhi(\bq_0) - P^{k+1}\BPhi(Q_n(\bq_0))\right)
    +\frac{ M^{\Phi}_n}{n} 
    \notag\\
  &:= T_1^{(n)} + T_2^{(n)}.
    \label{eq:MG:EA:decomposition}
\end{align}
Let us show that, for each of the terms $T_j^{(n)}$,
$\lim_{n \to \infty} T_j^{(n)}=0$ a.s. in order to infer
the desired conclusion.

Start with $T_1^{(n)}$.  Here note that, with \eqref{eq:hyper:KWD:2},
\begin{align}
  |T_1^{(n)}| \leq L_\Phi \frac{G(\bq_0) + G(Q_n(\bq_0))}{n},
  \label{eq:LLN:Bnd:T2:1}
\end{align}
where $L_\Phi$ is the Lipschitz constant associated with $\Phi$.  Form the sets
$A_n := \{|T_1^{(n)}| \geq n^{-1/4}\}$.  With \eqref{eq:LLN:Bnd:T2:1}
and the Markov inequality we find
\begin{align*}
  \sum_{n = 1}^\infty \Prb( A_n) 
  \leq  L_\Phi  \sum_{n = 1}^\infty \frac{\E\left(G(\bq_0)
  +  G(Q_n(\bq_0)\right)^2}{n^{3/2}}
  \leq 2 L_\Phi (G(\bq_0)^2 +\sup_{n \geq 1} \E G(Q_n(\bq_0))^2)
  \sum_{n = 1}^\infty \frac{1}{n^{3/2}}.
\end{align*}
Hence, invoking the Borel-Cantelli lemma and the condition
\eqref{eq:Mix:Mom:Cond}, we infer that
$\Prb(A_n \text{ infinitely often}) = 0$ which amounts to the desired
convergence for $T_1^{(n)}$.

Regarding the second term $T_2^{(n)}$, we claim that $\{M^{\Phi}_n\}_{n \geq 0}$ is a mean zero, square integrable martingale. From the definition of $\{M^{\Phi}_n\}_{n\in \mathbb{N}}$ in \eqref{eq:MG:EA:decomposition} it follows immediately that $M_0 = 0$. Now in view of \eqref{eq:MG:EA:decomposition:1}, notice that
for any $n \geq 1$ the increments $M_n^\Phi - M_{n-1}^\Phi$ have the form
\begin{align}
  M_n^\Phi - M_{n-1}^\Phi
     = \BPhi(Q_n(\bq_0)) 
  + \sum_{k = 0}^\infty
       \left(P^{k+1}\BPhi(Q_n(\bq_0))
          - P^{k+1}\BPhi(Q_{n-1}(\bq_0))\right).
  \label{eq:MP:MG:Diff}
\end{align}
Thus, for any $n \geq 1$, using that $\BPhi \in \mbox{Lip}_{\dlf}$ and
recalling the definition of $G$ we have
\begin{align}
  \E (M_n^\Phi - M_{n-1}^\Phi)^2 
  \leq  4 \BPhi(\bar{\bq})^2 + 4 L_\Phi^2 \left[ \E \dlf(\bar{\bq}, Q_n(\bq_0))^2 + \E G(Q_{n-1}(\bq_0))^2 +  \E G(Q_n(\bq_0))^2 \right]
  \label{eq:inc:est}
\end{align}
where $\bar{\bq} \in \Hp$ is selected as in \eqref{eq:Mix:Mom:Cond:2}. With \eqref{eq:Mix:Mom:Cond}, \eqref{eq:Mix:Mom:Cond:2} and noticing that
\[
    \E M_n^2 = \E \left( \sum_{k=1}^n (M_k - M_{k-1}) \right)^2 \leq c(n) \sum_{k=1}^n \E \left( M_k - M_{k-1} \right)^2,
\]
we conclude
that $\{M_n^\Phi\}_{n \in \mathbb{N}}$ is square integrable. To show that $\{M^\Phi_n\}_{n \in \mathbb{N}}$ is a martingale observe that for any $n \geq 0$, using standard properties of conditional expectations,
\begin{align}
  \E( M^\Phi_{n+1} | \mathcal{F}_n)
  = \sum_{k = 0}^\infty \left(\E (\E( \BPhi(Q_k(\bq_0)  |\mathcal{F}_{n+1})|\mathcal{F}_{n}) 
     -\E( \BPhi(Q_k(\bq_0) )\right)
  = M^{\Phi}_{n}.
    \label{eq:MG:ver}
\end{align}

With this in hand we recall a martingale convergence theorem from
\cite{chow1967strong} (see also \cite[Appendix A.12]{kuksin2012mathematics}) 
which can be stated as follows:  Let $\{M_n\}_{n \in \mathbb{N}}$ be a square
integrable, mean zero martingale. If
\begin{align}
  \sum_{k = 1}^\infty \frac{ \E (M_k - M_{k -1})^2}{k^2} < \infty
  \label{eq:MG:LLN:cond}
\end{align}
then
\begin{align*}
  \lim_{n \to \infty} \frac{M_n}{n} = 0 \quad \text{ almost surely.}
\end{align*}
In view of the bound \eqref{eq:inc:est} and again invoking the
standing conditions \eqref{eq:Mix:Mom:Cond}, \eqref{eq:Mix:Mom:Cond:2}
we find that the condition \eqref{eq:MG:LLN:cond} is satisfied for
$\{M_n^\Phi\}_{n \in \mathbb{N}}$ and hence we infer that
$\lim_{n \to \infty} T_2^{(n)} =0$ almost surely. The proof is now complete.
\end{proof}

In order to obtain rates of convergence for \eqref{eq:obs:LLN:conv:c1}
we can furthermore establish a central limit theorem (CLT) result by now directly imposing a `spectral
gap' condition.  For this stronger convergence result we again
rely on the decomposition \eqref{eq:MG:EA:decomposition:0},
\eqref{eq:MG:EA:decomposition} now in conjunction with a Martingale
central limit result from \cite{komorowski2012central} which we recall
as \cref{thm:MG:CLT} below.
\begin{Prop}
  \label{prop:CLT:LLN}
  Let $P$ be a Markov kernel on a complete metric space $(\Hp,
  \rho)$. Take $\{Q_n(\bq_0)\}_{n \geq 0, q_0 \in \Hp}$ to be the
  associated Markov process. Let $V: \Hp \to \RR^+$ be a function satisfying the following Lyapunov type assumption:
  \begin{align}
    \E [V(Q_n(\bq_0))^{2}]
    \leq \kappa^n V(\bq_0)^{2} + K
  \label{eq:lyapunov:CLT}
  \end{align}
  for some constants $\kappa \in (0,1)$, $K > 0$ independent of $n \geq 0$. Consider the distance-like functions
  \begin{align}
    \dlf_p(\bq, \btq) 
    = \sqrt{[ 1 \wedge \rho(\bq, \btq) ]
           (1 + V(\bq)^p + V(\btq)^p)}
    \label{eq:dlf:HM:gen:p}
  \end{align}
  for $p \geq 1$. We assume that for $p = 1, 2$ the contraction condition
  \begin{align}
    \Wass_{\dlf_p}( \nu_1 P^n, \nu_2 P^n)
    \leq c_1 e^{- c_2 n} \Wass_{\dlf_p}(\nu_1, \nu_2) \quad
    \text{ for any } \nu_1, \nu_2 \in Pr(\Hp),
    \label{eq:spec:gap:CLT}
  \end{align}
  is maintained, where $c_1, c_2$ are constants independent of $n$ but
  which may depend on $p$. 

  For $\Phi \in \mbox{Lip}_{\dlf_1}$, let
  \begin{align*}
    X_n(\Phi)  := \frac{1}{n} \sum_{k =1}^n \Phi( Q_k(\bq_0))
    - \int \Phi(\bq') \mu_*(d\bq'),
  \end{align*}
  where $\mu_*$ is the unique invariant measure for $P$;
  cf. \cref{rmk:covered:options}.  Then, under these circumstances,
  for any $\Phi \in \mbox{Lip}_{\dlf_1}$,
  \begin{align}
    X_n(\Phi) \to 0 \quad \mbox{ as } n \to \infty
    \label{eq:LLN:CLT:restate}
  \end{align}
  almost surely and moreover
  \begin{align}
  \sqrt{n} X_n(\Phi) \Rightarrow N(0, \sigma^2(\Phi)) \quad \mbox{ as } n \to \infty,
  \label{eq:CLT:obs}
  \end{align}
  i.e. $\sqrt{n} X_n(\Phi)$ converges weakly to a real-valued gaussian random variable with mean zero and covariance $\sigma^2(\Phi)$, where $\sigma^2(\Phi)$ is specified explicitly as \eqref{eq:CLT:sig:def} below.
\end{Prop}
\begin{Rmk}
  \label{rmk:covered:options}
  The condition \eqref{eq:spec:gap:CLT} ensures the existence and
  uniqueness of the invariant measure $\mu_*$ as observed in
  \cite{hairer2011asymptotic}.  Moreover, \eqref{eq:lyapunov:CLT}
  implies the following moment bound for $\mu_*$
  \begin{align}
    \int V(\bq')^2 \mu_*(d \bq') \leq K < \infty.
    \label{eq:finite:IM:CLT}
  \end{align}
  As such, using that $\Phi \in \mbox{Lip}_{\dlf_1}$ and
  \eqref{eq:dlf:HM:gen:p}, we have
  \begin{align*}
    \int |\Phi(\bq')| \mu_*(d\bq')
    \leq |\Phi(\bar{\bq})| +  L_\Phi \left(1 + \sqrt{V(\bar{\bq})}
              +\int \sqrt{V(\bq')} \mu_*(d\bq')\right)
  \end{align*}
  for any
  $\bar{\bq} \in \Hp$ so that with \eqref{eq:finite:IM:CLT} we are
  guaranteed that $\int |\Phi(\bq')| \mu_*(d\bq') < \infty$.
\end{Rmk}

Our proof relies on the following abstract result from \cite[Theorem
5.1]{komorowski2012central} which we reformulate here for clarity and the
convenience of the reader.
\begin{Thm}
  \label{thm:MG:CLT}
  Let $\{M_n\}_{n \geq 0}$ be a square integrable, mean zero martingale,
  relative to a filtration $\{\mathcal{F}_n\}_{n \geq 0}$.  Assume
  that:
  \begin{itemize}
  \item[(i)] we have the uniform bound
    \begin{align}
      \sup_{n \geq 0} \E (M_{n+1} - M_{n})^2 < \infty.
      \label{eq:CLT:MG:C1}
    \end{align}
  \item[(ii)] For every $\epsilon > 0$
    \begin{align}
      \lim_{n \to \infty} \frac{1}{n} \sum_{m = 0}^{n-1}
      \E [(M_{m+1} - M_{m})^2  \indFn{|M_{m+1} - M_{m}|  \geq \epsilon \sqrt{n}}] =0.
      \label{eq:CLT:MG:C2}
    \end{align}
  \item[(iii)] For every $\epsilon > 0$,
    \begin{align}
      \lim_{k \to \infty} \limsup_{n \to \infty} 
      \frac{1}{nk} \sum_{m = 1}^n 
         \sum_{j = (m -1)k}^{mk -1} 
      \E \left[ (1 + (M_{j+1} - M_{j})^2)  
         \indFn{|M_j - M_{(m-1)k}|  \geq \epsilon \sqrt{nk}} \right] =0.
      \label{eq:CLT:MG:C3}
    \end{align}
  \item[(iv)] There exists a constant $\sigma^2 \geq 0$ such that 
    \begin{align}
      \lim_{k \to \infty} \limsup_{n \to \infty} 
      \frac{1}{n} \sum_{m = 1}^n 
      \E \left|
      \frac{1}{k} \sum_{j = (m-1)k }^{mk-1} \E((M_{j+1} - M_{j})^2 | \mathcal{F}_{(m-1)k})
       - \sigma^2
      \right| = 0.
      \label{eq:CLT:MG:C4}
    \end{align}
  \end{itemize}
  Then, under these four conditions, 
  \begin{align*}
    \frac{M_n}{\sqrt{n}} \Rightarrow N(0, \sigma^2) \quad \mbox{ as } n \to \infty,
  \end{align*}
  (that is in distribution) where $\sigma^2$ is the constant appearing
  in \eqref{eq:CLT:MG:C4}.
\end{Thm}

With this result in hand we turn to the proof of \cref{prop:CLT:LLN}.
\begin{proof}[Proof of \cref{prop:CLT:LLN}]
  To prove \eqref{eq:LLN:CLT:restate} we simply show that
  \eqref{eq:spec:gap:CLT}, \eqref{eq:lyapunov:CLT} imply
  \eqref{eq:Mix:Mom:Cond}, \eqref{eq:Mix:Mom:Cond:2}, with
  $\dlf = \dlf_1$, so that we can directly apply \cref{prop:gen:LLN}.
  Observe that, for any $\bar{\bq} \in \Hp$ we have 
  \begin{align*}
    \sum_{k =0}^\infty \Wass_{\dlf_1}(P^k(\bar{\bq},\cdot),\mu_*)
      \leq \Wass_{\dlf_1}(\delta_{\bar{\bq}},\mu_*) \sum_{k =0}^\infty c_1 e^{-c_2 k}
    \leq c \left(1 + \sqrt{V(\bar{\bq})}
                + \int \! \sqrt{V(\bq')}\mu_*(d\bq')\right).
  \end{align*}
  Noting that, with \eqref{eq:finite:IM:CLT}, we have
  $\int \sqrt{V(\bq')}\mu_*(d\bq') < \infty$ and with
  \eqref{eq:lyapunov:CLT} we infer
  $\sup_{k \geq 0} \E V(Q_k(\bq_0)) < \infty$ so that
  \eqref{eq:Mix:Mom:Cond} holds.  Regarding \eqref{eq:Mix:Mom:Cond:2}
  we have, for any $\bq_0, \bar{\bq} \in \Hp$
  \begin{align*}
    \sup_{n \geq 1} \E \dlf_{1}(Q_n(\bq_0), \bar{\bq})
    \leq c \left(1 + \sup_{n \geq 1} \E \sqrt{V(Q_n(\bq_0))} + \sqrt{V(\bar{\bq})} \right)
    \leq c \left(1 + \sqrt{V(\bq_0)} + \sqrt{V(\bar{\bq})} \right),
  \end{align*}
  where the last inequality again follows from \eqref{eq:lyapunov:CLT}.

  Let us next turn to establish the convergence to normality,
  \eqref{eq:CLT:obs}. Fix $\Phi \in \mbox{Lip}_{\dlf_1}$. Here, working from the identity \eqref{eq:MG:EA:decomposition}, we have
  \begin{align}
  \sqrt{n}  X_n(\Phi)
  &= \frac{1}{\sqrt{n}} \sum_{k=0}^\infty
        \left(  P^k\BPhi(\bq_0) - P^{k+1}\BPhi(Q_n(\bq_0))\right)
    +\frac{ M^{\Phi}_n}{\sqrt{n}} 
    := \bar{T}_1^{(n)}+ \bar{T}_2^{(n)},
    \label{eq:MG:CLT:decomposition}
  \end{align}
  where $M^{\Phi}_n$ is the martingale defined as in
  \eqref{eq:MG:EA:decomposition:1}.  We would like to show that
  $\lim_{n \to \infty} \bar{T}_1^{(n)} = 0$ in probability and that
  $\bar{T}_2^{(n)}$ converges in distribution to a normal random
  variable in order to conclude \eqref{eq:CLT:obs} from the
  `converging together lemma'; cf. \cite{durrett2019probability}.

Regarding the first term $\bar{T}_1^{(n)}$, with \eqref{eq:hyper:KWD:2}
and \eqref{eq:spec:gap:CLT}, it follows 
\begin{align*}
  |\bar{T}_1^{(n)}|
  &\leq  \frac{L_\Phi}{\sqrt{n}}
  \sum_{k=0}^\infty (\Wass_{\dlf_1} (P^k(\bq_0, \cdot), \mu^*)
    +\Wass_{\dlf_1} (P^{k+1}(Q_n(\bq_0), \cdot), \mu^*))\\
  &\leq \frac{c}{\sqrt{n}}
     (\Wass_{\dlf_1} (\delta_{\bq_0},\mu^*) + \Wass_{\dlf_1} (\delta_{Q_n(\bq_0)}, \mu^*))
  \leq  \frac{c\left( 1 + \sqrt{V(\bq_0)} + \sqrt{V(Q_n(\bq_0))}\right)}{\sqrt{n}}
\end{align*}
where we used that $\dlf_1$ has the form \eqref{eq:dlf:HM:gen:p} for the final bound. With this estimate and our assumption \eqref{eq:lyapunov:CLT} we find that 
$\lim_{n \to \infty} \E | T_1^{(n)}| = 0$ so that $T_1^{(n)}$ decays to zero in 
probability as desired.

We address the second term $\bar{T}^{(2)}_n$ by verifying the
conditions of \cref{thm:MG:CLT}. As in \eqref{eq:inc:est},
\eqref{eq:MG:ver}, it is clear that $\{M^\Phi_n\}_{n \geq 0}$ is a mean
zero square integrable martingale.  We therefore proceed to establish
each of the bounds \eqref{eq:CLT:MG:C1}--\eqref{eq:CLT:MG:C4} for
$\{M^\Phi_n\}_{n \geq 0}$ in turn.

Start with \eqref{eq:CLT:MG:C1}. Working from the identity
\eqref{eq:MP:MG:Diff}, we observe that, for any $m \geq 0$,
\begin{align*}
  (M^\Phi_{m+1} - M^\Phi_{m})^4 
  &\leq c \BPhi(Q_{m+1}(\bq_0))^4 
  + c\left(\sum_{k = 0}^\infty P^{k+1}\BPhi(Q_{m+1}(\bq_0)) 
    - P^{k+1}\BPhi(Q_{m}(\bq_0))\right)^4
  \notag\\
  &\leq c (\dlf_1 (Q_{m+1}(\bq_0), 0)^4 + V(Q_{m+1}(\bq_0))^2 + V(Q_{m}(\bq_0))^2 +1)
    \notag\\
  &\leq c (V(Q_{m+1}(\bq_0))^2 + V(Q_{m}(\bq_0))^2 +1)
\end{align*}
where we have used \eqref{eq:hyper:KWD:2} and \eqref{eq:spec:gap:CLT}.
Therefore, invoking \eqref{eq:lyapunov:CLT}, we have now shown
\begin{align}
  \sup_{m \geq 0} \E (M^\Phi_{m+1} - M^\Phi_{m})^4 < \infty
     \label{eq:towards:4th:mom}
\end{align}
so that, in particular, \eqref{eq:CLT:MG:C1} holds.  Furthermore since, for any
$\epsilon > 0$ and any $0 \leq m \leq n$
\begin{align*}
  \E [(M^\Phi_{m+1} - M^\Phi_{m})^2
       \indFn{| M^\Phi_{m+1} - M^\Phi_{m}| \geq \epsilon \sqrt{n}}]
  &\leq \left(\E (M^\Phi_{m+1} - M^\Phi_{m})^4\right)^{1/2} 
  \Prb(| M^\Phi_{m+1} - M^\Phi_{m}| \geq \epsilon \sqrt{n})^{1/2}\\
  &\leq \frac{1}{\epsilon^2 n}\E (M^\Phi_{m+1} - M^\Phi_{m})^4
\end{align*}
we infer \eqref{eq:CLT:MG:C2}.

Regarding \eqref{eq:CLT:MG:C3} we proceed in a similar fashion. For $(m-1)k \leq j \leq mk-1$ and any $m, n, k \geq 1$ we have
\begin{align}\label{est:item:iii}
\E [(1+&(M^\Phi_{j+1} - M^\Phi_{j})^2 )\indFn{| M^\Phi_{j} - M^\Phi_{(m-1)k}| \geq \epsilon \sqrt{nk}}] \notag\\
  \leq& \frac{c}{\epsilon^{1/2} (nk)^{1/4}}
  \left(\E (1+(M^\Phi_{j+1} - M^\Phi_{j})^4) \right)^{1/2}
  \left(\E|M^\Phi_{j} - M^\Phi_{(m-1)k}|\right)^{1/2} 
\end{align}
We estimate the last term between parentheses in \eqref{est:item:iii} as
\begin{align}\label{est:item:iii:b}
    \E|M^\Phi_{j} - M^\Phi_{(m-1)k}| 
    \leq \sum_{l= (m-1)k}^{j-1}\E|M^\Phi_{l+1} - M^\Phi_l|
    \leq c ( j - (m-1)k ) 
    \leq c k,
\end{align}
where in the second inequality we used \eqref{eq:CLT:MG:C1}. Combining \eqref{est:item:iii} and \eqref{est:item:iii:b} now yields \eqref{eq:CLT:MG:C3}, where we notice carefully that having the $\limsup$ as $n \to \infty$ applied first is crucial.

Let us turn to the final bound \eqref{eq:CLT:MG:C4}.  Take 
\begin{align}
  \Psi(\bq, \btq) 
  := \left[\BPhi(\bq) 
    + \sum_{k = 0}^\infty (P^{k+1} \BPhi(\bq) - P^{k+1} \BPhi(\btq))   
  \right]^2
  \label{eq:b:psi:def}
\end{align}
Now for any $j \geq (m-1)k$ and with $m, k \geq 1$
we have
\begin{align*}
  \E ( (M^\Phi_{j+1} - M^\Phi_{j})^2| \mathcal{F}_{(m -1)k})
  =& \E \Psi(Q_{(j+1 - (m-1)k) + (m-1)k}(\bq_0), Q_{(j- (m-1)k) + (m-1)k}(\bq_0)) | \mathcal{F}_{(m -1)k})\\
  =& H_{j - (m-1)k}( Q_{(m-1)k}(\bq_0))
\end{align*}
where we have used the Markov property at the last step.   Here
for any $l \geq 0$
\begin{align*}
  H_{l}(\bq_0) := \E \Psi(Q_{l+1}(\bq_0), Q_{l}(\bq_0))
            = P^l \Gamma(\bq_0)
\end{align*}
with 
\begin{align}
  \Gamma(\bq_0) = \E \Psi(Q_{1}(\bq_0), \bq_0).
  \label{eq:b:gamma:def}
\end{align}
Working from these identities we find, again for any $j \geq (m-1)k$
and with $m, k \geq 1$
\begin{align}
  \frac{1}{k} \sum_{j = (m-1)k }^{mk-1} 
  \E((M_{j+1}^{\Phi} - M_{j}^{\Phi})^2 | \mathcal{F}_{(m-1)k})
  =&   \frac{1}{k} \sum_{j = (m-1)k }^{mk-1}  H_{j - (m-1)k}( Q_{(m-1)k}(\bq_0))
  =   \frac{1}{k} \sum_{j = 0}^{k-1}  H_{j}( Q_{(m-1)k}(\bq_0)) \notag\\
  =&   \frac{1}{k} \sum_{j = 0}^{k-1}  P^j\Gamma( Q_{(m-1)k}(\bq_0)).
     \notag
\end{align}
As such,
\begin{align}
  \frac{1}{n}\sum_{m = 1}^n\E &\left|\frac{1}{k} \sum_{j = (m-1)k }^{mk-1} 
  \E((M_{j+1}^{\Phi} - M_{j}^{\Phi})^2 | \mathcal{F}_{(m-1)k}) - \sigma^2\right|
        \leq \frac{1}{n} \sum_{m = 1}^n 
          P^{(m-1)k}\left(\frac{1}{k}\sum_{j = 0}^{k-1}
                                |P^l\Gamma(\bq_0) - \sigma^2|\right),
        \label{eq:quad:diff:bnd:1}
\end{align}
which is valid for any $0 \leq \sigma^2 < \infty$.

With the aim of once again combining \eqref{eq:hyper:KWD:2}
with \eqref{eq:spec:gap:CLT} we now take 
\begin{align}
  \sigma^2 = \sigma^2(\Phi) := \int \Gamma(\bq) \mu_*(d \bq).
  \label{eq:CLT:sig:def}
\end{align}
with $\Gamma$ as in \eqref{eq:b:gamma:def}.  We will show presently that whenever $\Phi$ is $\dlf_1$-Lipshitz then 
$\Gamma$ is $\dlf_2$-Lipshitz, namely
\eqref{eq:Gamma:2:lip} below. This being so, as in
\eqref{eq:finite:IM:CLT}, it is clear that $\sigma^2(\Phi) < \infty$
for any $\dlf_1$-Lipshitz $\Phi$.  Moreover, invoking once again
\eqref{eq:hyper:KWD:2} and \eqref{eq:spec:gap:CLT} we obtain that
\begin{align}
  \frac{1}{k} \sum_{j = 0}^{k-1}  |P^j\Gamma( \bq_0) - \sigma^2(\Phi)|
  \leq \frac{L_{\Gamma}}{k}
  \sum_{j = 0}^{k -1} \Wass_{\dlf_2}( P^j(\bq_0, \cdot), \mu^*)
  \leq \frac{c( 1 + \sqrt{V(\bq_0)})}{k}.
        \label{eq:quad:diff:bnd:2}
\end{align}
Combining \eqref{eq:quad:diff:bnd:1}, \eqref{eq:quad:diff:bnd:2} with
\eqref{eq:lyapunov:CLT} we find
\begin{align*}
  \frac{1}{n}\sum_{m = 1}^n\E
  &\left|\frac{1}{k} \sum_{j = (m-1)k }^{mk-1} 
  \E((M_{j+1}^{\Phi} - M_{j}^{\Phi})^2 | \mathcal{F}_{(m-1)k}) -\sigma^2(\Phi)
    \right|\\
      \leq& \frac{c}{nk}\sum_{m = 1}^n P^{(m-1)k} (1 + \sqrt{V(\bq_0)})
      \leq \frac{c}{nk}\sum_{m = 1}^n (1 + \alpha^{(m-1)k}\sqrt{V(\bq_0)})
      \leq \frac{c(1+\sqrt{V(\bq_0)})}{k}                                
\end{align*}
which yields the final item \eqref{eq:CLT:MG:C4}.

We therefore conclude the proof by showing that
$\Gamma \in \mbox{Lip}_{\dlf_2}$ whenever
$\Phi \in \mbox{Lip}_{\dlf_1}$.  Observe that
from \eqref{eq:b:gamma:def} we have
\begin{align}
  \Gamma(\bq) - \Gamma(\btq)
  = 
  \E\left[ 
  (\sqrt{\Psi(Q_{1}(\bq), \bq)}- \sqrt{\Psi(Q_{1}(\btq), \btq)})
  (\sqrt{\Psi(Q_{1}(\bq), \bq)} + \sqrt{\Psi(Q_{1}(\btq), \btq)})
  \right]
  \label{eq:crz:bnd:CLT:1}
\end{align}
From \eqref{eq:b:psi:def} and invoking \eqref{eq:hyper:KWD:2},
\eqref{eq:spec:gap:CLT} we have that
\begin{align}
  |&\sqrt{\Psi(Q_{1}(\bq), \bq)}- \sqrt{\Psi(Q_{1}(\btq), \btq)})|
  \notag\\
  &\quad \leq 
  |\BPhi(Q_1(\bq)) - \BPhi(Q_1(\btq))|
   + |\sum_{k = 0}^\infty (P^{k+1} \BPhi(Q_1(\bq)) - P^{k+1} \BPhi(Q_1(\btq)))|
   + |\sum_{k = 0}^\infty (P^{k+1} \BPhi(\bq) - P^{k+1} \BPhi(\btq))|
  \notag\\
   &\quad \leq c(\dlf_1(Q_1(\bq), Q_1(\btq)) + \dlf_1(\bq, \btq)).
  \label{eq:crz:bnd:CLT:2}
\end{align}
On the other hand, again with \eqref{eq:hyper:KWD:2},
\eqref{eq:mean:zero:Obs} and \eqref{eq:spec:gap:CLT} we also obtain
the bound
\begin{multline} \label{eq:crz:bnd:CLT:3}
  |\sqrt{\Psi(Q_{1}(\bq), \bq)} + \sqrt{\Psi(Q_{1}(\btq), \btq)})| \\
 \leq c \left( |\bar{\Phi}(Q_1(\bq))| + |\bar{\Phi}(Q_1(\btq))| + \Wass_{\dlf_1}(\delta_{Q_{1}(\bq)}, \mu_*) + \Wass_{\dlf_1}(\delta_{Q_{1}(\btq)}, \mu_*)
  +\Wass_{\dlf_1}(\delta_{\bq}, \mu_*) + \Wass_{\dlf_1}(\delta_{\btq}, \mu_*) \right)
  \\
 \leq c \left(1 + \sqrt{V(Q_1(\bq))}+\sqrt{V(Q_1(\btq))} + \sqrt{V(\bq)} + \sqrt{V(\btq)}  \right)
  \\
 \leq c \left(\sqrt{1 + V(Q_1(\bq)) + V(Q_1(\btq))} + \sqrt{1 + V(\bq) + V(\btq)} \right)
\end{multline}
Now observe that, for any $\bq, \btq \in \Hp$
\begin{align}
  \dlf_1(\bq,\btq)\sqrt{1 + V(\bq) + V(\btq)}
  \leq  2  \dlf_2(\bq,\btq),
  \label{eq:dlf:relation:1:2}
\end{align}
so that combining this simple observation with \eqref{eq:crz:bnd:CLT:1}--\eqref{eq:crz:bnd:CLT:3} we find
\begin{align}
  |\Gamma(\bq) - \Gamma(\btq)|
  \leq& c 
     \E[ (\dlf_1(Q_1(\bq), Q_1(\btq)) + \dlf_1(\bq, \btq))
      (\sqrt{1 + V(Q_1(\bq)) + V(Q_1(\btq))}
        + \sqrt{1 + V(\bq) + V(\btq)})]
  \notag\\
  \leq& c \E \dlf_2(Q_1(\bq), Q_1(\btq))
        + c\dlf_1(\bq, \btq)
        \E \left(\sqrt{1 + V(Q_1(\bq)) + V(Q_1(\btq))}\right)
  \notag\\
       &+ c\sqrt{1 + V(\bq) + V(\btq)} \E \dlf_1(Q_1(\bq), Q_1(\btq))
         + c\dlf_2(\bq, \btq).
           \label{eq:crz:bnd:CLT:4}
\end{align}
Now notice that, under \eqref{eq:lyapunov:CLT} we have
\begin{align*}
  \E \left(\sqrt{1 + V(Q_1(\bq)) + V(Q_1(\btq))}\right)
  \leq c\sqrt{1 + V(\bq) + V(\btq)}.
\end{align*}
On the other hand, notice that we may take $Q_1(\bq)$ and $Q_1(\btq)$ to be
any coupling of $P(\bq, \cdot)$ and $P(\btq, \cdot)$ in
\eqref{eq:crz:bnd:CLT:4}. As such, with \eqref{eq:crz:bnd:CLT:4}
and these two observations
\begin{multline*}
  |\Gamma(\bq) - \Gamma(\btq)|
  \leq \Wass_{\dlf_2}(P(\bq, \cdot),P(\btq, \cdot))
  + c\dlf_1(\bq, \btq)\sqrt{1 + V(\bq) + V(\btq)} \\
  + c\Wass_{\dlf_1}(P(\bq, \cdot),P(\btq, \cdot))
      \sqrt{1 + V(\bq) + V(\btq)} + c \dlf_2(\bq, \btq),
\end{multline*}
so that with \eqref{eq:dlf:relation:1:2} and a final invocation
of \eqref{eq:spec:gap:CLT}, we have
\begin{align}
  |\Gamma(\bq) - \Gamma(\btq)| \leq c \dlf_2(\bq, \btq).
  \label{eq:Gamma:2:lip}
\end{align}
The proof is now complete.
\end{proof}

We conclude this section with the following proposition which
gives a sufficient condition for a function to be $\dlf$-Lipschitz for
a class of distance-like functions including those appearing in the
main results of this work.
\begin{Prop}
  \label{prop:Lip:2:obs}
  Let $(\Hp, \| \cdot \|)$ be a Banach space and consider
  distance-like functions of the form
  \begin{align}
    \dlf(\bq, \btq) 
    = \sqrt{\left( \frac{\| \bq - \btq\|}{\varepsilon} \wedge 1\right)
           (1 + V(\bq) + V(\btq))}
    \label{eq:dlf:HM:gen}
  \end{align}
  where we suppose that $\varepsilon > 0$ and $V: \Hp \to [0,\infty)$
  is convex.  Given any continuously differentiable function
  $\Phi: \Hp \to \RR$ define
  \begin{align}
   L_{\Phi} := 
      \sup_{\bq \in \Hp} 
    \frac{\max\{ 
         2   |\Phi(\bq)|,
         \sqrt{\varepsilon} \|D \Phi(\bq)\| \}}{\sqrt{1+V(\bq)}}.
    \label{eq:Lip:const}
  \end{align}
  If $L_{\Phi} < \infty$ then $\Phi$ is $\dlf$-Lipschitz with $L_{\Phi}$
  providing a suitable Lipschitz constant.
\end{Prop}

\begin{proof}
  Fix any $\bq, \btq \in \Hp$.  We consider separately the cases when
  $\|\bq - \btq\| > \varepsilon$ and when $\| \bq- \btq\| \leq \varepsilon$.
  In the first situation when $\|\bq - \btq\| > \varepsilon$ we estimate
  \begin{align*}
    |\Phi(\bq) - \Phi(\btq)| 
    \leq \sqrt{1+V(\bq) + V(\btq)} 
        \left(\frac{|\Phi(\bq)|}{\sqrt{1+V(\bq)}} 
           +  \frac{|\Phi(\btq)|}{\sqrt{1+V(\btq)}}\right)  
    \leq L_{\Phi} \dlf(\bq, \btq).
  \end{align*}
  Now consider the case when $\| \bq- \btq\| \leq \varepsilon$.  Let
  $\bq_s = \bq+s (\btq - \bq)$, for $s \in [0,1]$ and observe that
  \begin{align*}
    |\Phi(\bq) - \Phi(\btq)|
    &\leq \|\bq - \btq\| \int_0^1 \| D \Phi( \bq_s ) \| ds
      \notag\\
    &\leq \int_0^1
      \sqrt{\left(\frac{\| \bq - \btq\|}{\varepsilon} \right) ( 1+ V(\bq_s))} 
        \cdot \frac{\sqrt{\varepsilon} \| D \Phi( \bq_s ) \|}{\sqrt{1+ V(\bq_s)}}  ds 
    \notag \\
    &\leq L_{\Phi} \int_0^1
      \sqrt{\left(\frac{\| \bq - \btq\|}{\varepsilon} \right)
                   (1+sV(\bq) + (1-s) V(\btq))}  ds 
      \leq L_{\Phi} \dlf(\bq, \btq)
\end{align*}
where we have used the convexity of $V$ for the penultimate bound.
The proof is complete.
\end{proof}

\section*{Acknowledgements}

This work would not have been able to be completed in its current
form without the extensive feedback and input of Nawaf Bou-Rabee
during two visits to Tulane university in the spring and summer of 2019.  We
would also like to acknowledge D. Albritton, G. Didier, J. Foldes,
E. T. Holtzchetti, J. Krometis, S. McKinley, J. Mattingly, S. Punshon-Smith,
G. Richards, V. Sverak for further fruitful discussions surrounding on
this work.  Our efforts here were partially supported under the grants
DMS-1313272 (NEGH), DMS-1816551 (NEGH), and the Simons Foundation
travel grant under 515990 (NEGH).

\addcontentsline{toc}{section}{References}
\begin{footnotesize}
\bibliographystyle{alpha}
\bibliography{bibphmc}
\end{footnotesize}

\vspace{1in}
\begin{multicols}{2}
\noindent
Nathan E. Glatt-Holtz\\ {\footnotesize
Department of Mathematics\\
Tulane University\\
Web: \url{http://www.math.tulane.edu/~negh/}\\
Email: \url{negh@tulane.edu}} 

\columnbreak

\noindent Cecilia F. Mondaini \\
{\footnotesize
  Department of Mathematics\\
  Drexel University\\
  Web: \url{https://drexel.edu/coas/faculty-research/faculty-directory/mondaini-cecilia/}\\
  Email: \url{cf823@drexel.edu}}\\[.2cm]
 \end{multicols}

\end{document}